\documentclass[12pt]{amsart}

\usepackage{amscd,amsfonts,amssymb}

\copyrightinfo{2000}{American Mathematical Society}

\newcommand{\cA}{{\mathcal A}}

\newcommand{\cD}{{\mathcal D}}
\newcommand{\cE}{{\mathcal E}}

\newcommand{\cH}{{\mathcal H}}

\newcommand{\cS}{{\mathcal S}}

\newcommand{\bA}{{\mathbb A}}
\newcommand{\bC}{{\mathbb C}}
\newcommand{\bF}{{\mathbb F}}
\newcommand{\bI}{{\mathbb I}}
\newcommand{\bJ}{{\mathbb J}}
\newcommand{\bN}{{\mathbb N}}

\newcommand{\bQ}{{\mathbb Q}}
\newcommand{\bR}{{\mathbb R}}
\newcommand{\bZ}{{\mathbb Z}}

\newcommand{\supp}{\mathop{\rm supp}\nolimits}

\numberwithin{equation}{section}

\newtheorem{Theorem}{Theorem}[section]
\newtheorem{Lemma}{Lemma}[section]
\newtheorem{Corollary}{Corollary}[section]
\newtheorem{Definition}{Definition}[section]
\newtheorem{Remark}{Remark}[section]
\newtheorem{Example}{Example}[section]
\newtheorem{Proposition}{Proposition}[section]

\author{A.~Yu.~Khrennikov}
\address{International Center for Mathematical Modelling in Physics
and Cognitive Sciences MSI, V\"axj\"o University, \ SE-351 95,
V\"axj\"o, \ Sweden.}
\email{andrei.khrennikov@msi.vxu.se}

\author{A.~V.~Kosyak}
\address{Institute of Mathematics, Ukrainian National Academy of Sciences, 3
Tereshchenkivs'ka, Kyiv, 01601, Ukraine.}
\email{kosyak02@gmail.com}

\author{V.~M.~Shelkovich}
\address{Department of Mathematics, St.-Petersburg State Architecture
and Civil Engineering University, \ 2 Krasnoarmeiskaya 4, 190005,
St. Petersburg, \ Russia.}
\email{shelkv@yahoo.com}

\title{Wavelet analysis on adeles and pseudo-differential operators}

\thanks{The first and third authors (A.~K. and V.~S.) were supported by the
grant of the Profile Mathematical Modeling and System Collaboration of V\"axj\"o
University (Sweden).
The third author (V.~S.) was  supported in part by Grant 09-01-00162 of Russian
Foundation for Basic Research.
The second (A.~K.) and third (V.~S.) were supported by the DFG Projects.}



\date{ }

\begin{document}

\begin{abstract}
This paper is devoted to wavelet analysis on adele ring $\bA$ and the theory 
of pseudo-differential operators.
We develop the technique which gives the possibility to generalize
finite-dimensional results of wavelet analysis to the case of adeles $\bA$ by using
infinite tensor products of Hilbert spaces. The adele ring is roughly
speaking a subring of the direct product of all possible ($p$-adic
and Archimedean) completions $\bQ_p$ of the field of rational
numbers $\bQ$ with some conditions at infinity. Using our technique, 
we prove that $L^2(\bA)=\otimes_{e,p\in\{\infty,2,3,5,\dots\}}L^2({\bQ}_{p})$ 
is the infinite tensor product of the spaces $L^2({\bQ}_{p})$ with a
stabilization $e=(e_p)_p$, where $e_p(x)=\Omega(|x|_p)\in L^2({\bQ}_{p})$, 
and $\Omega$ is a characteristic function of the
unit interval $[0,\,1]$, $\bQ_p$ is the field of $p$-adic numbers,
$p=2,3,5,\dots$; $\bQ_\infty=\bR$.
This description allows us to construct an infinite
family of Haar wavelet bases on $L^2(\bA)$ which can be obtained by
shifts and multi-delations. The adelic multiresolution analysis
(MRA) in $L^2(\bA)$ is also constructed. In the framework of this
MRA another infinite family of Haar wavelet bases is constructed.
We introduce the adelic Lizorkin spaces of test functions and
distributions and give the characterization of these spaces in terms
of wavelet functions. One class of pseudo-differential operators
(including the fractional operator) is studied on the Lizorkin spaces. 
A criterion for an adelic wavelet function to be an eigenfunction for
a pseudo-differential operator is derived.
We prove that any wavelet function is an eigenfunction of the
fractional operator. These results allow one to create
the necessary prerequisites for intensive using of adelic wavelet
bases and pseudo-differential operators in applications.
\end{abstract}

\maketitle

\section{Introduction}
\label{s1}

\subsection{$p$-Adic and adelic analysis.}\label{s1.1}
During a few hundred years theoretical physics has been developed
on the basis of real and, later, complex numbers.
The $p$-adic numbers were described by K.~Hensel in 1897 to transfer
the ideas and techniques of power series methods to number theory.
According to the well-known Ostrovsky theorem, {\it any nontrivial
valuation on the field $\bQ$ of rational numbers is equivalent either to the real
valuation $|\cdot|$ or to one of the $p$-adic valuations $|\cdot|_p$}.
The corresponding completions of $\bQ$ give the fields $\bR$ or $\bQ_p$.
The theory of $p$-adic numbers has already penetrated intensively
into several areas of mathematics and its applications.
In the last 20 years the field of $p$-adic numbers $\bQ_p$ (as
well as its algebraic extensions, including the field of complex
$p$-adic numbers $\bC_p$) has been intensively used in theoretical
and mathematical physics, $p$-adic string theory, gravity and cosmology,
the theory of stochastic differential equations over
the field of $p$-adic numbers, Feynman path integration over $p$-adics,
the theory of $p$-adic valued probabilities and dynamical systems,
in theory of disordered systems (spin glasses)
(see~\cite{AX},~\cite{Kh1},~\cite{Kh2},~\cite{Kh1-N},~\cite{Koch3},
~\cite{Koz-00},~\cite{Vl-V-Z} and the references therein).
Applications were, however, not only restricted to physics. $p$-Adic models were also
proposed in psychology, cognitive and social sciences, and, e.g., in
biology, image analysis (see~\cite{Kh2},~\cite{Kh4}).

These applications induced and stimulated a development of new branches of
$p$-adic analysis, in particular, the theory of $p$-adic wavelets.
Recall that nowadays wavelets are applied in a lot of branches of modern
mathematics and engineering area.
The first real wavelet basis (\ref{m=1-h}), (\ref{m=2-h})
was introduced by Haar in 1910. However, for almost a century nobody could find
another wavelet function (a function whose shifts and delations form an
orthogonal basis). Only in the early nineties a method for a more general construction
of the wavelet functions appeared~\cite{Mallat-1},~\cite{Mallat-2},~\cite{Meyer-1},
~\cite{Meyer-2}. This method is based on the notion of
{\em multiresolution analysis} (MRA in the sequel).
In the $p$-adic setting the situation was the following.
In 2002 S.~V.~Kozyrev~\cite{Koz0} constructed a compactly supported
$p$-adic wavelet basis (\ref{62.0-1=}) in $L^2(\bQ_p)$,
which is an analog of the real Haar basis (\ref{m=1-h}), (\ref{m=2-h}).
J.~J.~Benedetto and R.~L.~Benedetto~\cite{Ben-Ben},
R.~L.~Benedetto~\cite{Ben1} suggested a method for constructing wavelet
bases on locally compact abelian groups with compact open
subgroups. This method is applicable for the $p$-adic setting.
It is based on a {\em theory of wavelet sets} and only allows the construction
of wavelet functions whose Fourier transforms are the characteristic functions
of some sets (see~\cite[Proposition~5.1.]{Ben-Ben}).
Moreover, these authors doubted that the development of the MRA approach
is possible. In spite of the above opinions and arguments~\cite{Ben-Ben},~\cite{Ben1},
in~\cite{S-Skopina1}, the $p$-adic {\em MRA theory} in $L^2(\bQ_p)$
was developed and new $p$-adic wavelet bases were constructed.
Some important results in $p$-adic wavelet theory were obtained
in~\cite{Al-Ev-Sk},~\cite{Al-Ev-Sk-2},~\cite{Kh-Sh-Sk},~\cite{Kh-Sh-Sk-1}.
It turned out that the theory of $p$-adic wavelets plays an important role
in the study of $p$-adic pseudo-differential operators and equations~\cite{Al-Kh-Sh=book},
~\cite{Al-Kh-Sh3},~\cite{Al-Kh-Sh8},~\cite{Kh-Sh3-n},~\cite{Koz0},~\cite{Koz-00},
\cite{S-Skopina1}. This theory gives a powerful technique to deal with $p$-adic
pseudo-differential equations.
Recall that on complex-valued functions defined on
$\bQ_p$, the operation of differentiation is {\it not defined\/}. As a result,
a large number of $p$-adic models use pseudo-differential equations instead of
differential equations. The $p$-adic multidimensional fractional operator $D^{\alpha}$ 
was introduced by M.~Taibleson~\cite{Taib1} (see also~\cite{Taib3}) in the space of 
distributions ${\cD}'(\bQ_p^n)$. The spectral theory of this fractional operator
was developed by V.S.~Vladimirov in~\cite{Vl0-1}, in particular, explicit formulas 
for the eigenfunctions of this operator were constructed (see also~\cite{Vl-V-Z}).
In~\cite{Vl0} (see also~\cite{Vl-V-Z}) V.S.~Vladimirov constructed the spectral theory 
of the Schr\"{o}dinger-type operator $D^{\alpha}+V(x)$, which was further developed
by A.N.~Kochubei~\cite{Koch0-1},~\cite{Koch0-2} (see also~\cite{Koch3}).

The adele ring $\bA$ is some subring of the direct product of all possible
($p$-adic and Archimedean) completions $\bQ_p$ of the field of
rational numbers $\bQ$. The group of ideles was introduced
by Chevalley in 1936~\cite{Chevalley} as a part of his program to formulate
class field theory so that it worked for infinite-degree extensions.
The adeles were introduced by Weil in the late 1930s as an additive analogue of ideles.
The ring $\bA$ is often used in advanced parts of number theory (for example, see~\cite{Borel},
~\cite{Cassels-Frohlich},~\cite{Conrad},~\cite{Manin-1},~\cite{Manin-Panch},~\cite{Weil}).
For example, Tate's proof (see in~\cite[J.~T.~Tate, pp.305--347]{Cassels-Frohlich})
of the functional equation $Z(s)=Z(1-s)$, where
$Z(s)=\pi^{-s/2}\Gamma(s/2)\zeta(s) $ for the Riemann zeta function
$\zeta(s)=\sum_{n\in\bN}\frac{1}{n^s}$, $\Re(s)>1$,
is based on the Fourier analysis on adeles $\bA$ and ideles (see also \cite[4.7]{Conrad}).
Recently the theory of adeles has been successfully applied in various parts of
contemporary mathematical and theoretical physics. Namely, there are
adelic constructions, e.g., in statistical mechanics, stochastics, string theory,
quantum cosmology, and quantum mechanics (see~\cite{Ar-Dr-V},~\cite{Br-F},
~\cite{Dr-1}--~\cite{Dr-3},~\cite{Karwowski},~\cite[5.8]{Koch3}~\cite{Kh1},
~\cite{Kh2},~\cite{Vl-Z},~\cite{Vlad-2},~\cite{Vl-V-Z},~\cite{Yasuda}, and
the references therein).

However it should be noted that in contrast to the real and $p$-adic analysis,
the adelic analysis practically (in particular, the adelic wavelet theory and
the theory of adelic pseudo-differential operators and equations) has
not been developed so far. In particular, the necessity of the development
the wavelet theory on adeles was mentioned in~\cite{Ben-Ben}.

\subsection{Contents of the paper.}\label{s1.2}
In this paper some problems of the adelic harmonic analysis are studied.
The points to be considered are, first, the theory of adelic Haar wavelets;
and second, the theory of simplest adelic pseudo-differential operators
in connection with wavelets.

In Sec.~\ref{s2}, some facts from the $p$-adic and adelic analysis are given.
In Sec.~\ref{s3}, we recall definitions of the real MRA and present some results
on $p$-adic MRA and wavelet bases from the papers~\cite{Kh-Sh-Sk-1},~\cite{S-Skopina1}.

In Sec.~\ref{s4}, the theory of infinite tensor products of Hilbert spaces
~\cite{Neu},~\cite{Ber86} is used to generalize finite-dimensional
results to the case of adeles.
We recall the constructions of infinite tensor product of Hilbert
spaces $\cH_e=\otimes_{e,n\in\bN}H_n$, the complete von Neumann
product of infinitely many Hilbert spaces.
The space $\cH_e$ can be obtained as the closure of the union of some subspaces. We
observe {\em certain stability} of the space $\cH_e$ with respect to the
variation of the corresponding subspaces (see Lemma~\ref{X-in-H}).
In Subsec.~\ref{s4.3}, using Lemma~\ref{X-in-H} for the special stabilizing sequence
$e=(e_p)_p$, $e_p(x_p)=\phi_p(x_p)=\Omega(|x_p|_p)$ (here $\Omega(t)$ is the characteristic function
of the segment $[0,1]\subset\bR$) we show that $L^2(\bA)$ coincides
with the infinite tensor product of the Hilbert spaces $L^2(\bQ_p)$
over all possible completions of the field $\bQ$: $L^2(\bA)=\otimes_{e,p}L^2({\bQ}_{p})$
(see Lemma~\ref{l.H(Q)=L^2(A)}).

In Sec.~\ref{s5}, we apply the above scheme to construct adelic wavelet
bases (\ref{1-ad+=11}) on $L^2({\bA},dx)$ {\em generated by the tensor product of
one-dimensional Haar wavelet bases} (\ref{m=1-h}), (\ref{62.0-1=111}), (\ref{62.0-1=zz-11}).
We would like to stress that to construct adelic wavelet bases, we need the $p$-adic
wavelet bases that contain {\it functions $\phi_p(x_p)=\Omega(|x_p|_p)$}.
Thus, instead of the Haar wavelet basis (\ref{p-101**-basis})
we will use {\em modified Haar} basis (\ref{62.0-1=111}).
According to~\cite{Al-Ev-Sk}, there are no orthogonal $p$-adis MRA based wavelet
bases except for those described in Theorem~\ref{p-th4} and Corollary~\ref{m=th4}.
Thus, formula (\ref{1-ad+=11}) gives all adelic Haar wavelet bases
generated by the tensor product of one-dimensional Haar wavelet bases.
Here infinite tensor product depends on the special stabilization $e=(e_p)_p$,
where $e_p(x_p)=\Omega(|x_p|_p)$.
Recall that in~\cite[Sec.4, formula (4.1)]{Dr-2}, a basis in the space $L^2(\bA)$ associated
with eigenfunctions of harmonic oscillators was constructed (see Remark~\ref{basis-Drag}).
To construct adelic wavelet bases, we use the standard Haar measures on $\bQ_p$
and on the adele ring $\bA$. For measures on $\bQ_p$ different from Haar measure,
see~\cite[Appendix D]{Al-Kh-Sh=book}, based on~\cite{Kos_B_09}.

In Sec.~\ref{s6}, using the idea of constructing {\em separable multidimensional MRA by means
of the tensor product of one-dimensional MRAs} suggested by Y.~Meyer~\cite{Meyer-1}
and S.~Mallat~\cite{Mallat-1},~\cite{Mallat-2}, we construct adelic wavelet bases in $L^2(\bA)$.
In a general situation we show how using some system of closed subspaces $V_j^{(k)}$, $j\in\bZ$, 
in a Hilbert spaces $H_n,\,n\in\bN$, with properties (a)--(c) of the MRA (see 
Definitions~\ref{de1-Real},~\ref{de1}) one can construct various systems of subspaces with 
the same properties in the infinite tensor product of the spaces $H_n$.
In Theorem~\ref{th1-ad-mra} adelic separable MRA is constructed. The refinable function of this MRA
given by (\ref{d3-1}) is an infinite product refinable functions of all one-dimensional MRAs.
In the framework of this MRA an infinite family of adelic Haar wavelet bases (\ref{w-62.8=10**-ad-ss-1})
are constructed.

In Sec.~\ref{s7}, we introduce the adelic Lizorkin spaces of test
functions $\Phi(\bA)$ and distributions  $\Phi'(\bA)$.
These spaces are constructed by using the original real Lizorkin spaces introduced
in~\cite{Liz1}--~\cite{Liz3} and the $p$-adic Lizorkin spaces introduced
in~\cite{Al-Kh-Sh3} (see also~\cite[Ch.~7]{Al-Kh-Sh=book}).
A basic motivation for using Lizorkin spaces rather than the Schwartz
and Bruhat--Schwartz spaces of distributions ${\cS}'(\bR)$ and ${\cD}'(\bQ_p)$ is due to
the fact that the latter spaces are not invariant under the fractional operators.
Next, in Sec.~\ref{s8}, the characterization of the adelic Lizorkin spaces
in terms of wavelets is given. Namely, it is proved that any test function from
$\Phi(\bA)$ can be represented in the form of a {\em finite} combination
of adelic wavelet functions (\ref{1-ad+=11-00}), and any distribution from $\Phi'(\bA)$ can
be represented as an {\em infinite} linear combination of adelic wavelet functions (\ref{1-ad+=11-00})
(in~\cite{Al-Koz-1}, assertions of these types were stated for ultrametric Lizorkin spaces).

In the framework of our constructions the following three facts
seem to have the same {\em deep reason}: (1) functions
$\phi^{(p)}(x_p)=\Omega(|x_p|_p)$ are stabilization
functions in the adelic Bruhat--Schwartz space (see Definition~\ref{de-B-Sch}),
(2) we use the stabilization sequence $e=(e_p)_p$ where
$e_p(x_p)=\Omega(|x_p|_p)$ in proving the fact that $L^2(\bA)=\otimes_{e,p}L^2({\bQ}_{p})$,
(3) under the projection of the space $L^2(\bQ_p)$ onto $L^2(\bZ_p)$ some elements of
wavelet basis (\ref{p-101**-basis}) and Kozyrev's wavelet basis (\ref{62.0-1=})
are transformed into functions which are proportional to the same function
$\Omega(|x_p|_p)$ (see Propositions~\ref{lem-zz-1}--\ref{lem-zz-1-00}).

In Sec.~\ref{s9}, by Definition~(\ref{ad-frac-op-pd}),~(\ref{op-ad-1-2}) a class
of pseudo-differential operators on the adelic Lizorkin spaces is introduced.
The fractional operators $D^{\widehat{\gamma}}$, $\widehat{\gamma}\in \bC^{\infty}$ (see
Definition~(\ref{ad-frac-op}),~(\ref{op-ad-2}),~(\ref{62**-ad})),
and $D^{\gamma}$, $\gamma\in \bC$ (see  Definition~(\ref{ad-frac-op-2}),~(\ref{op-ad-2-1}))
belong to this class.
We prove that the Lizorkin spaces of test functions $\Phi(\bA)$ and distributions $\Phi'(\bA)$
are {\it invariant\/} under the above-mentioned pseudo-differential operators. Moreover,
a family of fractional operators on the space of distributions $\Phi'(\bA)$ forms
an abelian group. Thus the Lizorkin spaces constitute ``natural'' domains of
definition for this class of pseudo-differential operators. Note, that in~\cite{Kh-Rad} the
fractional operator was considered in $L^2(\bA)$.
In Subsec.~\ref{s9.4}, the spectral theory of adelic pseudo-differential operators
is developed. By Theorem~\ref{th-o4.2-ad}, we derive a criterion for an adelic
wavelet function to be an eigenfunction for a pseudo-differential operator. It is
proved that any wavelet function is an eigenfunction of a fractional operator.
Thus the {\em adelic wavelet analysis is closely connected with the spectral analysis
of pseudo-differential operators}.
Using results of Sec.~\ref{s9}, similarly to the $p$-adic case, one can develop
the {\em ``variable separation method''} (an analog of the classical Fourier method)
to reduce solving adelic pseudo-differential equations to solving ordinary
differential equations with respect to the real variable $t$ (for details,
see~\cite[Ch.~10]{Al-Kh-Sh=book},~\cite{Al-Khr-Sh-2011}).

\section{Preliminary results}
\label{s2}

\subsection{$p$-Adic numbers.}\label{s2.1}
We shall systematically use the notation and results from the book~\cite{Vl-V-Z}.
Let $\bN$, $\bZ$, $\bQ$, $\bR$, $\bC$ be the sets of positive integers, integers, rational, real,
and complex numbers, respectively.

According to the well-known Ostrovsky theorem, {\it any nontrivial
valuation on the field of rational numbers ${\bQ}$ is equivalent either to the real
valuation $|\cdot|$ or to one of the $p$-adic valuations $|\cdot|_p$}.
This $p$-adic norm $|\cdot|_p$ is defined as follows:
if an arbitrary rational number $x\ne 0$ is represented as
$x=p^{\gamma}\frac{m}{n}$, where $\gamma=\gamma(x)\in \bZ$ and
the integers $m$, $n$ are not divisible by $p$, then
\begin{equation}
\label{1}
|x|_p=p^{-\gamma}, \quad x\ne 0, \qquad |0|_p=0.
\end{equation}
The norm $|\cdot|_p$ is {\em non-Archimedean} and satisfies
the {\em strong triangle inequality} $|x+y|_p\le \max(|x|_p,|y|_p)$.
The completion of $\bQ$ with respect to the usual absolute value $|\cdot|$
gives the field of real numbers $\bR$. The field $\bQ_p$ of $p$-adic numbers
is defined as the completion of the field of rational numbers $\bQ$ with respect
to the norm $|\cdot|_p$. Next we will denote $|x|_\infty=|x|$,  $\bQ_{\infty}=\bR$ and
$\bZ_{\infty}=\bZ$. By
\begin{equation}
\label{Q-valuations}
V_{\bQ}=\{\infty,2,3,5,\dots\}
\end{equation}
we denote the set of indices for all valuations on the field $\bQ$.

Any $p$-adic number $x\in\bQ_p$, $x\ne 0$, is represented in the {\em canonical form}
\begin{equation}
\label{8.1}
x=\sum_{k=\gamma}^\infty x_kp^k,
\end{equation}
where $\gamma=\gamma(x)\in \bZ$, \ $x_k\in {\bF}_p=\{0,1,\dots,p-1\}$,
$x_{\gamma}\ne 0$, $\gamma\leq k<\infty$.
The series is convergent in the $p$-adic norm $|\cdot|_p$, and one
has $|x|_p=p^{-\gamma}$. The {\it fractional part} of the number
$x\in \bQ_p$ (given by (\ref{8.1})) is defined as follows
\begin{equation}
\label{8.2**}
\{x\}_p=\left\{
\begin{array}{lll}
0,\quad \text{if} \quad \gamma(x)\geq 0 \quad  \text{or} \quad x=0,&&  \\
x_{\gamma}p^{\gamma}+\cdots+x_{-1}p^{-1},
\quad \text{if} \quad \gamma(x)<0. && \\
\end{array}
\right.
\end{equation}
The set $\bQ_p^{\times}=\bQ_p\setminus\{0\}$ is the {\em multiplicative group of the
field $\bQ_p$}.
$p$-Adic numbers $\bZ_p=\{x\in \bQ_p:|x|_p\le 1\}$
are called {\em integer $p$-adic numbers}. In view of (\ref{8.1}), $\bZ_p$ consists of
$p$-adic numbers
\begin{equation}
\label{8.2-kat}
x=\sum_{k=0}^{\infty}x_kp^k.
\end{equation}
$\bZ_p$ is a subring of the ring $\bQ_p$.
The {\em multiplicative group of the ring $\bZ_p$} is the set
$$
\bZ_p^{\times}=\bigl\{x\in\bZ_p:|x|_p=1\bigr\}
=\Bigl\{x\in\bZ_p:x=\sum_{k=0}^{\infty}x_kp^k,\quad x_0\ne 0\Bigr\}.
$$

Denote by $B_{\gamma}(a)=\{x\in \bQ_p: |x-a|_p \le p^{\gamma}\}$ the ball
of radius $p^{\gamma}$ with the center at the point $a\in \bQ_p$
and by $S_{\gamma}(a)=\{x\in \bQ_p: |x-a|_p = p^{\gamma}\}
=B_{\gamma}(a)\setminus B_{\gamma-1}(a)$ the corresponding sphere,
$\gamma \in \bZ$. For $a=0$ we set $B_{\gamma}=B_{\gamma}(0)$ and
$S_{\gamma}=S_{\gamma}(0)$.

\subsection{$p$-Adic distributions.}\label{s2.2}
A complex-valued function $f$ defined on $\bQ_p$ is called
{\it locally-constant} if for any $x\in \bQ_p$ there exists an
integer $l(x)\in \bZ$ such that
$$
f(x+y)=f(x), \quad y\in B_{l(x)}.
$$

Let ${\cE}(\bQ_p)$ and ${\cD}(\bQ_p)$ be the linear spaces of
locally-constant $\bC$-valued functions on $\bQ_p$ and locally-constant
$\bC$-valued compactly supported functions (so-called Bruhat--Schwartz test functions),
respectively~\cite[VI.1.,2.]{Vl-V-Z}. If $\varphi \in {\cD}(\bQ_p)$, then
according to Lemma~1 from~\cite[VI.1.]{Vl-V-Z}, there exists $l\in \bZ$,
such that
$$
\varphi(x+y)=\varphi(x), \quad y\in B_l, \quad x\in \bQ_p.
$$
The largest of such numbers $l=l(\varphi)$ is called the
{\em parameter of constancy} of the function $\varphi$.
Let us denote by ${\cD}^l_N(\bQ_p)$ the finite-dimensional space of test
functions ${\cD}^l_N(\bQ_p)$ from ${\cD}(\bQ_p)$ having supports in
the ball $B_N$ and with {\em parameters of constancy} \, $\ge l$. The
following embedding holds:
$$
{\cD}^l_N(\bQ_p) \subset {\cD}^{l'}_{N'}(\bQ_p), \quad N\le N', \quad l\ge l'.
$$
We have ${\cD}(\bQ_p)=\lim\limits_{N\to \infty}{\rm ind}\,{\cD}_N(\bQ_p)$, where
${\cD}_N(\bQ_p)=\lim\limits_{l\to -\infty}{\rm ind}\,{\cD}_N^l(\bQ_p)$
(see~\cite[VI.2]{Vl-V-Z}). These representations give us the {\em inductive limit topology}
on the corresponding spaces.

Denote by ${\cD}'(\bQ_p)$ the set of all linear functionals (Bruhat--Schwartz distributions)
on ${\cD}(\bQ_p)$~\cite[VI.3.]{Vl-V-Z}.

The Fourier transform of $\varphi\in {\cD}(\bQ_p)$ is defined by the
formula
$$
{\widehat\varphi}(\xi)
=F[\varphi](\xi)=\int_{\bQ_p}\chi_p(\xi x)\varphi(x)\,dx,
\qquad \xi \in \bQ_p,
$$
where $dx$ is the Haar measure on $\bQ_p$ such that
$\int_{|x|_p\le 1}\,dx=1$, and
\begin{equation}
\label{8.2-char}
\chi_p(x)=e^{2\pi i\{x\}_p}
\end{equation}
is the additive character on $\bQ_p$ (see~\cite[III.1.]{Vl-V-Z}), $\{x\}_p$ is the
{\it fractional part} (\ref{8.2**}) of the number $x\in \bQ_p$.
The Fourier transform is a linear isomorphism ${\cD}(\bQ_p)$ onto
${\cD}(\bQ_p)$~\cite[III,(3.2)]{Taib3},~\cite[VII.2.]{Vl-V-Z}.
Moreover,
\begin{equation}
\label{8.2-char-ad}
\varphi \in {\cD}^l_N(\bQ_p) \quad \text{iff} \quad
F[\varphi] \in {\cD}^{-N}_{-l}(\bQ_p).
\end{equation}
The Fourier transform of a distribution $f\in {\cD}'(\bQ_p)$ is the
distribution ${\widehat f}=F[f]$ defined by the relation
$$
\langle F[f],\varphi\rangle=\langle f,F[\varphi]\rangle, \quad \forall \, \varphi\in {\cD}(\bQ_p).
$$
Here and in the sequel $\langle f,\varphi\rangle$ denotes
the action of a distribution $f$ on a test function $\varphi$.

If $f\in {\cD}'(\bQ_p)$, \ $a\in \bQ_p^{\times}$, $b\in \bQ_p$, then~\cite[VII,(3.3)]{Vl-V-Z}:
\begin{equation}
\label{alg-53A-test}
F[f(ax+b)](\xi)
=|a|_p^{-1}\chi_p\Big(-\frac{b}{a}\xi\Big)F[f(x)]\Big(\frac{\xi}{a}\Big),
\quad x,\xi\in \bQ_p.
\end{equation}

According to~\cite[IV,(3.1)]{Vl-V-Z},
\begin{equation}
\label{14.1}
F[\Omega(p^{-k}|\cdot|_p)](x)=p^{k}\Omega(p^k|x|_p), \quad k\in \bZ, \qquad x \in \bQ_p,
\end{equation}
where $\Omega(t)$ is the characteristic function of the segment $[0,1]\subset\bR$.
In particular,
\begin{equation}
\label{14.2}
F[\Omega(|\xi|_p)](x)=\Omega(|x|_p).
\end{equation}

\subsection{Adeles.}\label{s2.4}
We use the notation and results from~\cite[Ch.~III,\S1,2]{G-Gr-P} and~\cite{Conrad}.

\begin{Definition}
\label{df:adele}\rm
The {\em adeles} of $\bQ$ are
$$
\bA=\bA_{\bQ}=\bigl\{(x_p)_{p\in V_{\bQ}}\in \prod_{p\in V_{\bQ}}\bQ_p:x_p\in\bZ_p\,\,
\text{for almost all}\,\,p\not=\infty\bigr\},
$$
and the {\em ideles} of $\bQ$ are
$$
\bJ=\bJ_{\bQ}=\bigl\{(x_p)_{p\in V_{\bQ}}\in \prod_{p\in V_{\bQ}}\bQ_p^\times:x_p\in
\bZ_p^\times\,\, \text{for almost all}\,\,p\not=\infty\bigr\},
$$
where $V_{\bQ}=\{\infty,2,3,5,\dots\}$ is the set of indices.
\end{Definition}

The adele ring $\bA$ is the {\em restricted direct product} of $\bR$ and $\bQ_p$ for all $p=2,3,\dots$
with respect to the integer rings $\bZ_p$ (see~\cite{Conrad}).
If componentwise operations of addition and multiplication are introduced,
$\bA$ is a {\em ring of adeles} and $\bJ$ is a multiplicative group. The
additive group of the ring $\bA$ is called the {\em adelic group}.

There is a natural imbedding $\bQ\mapsto \bA$ given by
$$
\bQ\ni r\mapsto (r,r,\dots,r,\dots)\in \bA.
$$
Indeed, any constant sequence $(r,r,\dots,r,\dots)$ is an adele since $r\in \bZ_p$
for any $p$ not dividing the denominator of $r$.
The adeles (ideles) of the form $(r,r,\dots,r,\dots)$, where $r\in \bQ$, are called
{\em principal adeles} (respectively, {\em ideles}).

To define a {\em topology on }$\bA$, we show that $\bA$ is a union of
locally compact groups. Note that the direct product $\prod_{p\in V_{\bQ}}\bQ_p$
is not a locally compact group.

\begin{Definition}
\label{S-adele} \rm
Let $S$ be a finite subset of $V_\bQ$ such that $\infty\in S\subset V_\bQ$.
Define the {\em $S$-adeles} of $\bQ$ as
$$
\bA_{S}=\bA_{\bQ,S}=\prod_{p\in S}\bQ_p\times \prod_{p\not\in S}\bZ_p.
$$
\end{Definition}
For an arbitrary $S$, the space $\bA_{S}$ of $S$-adeles is locally compact (in the Tikhonov product topology)
as the product of a finite product of locally compact spaces $\bQ_p$, $p\in S$ by an infinite product of
compact spaces $\bZ_p$, $p\not\in S$. Since we have $\bA=\bigcup_{S}\bA_{S}$ \cite[Theorem 2.15]{Conrad}, the
adele group $\bA$ is locally compact (i.e. {\em any neighborhood of a point contains a compact
neighborhood of this point}). A {\em fundamental system of open neighborhoods of zero} in $\bA$ is the
following set (see~\cite[Theorem 2.17]{Conrad})
\begin{equation}
\label{fund-open-syst}
\prod_{p\in S}U_p\times \prod_{p\not\in S}\bZ_p,
\end{equation}
where $S$ is a finite subset of $V_\bQ$ that contains $\infty$ and $U_p$ is an open set in
$\bQ_p$ containing $0\in \bQ_p$ for $p\in S$.

The {\em non-Archimedean part} ${\widetilde\bA}$ of the adele ring $\bA$ is defined (see~\cite{Kh-Rad})
as the set of infinite sequences
$$
x'=(x_{2},\dots x_{p},\dots), \quad \text{where} \quad x_{p}\in \bQ_{p}, \quad p=2,3,\dots,
$$
and there exists a prime number $P=P(x')$ such that $x_{p}\in \bZ_{p}$ for $p \ge P$.

The sequence of adeles $\{x^{(n)},n\in \bN\}$ converges to the adele
$x$ ($x^{(n)}\to x,\, n\to \infty$) if

\, $(1)$  $x^{(n)}_p\to x_p$, $n\to \infty$, for any $p\in V_\bQ$, where $V_\bQ$ is defined by (\ref{Q-valuations});

\, $(2)$  there exists $N$ such that $x^{(n)}_p-x_p\in \bZ_p$ for all $n\ge N$.

Since the adele group $\bA$ is a locally compact commutative group,
{\em it possesses the Haar measure} which will be denoted by $dx$,
where $x=(x_{\infty},x_{2},\dots,x_{p},\dots)$. The Haar measure $dx$ can be expressed
in terms of the measures $dx_p$ on the groups $\bQ_p$ as follows:
\begin{equation}
\label{ad-meas}
dx=dx_{\infty}\,dx_{2}\cdots dx_{p}\cdots,
\end{equation}
where
$$
\int_{0}^{1}\,dx_{\infty}=1, \qquad \int_{\bZ_p}\,dx_p=1, \quad p\in V_\bQ.
$$
Here formula (\ref{ad-meas}) is understood in the following sense:
if
$$
f(x)=f_{\infty}(x_{\infty})f_{2}(x_{2})\cdots f_{p}(x_{p})
$$
is a cylindrical function, then
$$
\int_{\bA}f(x)\,dx=\int_{\bR}f_{\infty}(x_{\infty})\,dx_{\infty}
\int_{\bQ_2}f_{2}(x_{2})\,dx_2\cdots\int_{\bQ_p}f_{p}(x_{p})\,dx_p.
$$

Any {\em additive character} $\chi(x)$ on the adelic ring has the form
$$
\chi(x)=\chi_0(ax), \quad x=(x_{\infty},x_{2},\dots x_{p},\dots)\in
\bA,
$$
for some $a=(a_{\infty},a_{2},\dots a_{p},\dots)\in \bA$
(see~\cite[Ch.~III,\S1.5,(3)]{G-Gr-P}). Here
\begin{equation}
\label{ad-21}
\chi_0(a)=\prod_{p\in V_\bQ}\chi_p(a_p), \quad a\in \bA,
\end{equation}
where $\chi_{\infty}(a_{\infty})=e^{2\pi i a_{\infty}}$, and $\chi_{p}(a_{p})$
is defined by (\ref{8.2-char}), $p=2,3,\dots$. It is clear that for any
$a\in \bA$ there exists a prime number $P=P(a)$ such that $\chi_{p}(a_{p})=1$
for $p \ge P$, i.e., in fact, the product (\ref{ad-21}) is finite. In other words,
$\chi_0(a)=\exp\big(2\pi i \sigma(a)\big)$, where
$\sigma(a)=\sum_{p\in V_\bQ}a_p \,\, ({\rm mod}\,1)$.

\begin{Definition}
\label{de-B-Sch} \rm
{\rm(}see~{\rm\cite[Ch.~III,\S2.1]{G-Gr-P})}
Let $\cS(\bR)$ be the real Schwartz space of tempered test functions.
The {\em space of Bruhat--Schwartz adelic test functions} $\cS(\bA)$ consists of finite
linear combinations of elementary functions of the form
\begin{equation}
\label{ad-1}
\varphi(x)=\prod_{p\in V_\bQ}\varphi_{p}(x_{p}),
\quad x\in \bA,
\end{equation}
where the factors $\varphi_{p}(x_{p})$ are such that:

(i) $\varphi_{\infty}(x_{\infty})\in \cS(\bR)$;

(ii) $\varphi_{p}(x_{p})\in \cD(\bQ_p)$, $p=2,3,\dots$;

(iii) there exists $P=P(\varphi)$ such that $\varphi_{p}(x_{p})=\Omega(|x_{p}|_p)$ for all $p> P$
(the number $P(\varphi)$ is called the {\em parameter of finiteness} of an elementary function $\varphi$).

\end{Definition}

In view of condition $(iii)$ the space of test functions $\cS(\bA)$ admits a natural representation
\begin{equation}
\label{ad-1-test-1}
\cS(\bA)=\lim\limits_{m\in V_{\bQ}\setminus\infty}{\rm ind}\,\cS^{[m]}(\bA),
\end{equation}
where $\cS^{[m]}(\bA)$ is the subspace of all test functions with the
{\em parameter of finiteness} $m$, $m\in V_{\bQ}\setminus\infty$. The representation 
(\ref{ad-1-test-1}) equips the space $\cS(\bA)$ with the {\em inductive limit topology}.

The Bruhat--Schwartz adelic test functions are continuous on $\bA$.
The Bruhat--Schwartz space of {\em adelic distributions} $\cS'(\bA)$ was studied
in~\cite{Kh-Rad},~\cite{Dr-Kh-Rad}.

The spaces of test functions and distributions connected with the {\em non-Archimedean part} of
adeles ${\widetilde\bA}$ we denote by $\cS({\widetilde\bA})$ and  $\cS'({\widetilde\bA})$, correspondingly
(see~\cite{Kh-Rad}).

The Fourier transform of $\varphi\in {\cS}(\bA)$ is defined by the formula
$$
{\widehat\varphi}(\xi)
=F[\varphi](\xi)\stackrel{def}{=}\int_{\bA}\chi_0(\xi x)\varphi(x)\,dx,
\qquad \xi \in \bA,
$$
where $\chi_0$ is defined by (\ref{ad-21}).
It is clear that $F[{\cS}(\bA)]={\cS}(\bA)$.
If $f\in \cS'(\bA)$, then
$$
\langle F[f], \varphi\rangle\stackrel{def}{=}\langle f, F[\varphi]\rangle,
\quad \forall \, \varphi\in {\cS}(\bA),
$$
and $F[f]\in \cS'(\bA)$.

For $f,g\in L^2(\bA)$ we have
\begin{equation}
\label{11=2-ad}
(f,g)=(F[f],F[g]), \qquad ||f||_2=||F[f]||_2,
\end{equation}
where
$$
(f,g)=\int_{\bA}f(x)\overline{g(x)}\,dx
$$
is the scalar product of the functions $f$ and $g$ in $L^{2}(\bA)$.

It is clear that the space ${\cS}(\bA)$ is dense in $L^{2}(\bA)$.

\section{Real and $p$-adic MRAs and wavelet bases}
\label{s3}

\subsection{Real MRA.}\label{s3.1}
Now we recall the definitions of the real and $p$-adic multiresolution analysis.

\begin{Definition}
\label{de1-Real} \rm
(for example, see~\cite[1.2]{NPS}) A collection of closed spaces
$V_j\subset L^2(\bR)$, $j\in\bZ$, is called a
{\it multiresolution analysis {\rm(}MRA{\rm)} in $L^2(\bR)$} if the
following axioms hold

(a) $V_j\subset V_{j+1}$ for all $j\in\bZ$;

(b) $\bigcup\limits_{j\in\bZ}V_j$ is dense in $L^2(\bR)$;

(c) $\bigcap\limits_{j\in\bZ}V_j=\{0\}$;

(d) $f(\cdot)\in V_j \Longleftrightarrow f(2\cdot)\in V_{j+1}$
for all $j\in\bZ$;

(e) there exists a function $\phi \in V_0$
such that the system $\{\phi(\cdot-n), n\in \bZ\}$ is an orthonormal
basis for $V_0$.
\end{Definition}

The function $\phi$ from axiom (e) is
called {\em refinable} or {\em scaling}. One also says that a
{\em MRA is generated by its scaling function}.
The function $\phi$ is a solution of a special kind of functional
equations which are called {\em refinement equations}. Their solutions
are called {\em refinable functions}.

According to the standard MRA-scheme (see, e.g.,~\cite[\S 1.3]{NPS})
for each $j$, we define ({\em wavelet spaces})
\begin{equation}
\label{61-ad-Haar}
W_j=V_{j+1}\ominus V_j, \quad j\in\bZ.
\end{equation}
It is easy to see that
$$
f\in W_j \Longleftrightarrow f(2\cdot)\in W_{j+1},
\quad\text{for all}\quad j\in \bZ
$$
and $W_j\perp W_k$, $j\ne k$.
Taking into account axioms (b) and (c), we obtain
\begin{equation}
\label{61.1-ad-Haar}
L^2(\bR)={\bigoplus\limits_{j\in\bZ}W_j}
=V_0\oplus\Big({\bigoplus\limits_{j\in \bZ_{+}}W_j}\Big),
\end{equation}
where $\bZ_{+}=\{0\}\cup \bN$.

It is well known that the Haar {\em refinement equation} has the following form
\begin{equation}
\label{m=62.0-3-haar}
\phi^H(t)=\phi^H(2t)+\phi^H(2t-1), \quad t\in \bR.
\end{equation}
Its solution (the characteristic function $\chi_{[0,1]}(t)$ of the unit interval $[0,1]$)
\begin{equation}
\label{m=3-h}
\phi^H(t)=\chi_{[0,1]}(t)=\left\{
\begin{array}{rll}
1, && t\in [0,1], \\
0, && t\notin [0,1],  \\
\end{array}
\right. , \quad t\in \bR,
\end{equation}
generates the Haar MRA. In the framework of the Haar MRA
one can construct the well-known Haar wavelet basis in $L^2(\bR)$
\begin{equation}
\label{m=1-h}
\psi_{j n}^{H}(t)=2^{j/2}\psi^{H}\big(2^{j}t-n\big), \quad t\in \bR,
\quad j\in \bZ, \quad n\in \bZ,
\end{equation}
where
\begin{equation}
\label{m=2-h}
\psi^{H}(t)=\left\{
\begin{array}{rll}
1, && 0 \le t < \frac{1}{2}, \\
-1, && \frac{1}{2} \le t < 1, \\
0, && t\not\in [0,1), \\
\end{array}
\right.
=\chi_{[0,\frac{1}{2})}(t)-\chi_{[\frac{1}{2},1)}(t), \quad t\in \bR,
\end{equation}
is called the {\em  Haar wavelet function}
(whose dyadic shifts and delations form the {\em Haar basis} (\ref{m=1-h})).

Using the second decomposition (\ref{61.1-ad-Haar}), instead of the Haar wavelet
basis (\ref{m=1-h}) in $L^2(\bR)$, one can consider the following wavelet basis
\begin{equation}
\label{62.0-1=111-Haar}
\begin{array}{rcl}
\displaystyle
\phi^{H}(t-n)\in V_0, && n\in \bZ, \\
\displaystyle
\psi_{j n}^{H}(t)=2^{j/2}\psi^{H}\big(2^{j}t-n\big)\in W_j, && \, j\in \bZ_{+}, \, n\in \bZ,
\quad t\in \bR, \\
\end{array}
\end{equation}
where the wavelet function $\psi^{H}$ is given by (\ref{m=2-h}).
Basis (\ref{62.0-1=111-Haar}) will be called {\em modified Haar} basis.

\subsection{$p$-adic MRA.}\label{s3.2}
In the $p$-adic case a ``natural'' set of shifts for $\bQ_p$ is the
following (see~\cite{Koz0} and~\cite{S-Skopina1}):
\begin{equation}
\label{62.0**}
I_p=\left\{a\in \bQ_p: \{a\}_p=a\right\}=
\{a=\frac{a_{-\gamma}}{p^\gamma}+\cdots+\frac{a_{-1}}{p}:a_j\in{\bF}_p,\,\,-\gamma\leq j\leq -1\}.
\end{equation}

In~\cite{S-Skopina1}, similarly to Definition~\ref{de1-Real}, the following definition
was introduced.

\begin{Definition}
\label{de1} \rm
(\cite{S-Skopina1}) A collection of closed spaces
$V_j\subset L^2(\bQ_p)$, $j\in\bZ$, is called a
{\it multiresolution analysis {\rm(}MRA{\rm)} in $L^2(\bQ_p)$} if the
following axioms hold

(a) $V_j\subset V_{j+1}$ for all $j\in\bZ$;

(b) $\bigcup\limits_{j\in\bZ}V_j$ is dense in $L^2(\bQ_p)$;

(c) $\bigcap\limits_{j\in\bZ}V_j=\{0\}$;

(d) $f(\cdot)\in V_j \Longleftrightarrow f(p^{-1}\cdot)\in V_{j+1}$
for all $j\in\bZ$;

(e) there exists a function $\phi \in V_0$
such that the system $\{\phi(\cdot-a), a\in I_p\}$ is an orthonormal
basis for $V_0$.
\end{Definition}

It follows immediately
from axioms (d) and (e) that the functions $p^{j/2}\phi(p^{-j}\cdot-a)$,
$a\in I_p$, form an orthonormal basis for $V_j$, $j\in\bZ$.

According to the standard scheme (see, e.g.,~\cite[\S 1.3]{NPS})
for the construction of MRA-based wavelets, for each $j$, we define
a space $W_j$ ({\em wavelet space}) as the orthogonal complement
of $V_j$ in $V_{j+1}$, i.e.,
\begin{equation}
\label{61-ad}
V_{j+1}=V_j\oplus W_j, \qquad j\in \bZ.
\end{equation}
It is not difficult to see that
\begin{equation}
\label{61.0-ad}
f\in W_j \Longleftrightarrow f(p^{-1}\cdot)\in W_{j+1},
\quad\text{for all}\quad j\in \bZ
\end{equation}
and $W_j\perp W_k$, $j\ne k$.
Taking into account axioms (b) and (c), we obtain
\begin{equation}
\label{61.1-ad}
L^2(\bQ_p)={\bigoplus\limits_{j\in\bZ}W_j}
=V_0\oplus\Big({\bigoplus\limits_{j\in \bZ_{+}}W_j}\Big).
\end{equation}
In view of (\ref{61-ad}) and axiom $(a)$, we have
\begin{equation}
\label{61.1-ad-vv}
V_{j}=V_0\oplus\Big({\bigoplus\limits_{0\le k\le j-1}W_k}\Big)
\quad j\in \bN.
\end{equation}

If we now find a finite number of functions $\psi_{\nu} \in W_0$,
$\nu\in A$, such that the system $\{\psi_{\nu}(x-a): a\in I_p, \nu\in A\}$
forms an orthonormal basis for $W_0$, then, due to (\ref{61.0-ad}),
(\ref{61.1-ad}), the system
$$
\{p^{j/2}\psi_{\nu}(p^{-j}\cdot-a): \, a\in I_p, \, j\in\bZ, \, \nu\in A\},
$$
is an orthonormal basis for $L^2(\bQ_p)$.
Such functions $\psi_{\nu}$, $\nu\in A$, are called {\em wavelet
functions} and the corresponding basis is called a {\em wavelet basis}.

In~\cite{Kh-Sh1}, the following conjecture was proposed: to construct
a $p$-adic analog of the real Haar MRA, we can use the following
{\it refinement equation}
\begin{equation}
\label{m=62.0-3}
\phi(x)=\sum_{r=0}^{p-1}\phi\Big(\frac{1}{p}x-\frac{r}{p}\Big),
\quad x\in \bQ_p,
\end{equation}
whose solution (a {\it refinable function}) $\phi$ is the
characteristic function $\Omega\big(|x|_p\big)$ of the unit disc.
The above {\it refinement equation} is {\it natural} and reflects
the fact that the unit disc $B_{0}=\{x: |x|_p \le 1\}$ is the union of $p$
mutually disjoint discs of radius $p^{-1}$:
$$
B_{0}=\bigcup_{r=0}^{p-1}B_{-1}(r),\quad \text{where}\quad
B_{-1}(r)=\Bigl\{x: \bigl|x-r\bigr|_p \le \frac{1}{p}\Bigr\},
\quad r\in {\bF}_p.
$$
This geometric fact is the result of the ultrametric structure of
the $p$-adic field $\bQ_p$.

The term {\em $p$-adic Haar MRA} is connected with the fact that
for $p=2$ the equation (\ref{m=62.0-3}) is a {\em $2$-adic refinement equation}
$$
\phi(x)=\phi\Big(\frac{1}{2}x\Big)+\phi\Big(\frac{1}{2}x-\frac{1}{2}\Big),
\quad x\in \bQ_2,
$$
which is a direct analog of the {\em refinement equation} (\ref{m=62.0-3-haar})
generating the Haar MRA and the Haar wavelet basis (\ref{m=1-h}), (\ref{m=2-h})
in the real case.

In~\cite{S-Skopina1}, using the above relation (\ref{m=62.0-3}) as a refinement equation,
the Haar MRA was constructed for $p=2$ (for the case $p \ne 2$, see~\cite[8.4]{Al-Kh-Sh=book}).
Here
\begin{equation}
\label{m=62.0-2-vv-ad}
V_j=\overline{{\rm span}\big\{p^{j/2}\phi\big(p^{-j}\cdot-a\big):a\in I_p\big\}},
\quad j\in \bZ,
\end{equation}
where $\phi(x)=\Omega\big(|x|_p\big)$. In contrast to the Haar MRA in $L^2(\bR)$
(which generates only one wavelet basis (\ref{m=1-h}), (\ref{m=2-h})),
in the $p$-adic setting there exist {\it infinity many different
Haar orthogonal bases} for $L^2(\bQ_p)$ generated by the same MRA.
Explicit formulas for generating $p$-adic wavelet functions were
obtained in~\cite{S-Skopina1} for $p=2$ and later
in~\cite{Kh-Sh-Sk-1} for $p \ne 2$. Let us recall these results.

\begin{Theorem}
\label{p-th4}
{\rm (~\cite{Kh-Sh-Sk-1})}
The set of all compactly supported wavelet functions is given by
\begin{equation}
\label{p-101**}
\psi_{\mu}(x)=\sum_{\nu=1}^{p-1}\sum_{k=0}^{p^s-1}\alpha_{\nu;k}^{\mu}
\psi_{\nu}^{(0)}\Big(x-\frac{k}{p^s}\Big),
\quad \mu\in \bF_{p}^{\times}=\{1,2,\dots,p-1\},
\end{equation}
{\rm(}here $\supp\,\psi_{\mu}\subset B_{s}(0)$, $s\ge 0${\rm)}, where
\begin{equation}
\label{m=62.0}
\psi^{(0)}_{\nu}(x)=\chi_p\Big(\frac{\nu}{p}x\Big)\Omega\big(|x|_p\big),
\quad \nu\in \bF_{p}^{\times}, \qquad x\in \bQ_p,
\end{equation}
are Kozyrev's wavelet functions, $s=0,1,2,\dots$, and
$$
\alpha_{\nu;k}^{\mu}=
\qquad\qquad\qquad\qquad\qquad\qquad\qquad\qquad\qquad\qquad\qquad\qquad\qquad\qquad\qquad\quad
$$
\begin{equation}
\label{p-101***}
=\left\{
\begin{array}{lll}
-p^{-s}\sum_{m=0}^{p^s-1}e^{-2\pi i\frac{-\frac{\nu}{p}+m}{p^s}k}\sigma_{\mu m}u_{\mu \mu},
&& \mu=\nu, \medskip\\
-p^{-2s}\sum_{m=0}^{p^s-1}\sum_{n=0}^{p^s-1}e^{-2\pi i\frac{-\frac{\nu}{p}+m}{p^s}k}
\frac{1-e^{2\pi i\frac{\mu-\nu}{p}}}{1-e^{2\pi i\frac{\frac{\mu-\nu}{p}+m-n}{p^s}}}
\sigma_{\nu m}u_{\nu \mu}, && \mu\ne\nu, \\
\end{array}
\right.
\end{equation}
$|\sigma_{\mu m}|=1$, $u_{\mu \nu}$ are entries of an arbitrary unitary $(p-1)\times (p-1)$
matrix $U$; \ $\mu,\nu \in \bF_{p}^{\times}$; $k=0,1,\dots,p^s-1$.
\end{Theorem}

\begin{Corollary}
\label{m=th4}
{\rm (~\cite{S-Skopina1})}
Let $p=2$. For every $s=0,1,2,\dots$ the function
\begin{equation}
\label{m=101}
\psi(x)=\sum_{k=0}^{2^s-1}\alpha_{k}\psi^{(0)}\Big(x-\frac{k}{2^s}\Big),
\end{equation}
is a compactly supported wavelet function {\rm(}$\supp\,\psi \subset B_{s}(0)${\rm)}
for the Haar MRA if and only if
\begin{equation}
\label{m=108}
\alpha_{k}
=2^{-s}\sum_{r=0}^{2^s-1}\gamma_re^{-i\pi\frac{2r-1}{2^{s}}k},
\end{equation}
where the wavelet function $\psi^{(0)}:=\psi^{(0)}_{1}$ is given by {\rm (\ref{m=62.0})},
$\gamma_r\in \bC$ is an arbitrary constant such that $|\gamma_r|=1$; $k,r=0,1,\dots,2^s-1$.
\end{Corollary}

According to the general wavelet theory, all dilations and shifts of
wavelet functions (\ref{m=101}), (\ref{m=108}) or (\ref{p-101**}),
(\ref{p-101***}) form a {\em $p$-adic orthonormal Haar wavelet basis} in $L^2(\bQ_p)$:
\begin{equation}
\label{p-101**-basis}
\psi_{k;j a}(x)=p^{j/2}\psi_{k}(p^{-j}x-a), \quad k\in \bF_{p}^{\times}, \, j\in \bZ, \, a\in I_p.
\end{equation}
In particular, Kozyrev's wavelet basis in $L^2(\bQ_p)$ generated by wavelet functions (\ref{m=62.0})
is the following
$$
\psi^{(0)}_{k;j a}(x)=p^{j/2}\psi^{(0)}_{k}(p^{-j}x-a)
\qquad\qquad\qquad\qquad\qquad\qquad\qquad\qquad\qquad\qquad
$$
\begin{equation}
\label{62.0-1=}
=p^{j/2}\chi_p\Big(\frac{k}{p}(p^{-j}x-a)\Big)\Omega\big(|p^{-j}x-a|_p\big),
\quad k\in \bF_{p}^{\times}, \, j\in \bZ, \, a\in I_p.
\end{equation}
The wavelet functions (\ref{m=62.0}) can be expressed in terms of the {\it refinable function}
$\phi(x)=\Omega\big(|x|_p\big)$:
$$
\psi^{(0)}_{k}(x)
=\sum_{r=0}^{p-1}
e^{2\pi i \frac{kr}{p}}\phi\Big(\frac{1}{p}x-\frac{r}{p}\Big),
\quad k\in \bF_{p}^{\times}, \quad x\in \bQ_p.
$$

It was proved in~\cite{Al-Ev-Sk} that there are no orthogonal MRA based wavelet
bases except for those described in Theorem~\ref{p-th4} and Corollary~\ref{m=th4}.

To construct adelic wavelet bases, we will need the $p$-adic wavelet bases
that contain {\it the function $\phi(x)$}. Therefore, taking into account the second
decomposition (\ref{61.1-ad}), instead of the $p$-adic wavelet basis (\ref{p-101**-basis})
in $L^2(\bQ_p)$, we consider the following wavelet basis
\begin{equation}
\label{62.0-1=111}
\begin{array}{rcl}
\displaystyle
\phi(x-a)\in V_0, && a\in I_p, \\
\displaystyle
\psi_{k;j a}(x)=p^{j/2}\psi_{k}(p^{-j}x-a)\in W_j, && k\in \bF_{p}^{\times}, \, j\in \bZ_{+}, \, a\in I_p, \\
\end{array}
\end{equation}
where the wavelet functions $\psi_{k}$, $k\in \bF_{p}^{\times}$, are given by (\ref{p-101**})--(\ref{p-101***}).
In particular, we have the following wavelet basis
\begin{equation}
\label{62.0-1=111-K}
\begin{array}{rcl}
\displaystyle
\phi(x-a)\in V_0, & a\in I_p, \\
\displaystyle
\psi^{(0)}_{k;j a}(x)=p^{j/2}\chi_p\Big(\frac{k}{p}(p^{-j}x-a)\Big)\Omega\big(|p^{-j}x-a|_p\big)\in W_j,
& k\in \bF_{p}^{\times}, \, j\in \bZ_{+}, \, a\in I_p. \\
\end{array}
\end{equation}
The fact that bases (\ref{62.0-1=111}) and (\ref{62.0-1=111-K})
contain the {\it refinable function} $\phi(x)=\Omega\big(|x|_p\big)$ plays a crucial role in
our construction (see below Sec.~\ref{s5}). Bases (\ref{62.0-1=111}), (\ref{62.0-1=111-K})
will be called {\em modified Haar type} bases.

Taking into account the second decomposition (\ref{61.1-ad}), we obtain the following statement.
\begin{Proposition}
\label{lem-zz-1}
The restriction of basis {\rm(\ref{62.0-1=111-K})} on $\bZ_p$ constitutes
an orthonormal basis in $L^2(\bZ_p,dx^0_p)$:
\begin{equation}
\label{62.0-1=zz}
\begin{array}{rcl}
\displaystyle
\phi(x)&=&\Omega\big(|x|_p\big), \medskip \\
\displaystyle
{\widetilde\psi}^{(0)}_{k;j a}(x)&=&p^{j/2}\psi_{k}(p^{-j}x-a)\Omega\big(|p^{-j}x-a|_p\big)\big|_{\bZ_p},
\, k\in \bF_{p}^{\times}, \, a\in I_p^j, \, j\in \bZ_{+},\\
\end{array}
\end{equation}
where $dx^0_p=dx_p\big|_{\bZ_p}$ and $I_p^j=\{a\in I_p: p^ja\in \bZ_p\}$.
\end{Proposition}

\begin{proof}
It is easy to see that the restriction of Kozyrev's wavelets (\ref{62.0-1=}) is the following
$$
{\widetilde\psi}^{(0)}_{k;j a}(x)=p^{j/2}\chi_p\Big(\frac{k}{p}(p^{-j}x-a)\Big)
\Omega\big(|p^{-j}x-a|_p\big)\Big|_{\bZ_p}
\qquad\quad\qquad\qquad\qquad\quad
$$
$$
=\left\{
\begin{array}{lll}
p^{j/2}\chi_p\big(\frac{k}{p}(p^{-j}x-a)\big)\Omega\big(|p^{-j}x-a|_p\big), && k\in \bF_{p}, j\in \bZ_{+}, a\in I_p^j, \\
p^{j/2}\phi(x), && j \le -1, \, a=0,  \\
0, && j \le -1, \, a\ne 0.\\
\end{array}
\right.
$$
For $j\in \bZ_{+}$, we have ${\widetilde\psi}^{(0)}_{k;j a}(x)\big|_{\bZ_p}\ne 0$ only if
$|p^{-j}x-a|_p\le 1$. Let $a=a_{-\gamma}p^{-\gamma}+\cdots+a_{-1}p^{-1}$ and $x=\sum_{k=0}^{\infty}x_kp^k$.
Then we have
$$
p^{-j}x-a=\big(x_{0}p^{-j}+\cdots+x_{j-1}p^{-1}\big)-
\big(a_{-\gamma}p^{-\gamma}+\cdots+a_{-1}p^{-1}\big)
+\sum_{k=j}^{\infty}x_kp^{k-j}.
$$
Since $p^{-j}x-a\in \bZ_p$, the latter relation implies that
$x_{0}p^{-j}+\cdots+x_{j-1}p^{-1}=a_{-\gamma}p^{-\gamma}+\cdots+a_{-1}p^{-1}$.
Hence $p^{j}a\in \bZ_p$.

Let us calculate the scalar product
$$
\big({\widetilde\psi}^{(0)}_{k;j a}, {\widetilde\psi}^{(0)}_{k';j' a'}\big)
\qquad\qquad\qquad\qquad\qquad\qquad\qquad\qquad\qquad\qquad\qquad\qquad\qquad\qquad
$$
$$
=p^{(j+j')/2}\int_{\bZ_p}\chi_p\Big(\frac{k}{p}(p^{-j}x-a)\Big)
\chi_p\Big(-\frac{k'}{p}(p^{-j'}x-a)\Big)\times
\qquad\qquad
$$
$$
\qquad\qquad\qquad
\times
\Omega\big(|p^{-j}x-a|_p\big)\Omega\big(|p^{-j'}x-a'|_p\big)\,dx, \quad
k,k'\in \bF_{p}, j,j'\in \bZ_{+}, a,a'\in I_p.
$$
Here $\big({\widetilde\psi}^{(0)}_{k;j a}, {\widetilde\psi}^{(0)}_{k';j' a'}\big)\ne 0$ only if
$|p^{-j}x-a|_p\le 1$ and $|p^{-j'}x-a'|_p\le 1$. Just as above, we conclude that
$a,a'\in I_p$ are such that $p^{j}a\in \bZ_p$, $p^{j'}a'\in \bZ_p$. Thus
for $a\in I_p^j$, $a'\in I_p^{j'}$, we have
${\rm supp}\ \Omega\big(|p^{-j}x-a|_p\big), \,\, {\rm supp}\ \Omega\big(|p^{-j'}x-a'|_p\big)\in \bZ_p$.
Now, taking into account orthonormality of wavelet functions $\psi^{(0)}_{k;j a}$, $\psi^{(0)}_{k';j' a'}$
in $\bQ_p$, one can conclude that
$\big({\widetilde\psi}^{(0)}_{k;j a}, {\widetilde\psi}^{(0)}_{k';j' a'}\big)=\delta_{k k'}\delta_{jj'}\delta_{a a'}$,
where $\delta_{ss'}$ is the Kronecker symbol.
\end{proof}

Using Proposition~\ref{lem-zz-1} and Theorem~\ref{p-th4}, one can prove the following statement.
\begin{Proposition}
\label{lem-zz-1-2}
The restriction of the basis {\rm(\ref{62.0-1=111})} to $\bZ_p$ constitutes
in $L^2(\bZ_p,dx^0_p)$ the orthonormal basis
\begin{equation}
\label{62.0-1=zz-11}
\begin{array}{rcl}
\displaystyle
\phi(x)&=&\Omega\big(|x|_p\big), \medskip \\
\displaystyle
{\widetilde\psi}_{k;j a}(x)&=&p^{j/2}\psi_{k}(p^{-j}x-a)\big|_{\bZ_p},
\, k\in \bF_{p}^{\times}, \, a\in I_p^{(j)}, \, j\in \bZ_{+}, \\
\end{array}
\end{equation}
where the wavelet functions $\psi_{k}$, $k\in \bF_{p}^{\times}$, are given by
{\rm(\ref{p-101**})--(\ref{p-101***})},
and the set $I_p^{(j)}$ is defined in Remark~{\rm\ref{r.restrict}}.
\end{Proposition}

\begin{Remark}
\label{r.restrict}
We note that the set $I_p^{(j)}$, for which the restricted functions
${\widetilde\psi}_{k;j a}(x)$, $a\in I_p^{(j)}$ are nonzero, is finite and its
description is similar to the description of the set $I_p^j$, though more complicated.
For other wavelet bases, we will have a similar situation.
\end{Remark}

Consider the orthogonal projection
\begin{equation}
\label{P^[m]-1}
P^{[0]}:L^2({\bQ}_p,dx_p)\mapsto L^2({\bZ}_p,dx^0_p),
\end{equation}
given by
\begin{equation}
\label{P^[m]-2}
L^2({\bQ}_p,dx_p)\ni f(x)\mapsto
(P^{[0]}f)(x) =f(x)\Omega\big(|x|_p\big)\in
L^2({\bZ}_p,dx^0_p).
\end{equation}
In our case, the projections of
{\em modified Haar} bases (\ref{62.0-1=111}) and (\ref{62.0-1=111-K}) to the subspace
$L^2({\bZ}_p,dx^0_p)$ are also bases in $L^2({\bZ}_p,dx^0_p)$, where $dx^0_p$ is
the restriction of the Haar measure $dx_p$ on ${\bQ}_p$ to ${\bZ}_p$.
We note that the projections of some basis elements become zeros.

\begin{Proposition}
\label{lem-zz-1-00}
The restriction of the $p$-adic Haar MRA {\rm(}given by formula {\rm(\ref{m=62.0-2-vv-ad})}{\rm)}
to the space $L^2(\bZ_p)$ consists of the spaces
\begin{equation}
\label{m=62.0-2-vv-ad-zz}
{\widetilde V}_j=V_j\big|_{\bZ_p}
=\overline{{\rm span}\big\{p^{j/2}\phi\big(p^{-j}\cdot-a\big), \, x\in \bZ_p: a\in I_p\big\}},
\quad j\in \bZ_{+},
\end{equation}
where $V_j\big|_{\bZ_p}={\widetilde V}_0={\rm span}\big\{\phi(\cdot)\big\}$ for all
$j \le -1$, $\phi(x)=\Omega\big(|x|_p\big)$, and
$$
{\widetilde V}_0\subset {\widetilde V}_{1}\subset {\widetilde V}_{2}\subset\cdots.
$$
\end{Proposition}

Here according to {\rm(\ref{61-ad})--(\ref{61.1-ad})},
\begin{equation}
\label{61.1-ad-zz}
L^2(\bZ_p)={\widetilde V}_0\bigoplus \Big({\bigoplus\limits_{j\in\bZ_{+}}{\widetilde W}_j}\Big),
\end{equation}
where
$$
{\widetilde W}_j={\widetilde V}_{j+1}\ominus {\widetilde V}_j
\qquad\qquad\qquad\qquad\qquad\qquad\qquad\qquad\qquad\qquad\qquad
$$
$$
=\overline{{\rm span}\big\{p^{j/2}\psi_{k}(p^{-j}x-a)\big|_{\bZ_p}: k\in \bF_{p}^{\times}, \, a\in I_p\big\}},
\quad j\in \bZ_{+},
$$
and the wavelet functions $\psi_{k}$, $k\in \bF_{p}^{\times}$, are given by (\ref{p-101**})--(\ref{p-101***}).
In particular, we can use the wavelet functions (\ref{62.0-1=111-K}).

\section{Basis in the space $L^2({\bA},dx)$}\label{s4}

\subsection{Infinite tensor product of Hilbert spaces.}\label{s4.1}
We recall~\cite{Neu} (see also~\cite[Ch.1, \S 2.3]{Ber86}) the
definition of {\it infinite tensor product
${\cH}_e=\bigotimes\limits_{e,n\in {\bN}}H_n$
of Hilbert spaces} $H_n,\,\,n\in {\bN}$. Fix the sequence
\begin{equation}
\label{vect_stab} e=(e^{(n)})_{n=1}^\infty,\,\,e^{(n)}\in
H_n,\,\,\Vert e^{(n)}\Vert_{H_n}=1,
\end{equation}
called a {\it stabilizing sequence}. Denote by  $E$ the {\it set
of all stabilizing sequences}. Fix an {\it orthonormal basis}
(o.n.b) $(e^{(n)}_k)_{k\in {\bN}}$ in the Hilbert space $H_n$
such that $e^{(n)}=e^{(n)}_1,\,\,n\in {\bN}$. Let $\Lambda$ be
the set of multi-indices $\alpha=(\alpha_n)_{n\in {\bN}},$
$\alpha_n\in{\bN},\,\,n\in {\bN} $ such that $\alpha_n=1$ for sufficiently big
$n$ depending on $\alpha$.

By definition, the o.n.b. of the space
$\otimes_{e,n\in {\bN}}H_n$ {\it with the stabilizing sequence $e$} consists
of all vectors $(e_\alpha)_{\alpha\in \Lambda}$ of the following form
\begin{equation}
\label{basis-tensor}
e_\alpha=e^{(1)}_{\alpha_1}\otimes e^{(2)}_{\alpha_2}
\otimes\cdots\otimes e^{(n)}_{\alpha_n}\otimes
e^{(n+1)}\otimes\cdots, \,\,\alpha\in \Lambda
\end{equation}
where $\alpha_n=1$ for sufficiently big $n$ depending on $\alpha$. An
arbitrary element $f$ of the space ${\cH}_e$ has the following form:
\begin{equation}
\label{vector} f=\sum_{\alpha\in \Lambda}f_\alpha
e_\alpha,\,\,\text{with}\quad \Vert f \Vert^2_{{\cH}_e}
=\sum_{\alpha\in \Lambda}\vert f \vert ^2_\alpha<\infty
\end{equation}
and the scalar product $(f,g)_{{\cH}_e}$ of two vectors
$f,g\in {\cH}_e$ has the following form:
\begin{equation}
\label{(.,.)} (f,g)_{{\cH}_e}=\sum_{\alpha\in
\Lambda}f_\alpha\overline{g_\alpha}.
\end{equation}

It will often be convenient for us (see \cite[Ch.1,\S\,\,2.3]{Ber86})
to represent the set $\Lambda$ as the union of
disjoint sets, each consisting of ``finite'' sequences. Namely, for
an $\alpha\in \Lambda$ let  $\nu(\alpha)$ denote the minimal
$m=1,2,\dots$ such that $\alpha_{m+1}=\alpha_{m+2}=\cdots =1$.
Let
\begin{equation}
\label{Lambda_n}
\Lambda_n=\{\alpha\in \Lambda: \nu(\alpha)=n\},\quad n\in\bN.
\end{equation}
Obviously, $\Lambda_n\bigcap\Lambda_m=\varnothing$ ($n\not=m$) and
$\Lambda=\bigcup_{n=1}^\infty\Lambda_n$.

\begin{Example} \rm
(\cite[Ch.1, \S\,2.3, example 1]{Ber86})
Let $L^2(X_k,\mu_k)$, $k\in{\bN}$ be the space of square
integrable complex functions on the measurable space $X_k$ with a
probability measure $\mu_k$. Choose the stabilizing sequence
$e=(e^{(k)})_{k=1}^\infty$, where $e^{(k)}(x)\equiv 1,\,x\in X_k$, $k\in {\bN}$.
In this case we have
\end{Example}
\begin{Theorem}
\label{t.L2-iso-He1} The following two spaces are isomorphic:
\begin{equation}
\label{L2-iso-He1}
\bigotimes\limits_{e,\,k\in {\bN}}L^2(X_k,\mu_k)\cong
L^2\Big(\prod_{k\in {\bN}}X_k,\otimes_{k\in {\bN}}\mu_k\Big).
\end{equation}
\end{Theorem}

\begin{Example}\rm
If in the previous example the first $m$ measures
$\mu_k$ are not necessarily probability, i.e., $\mu_k(X_k)=\infty$, then we get following
statement.
\end{Example}

\begin{Theorem}
\label{t.L2-iso-He2}
The following two spaces are isomorphic:
\begin{equation}
\label{L2-iso-He2}
L^2\Big(\prod_{k=1}^mX_k,\otimes_{k=1}^m\mu_k\Big)
\otimes\Big( \bigotimes\limits_{e,\,k=m+1}^\infty L^2(X_k,\mu_k)\Big)\cong
L^2\Big(\prod_{k\in {\bN}}X_k,\otimes_{k\in {\bN}}\mu_k\Big).
\end{equation}
\end{Theorem}

\begin{Remark}\rm
\label{A=lim-m-A-(m)}
Let us consider in the infinite tensor product space ${\cH}_e=\bigotimes\limits_{e,n\in {\bN}}H_n$
the infinite tensor product $A=\otimes_{n\in {\bN}}A_n$ of bounded  operators $A_n$ acting in
the space $H_n,\,\,n\in\bN$. By definition, an operator $A$ acts on the total set of well-defined
vectors $f=\otimes_{n\in {\bN}}f_n$ in the space ${\cH}_e$ (see Lemma \ref{lemma-111})
in the following way:
\begin{equation}
\label{Af:=in-H-e}
Af:=\otimes_{n\in {\bN}}A_nf_n=A_1f_1\otimes\cdots \otimes A_mf_m\otimes\cdots,
\end{equation}
if the latter expression is well-defined in ${\cH}_e$. In the sequel we {\em shall use the projections}
$P^{[m]}:{\cH}_e\rightarrow H^{(m)}$, where the subspaces $H^{[m]}\subset {\cH}_e$
are defined on the total family of vectors $f=\otimes_{n\in {\bN}}f_n$ as follows:
$$
{\cH}_e\ni f=\otimes_{n\in {\bN}}f_n\mapsto P^{[m]}f
:=\otimes_{n=1}^mf_n\otimes e^{(m+1)}\otimes e^{(m+2)}\otimes \cdots\in H^{[m]},
$$
and
$$
H^{[m]}=\bigotimes\limits_{n=1}^mH_n\otimes e^{(m+1)}\otimes e^{(m+2)}\otimes\cdots
$$
We shall denote by $ A_{(m)}$ the projections $P^{[m]}AP^{[m]}$ of the operator
$A$ in ${\cH}_e$ onto the space $H^{[m]}$, i.e. define $A_{(m)}:=P^{[m]}AP^{[m]}$. It is clear that
$$
A_{(m)}=\bigotimes\limits_{n=1}^mA_n\otimes\bigotimes\limits_{n>m} Id_n,
$$
where $Id_n$ is the identity operator in $H_n$. We have the strong convergence $A_{(m)}\!\to\! A$ as
$m\!\to\!\infty$ on a suitable set of vectors $f=\otimes_{n\in {\bN}}f_n$. It is sufficient to estimate
$\Vert (A-A_{(m)})f\Vert_{{\cH}_e}$. For details, see \cite[Ch. I, \S 2.7.]{Ber86}.

As usual, we denote by $A\vert_X$ {\em the restriction} of an operator $A$ acting in a Hilbert space $H$
to the invariant subspace $X\subset H$.
\end{Remark}

\subsection{Complete von Neumann product of infinitely many Hilbert spaces.}\label{s4.2}
The {\it complete von Neumann tensor product}
${\cH}=\otimes_{k\in {\bN}}H_k$  of Hilbert spaces
$H_k$, $k\in {\bN} $ is {\em by definition} the orthogonal sum of the
spaces ${\cH}_e$ (\cite{Neu}, see also \cite[Ch.1,\S \,2.10]{Ber86})
\begin{equation}
\label{complete} {\cH}=\bigoplus\limits_{e\in E_{/\sim}}{\cH}_e
\end{equation}
over all possible {\it equivalence classes} $E_{/\sim}$ of
stabilizing sequences $e$.

To be more precise, fix the space ${\cH}_e=\bigotimes\limits_{e,n\in {\bN}}H_n$. We
define the vector $f=\otimes_{n\in {\bN}}f^{(n)}$, where
$f^{(n)}\in H_{n}$, as the week limit (if it exists) in
${\cH}_e$ of the vectors
\begin{equation}
\label{f[m]} f[m]=f^{(1)}\otimes\cdots\otimes f^{(m)}\otimes
e^{(m+1)}\otimes e^{(m+2)}\otimes\cdots
\end{equation}
as $m\to\infty$. Since the set ${\rm span}\,\{e_\alpha :\alpha\in \Lambda\}$ is
dense in ${\cH}_e$, the week limit of the vectors $f[m]$ exists if and
only if: $1)$ the norms $\Vert f[m]\Vert_{{\cH}_e}$ are uniformly
bounded with respect to $m=1,2,\dots$, and
$2)\lim_{m\to\infty}(f[m],e_\alpha)_{{\cH}_e}$ exists for each
$\alpha\in \Lambda$. The following statements are proved in~\cite[Ch.I,\S \,2.10]{Ber86}.

\begin{Lemma}
\label{lemma-111}
The strong limit of vectors {\rm(\ref{f[m]})} exists in
${\cH}_e$, as $m\to\infty$, if and only if the product
$\prod_{n=1}^\infty\Vert f^{(n)}\Vert_{H_n}$ and
$\prod_{n=q}^\infty(f^{(n)},e^{(n)})_{H_n}$ ($q=1,2,\dots$) converge
to finite numbers, and we have $\prod_{n=1}^\infty\Vert f^{(n)}\Vert_{H_n}=0$ when
$\prod_{n=q}^\infty(f^{(n)},e^{(n)})_{H_n}=0$ for each $q$.
\end{Lemma}

\begin{Corollary}
\label{f[m]-stab} If $f^{(n)}$ in {\rm(\ref{f[m]})} are taken to be unit
vectors, then the strong limit $\lim_{m\to\infty}f[m]$ exists if and
only if for some $q=1,2,\dots$ the product
$\prod_{n=q}^\infty(f^{(n)},e^{(n)})_{H_n}$ converges to a finite
nonzero number.
\end{Corollary}

\begin{Definition}
\label{equivalence}\rm
(see \cite[ Ch. I,\,\S\,2.10, Theorem 2.9]{Ber86})
Consider the set $E$ of all stabilizing sequences
$e=(e^{(n)})_{n=1}^\infty$ of the form (\ref{vect_stab}). A stabilizing sequence
$l\in E$ is said {\em to be equivalent to a stabilizing sequence} $e\in E$ ($l\sim e$),
if each strong limit
$$
l^{(1)'}\otimes l^{(2)'}\otimes\cdots=\lim_{m\to\infty}l^{(1)'}\otimes\cdots
\otimes l^{(m)'}\otimes e^{(m+2)}\otimes e^{(m+2)}\otimes\cdots,
$$
exists in ${\cH}_e$, where $(l^{(n)'})_{k=1}^\infty$ is the
sequence $(l^{(n)})_{k=1}^\infty$, ``diluted'' by the vectors
$e^{(n)}$, i.e. each $l^{(n)'}$ is equal either to $l^{(n)}$ or to
$e^{(n)}$. The relation $\sim$ is an equivalence relation and we
denote by $E_{/\sim}$ the {\em set of all equivalent classes of $E$}.
\end{Definition}

\begin{Theorem}
\label{He-iso-He'}
Two infinite tensor products
${\cH}_e=\bigotimes\limits_{e,n\in {\bN}}H_n$ and
${\cH}_{l}=\bigotimes\limits_{l,n\in {\bN}}H_n$ corresponding to
two equivalent stabilizing sequences $e=(e^{(n)})_{n=1}^\infty$ and
$l=(l^{(n)})_{n=1}^\infty$ are isomorphic. For $e\not\sim l$, the
spaces  ${\cH}_e$ and ${\cH}_l$ are orthogonal.
\end{Theorem}

\subsection{On some subspaces of the infinite tensor products.}
Let $X_n$ be some subspaces in the Hilbert spaces $H_n,\,\,n\in \bN$
and let $(e^{(n)}_k)_{k\in \bZ}$  be an orthonormal basis in $H_n$
such that the orthonormal basis in $X_n$ is $(e^{(n)}_k)_{k\in
\bN}$, and let $(e^{(n)})_{n\in \bN}$ be the stabilizing
sequence $e^{(n)}=e^{(n)}_1,\,\,n\in {\bN}$. Note that
\begin{equation}
\label{X-in-H+stab}
e^{(n)}=e^{(n)}_1\in X_n\subset H_n ,\quad n\in \bN.
\end{equation}

Consider two spaces
${\cH}_e$ and ${\cH}_e^{l}(X)$, $l\in \bN$, where
$$
{\cH}_e=\bigotimes\limits_{e,n\in {\bN}}H_n
=H_1\otimes H_2\otimes\cdots \otimes H_l\otimes
H_{l+1}\otimes H_{l+2}\otimes\cdots,
$$
\begin{equation}
\label{H(X)}
{\cH}_e^{l}(X)=\bigotimes_{k=1}^l
H_k\otimes\bigotimes\limits_{e,k=l+1}^\infty X_k= H_1\otimes H_2\otimes\cdots \otimes H_l\otimes
X_{l+1}\otimes\cdots.
\end{equation}
In the particular case $X_n=\bC e^{(n)}$ we have
$$
{\cH}_e^{l}(X)=H_1\otimes H_2\otimes\cdots \otimes H_l\otimes
\bC e^{(l+1)}\otimes \bC e^{(l+2)}\otimes\cdots ,\quad l\in \bN.
$$

\begin{Lemma}
\label{X-in-H}
For arbitrary subspaces $X_n$ in $H_n$ and a stabilizing sequence $(e^{(n)})_{n\in \bN}$
with properties {\rm(\ref{X-in-H+stab})}, we have
\begin{equation}
\label{H=bigcup.H.(X)}
{\cH}_e=\overline{\bigcup_{l\geq 1}{\cH}_e^{l}(X)}.
\end{equation}
\end{Lemma}

\begin{proof}
The basis in the space  ${\cH}_e$ is the following (see (\ref{basis-tensor})):
\begin{equation}
\label{bas-1}
e_\alpha=e^{(1)}_{\alpha_1}\otimes e^{(2)}_{\alpha_2}
\otimes\cdots\otimes e^{(n)}_{\alpha_n}\otimes
e^{(n+1)}\otimes\cdots, \,\,\alpha\in \Lambda=\bigcup_{n\in \bN}\Lambda_n,
\end{equation}
where $\Lambda_n$ is defined by (\ref{Lambda_n}), and the basis in
the space ${\cH}_e^{l}(X)$ is
$$
e_\alpha^{l}=e^{(1)}_{\alpha_1}\otimes e^{(2)}_{\alpha_2}
\otimes\cdots\otimes e^{(l)}_{\alpha_l}\otimes
e^{(l+1)}_{\alpha_{l+1}} \otimes\cdots\otimes
e^{(l+k)}_{\alpha_{l+k}}\otimes e^{(l+k+1)}\otimes\cdots,
\qquad\qquad\qquad
$$
\begin{equation}
\label{bas-2}
\qquad\qquad\qquad\qquad\qquad\qquad\qquad\qquad
\alpha\in \Lambda_l=\bigcup_{k\in \bN\cup\{0\}}\Lambda_{l,k},
\end{equation}
where
$$
\Lambda_{l,k}=\{\alpha\in \Lambda: \alpha_{i}\in \bZ, \, 1\leq i\leq l,
\,\, \alpha_{i}\in \bN,\, l+1\leq i\leq l+k,\,\,\alpha_{l+k+i}=1,\,\,i>1\}.
$$
Obviously ${\cH}_e\supseteq{\cH}_e^{l}(X)$ for all $l\in \bN$,
hence ${\cH}_e\supseteq\overline{\bigcup_{l\geq 1}{\cH}_e^{l}(X)}$.

We show that ${\cH}_e\subseteq\overline{\bigcup_{l\geq 1}{\cH}_e^{l}(X)}$.
It is sufficient to show that all vectors $e_\alpha$ of the form (\ref{bas-1}) are
contained in the family of vectors $e_\alpha^{l}$ of the form (\ref{bas-2}).
Indeed, if we take in (\ref{bas-2}) $l=n$ and $k=0$, we obtain all vectors
$e_\alpha$ of the form (\ref{bas-1}).
\end{proof}

\subsection{Infinite tensor product of Hilbert spaces $L^2({\bQ}_{p},dx_{p})$.}\label{s4.3}
Using the results of the previous section we can define the infinite tensor product
\begin{equation}
\label{otimes-L^2(Q)}
{\cH}_e(\bQ)=
\bigotimes\limits_{e,\,p\in V_{\bQ}}L^2({\bQ}_{p},dx_{p})
\end{equation}
of Hilbert spaces $L^2({\bQ}_{p},dx_{p})$ with an {\em arbitrary stabilizing sequence}
$e=(e^{(p)})_{p\in V_{\bQ}}$ $e^{(p)}\in L^2({\bQ}_{p},dx_{p})$.

Let $\{e_{\alpha_{\infty}}^{(\infty)}\}_{\alpha_{\infty}\in {\bI}_\infty}$
and $\{e_{\alpha_{p}}^{(p)}\}_{\alpha_{p}\in {\bI}_p}$ be {\em arbitrary orthonormal bases}
in $L^2(\bQ_{\infty},dx_\infty)$ and $L^2(\bQ_{p},dx_p)$, respectively,
where ${\bI}_{\infty}$ and ${\bI}_{p}$ are the corresponding multi-indices.
Then the orthonormal basis in the space ${\cH}_e(\bQ)$ is the following:
\begin{equation}
\label{.}
e_\alpha =e_{\alpha_{\infty}}^{(\infty)}
\otimes\bigotimes\limits_{p\leq m}e_{\alpha_{p}}^{(p)}
\otimes\bigotimes\limits_{m<p}e^{(p)},\quad \alpha=(\alpha_\infty,
\alpha_2,\alpha_3,\dots)\in \Lambda,
\end{equation}
where $(e^{(p)})_{p\in V_{\bQ}}$, $e^{(p)}\in L^2({\bQ}_{p},dx_{p})$, is
some stabilizing sequence and $\Lambda$ is defined below.

\begin{Definition}
\label{d.Lambda}\rm
Define $\Lambda$ as the  set of multi-indices $\alpha=(\alpha_p)_{p\in V_{\bQ}},$ $\alpha_p\in {\bI}_p$
such that $e_{\alpha_p}^{(p)}=e^{(p)}$ for sufficiently big $p$ depending on $\alpha$.
\end{Definition}

Fix the stabilizing sequence of the form
\begin{equation}
\label{stab-phi}
e=(e^{(p)})_{p\in V_{\bQ}} \qquad e^{(p)}(x_p)=\phi_p(x_p)=\Omega\big(|x_{p}|_p\big)\in L^2({\bZ}_{p},dx^0_{p}), \quad p\in V_{\bQ}.
\end{equation}

Using Lemma~\ref{X-in-H}, for the stabilizing sequence (\ref{stab-phi}),
we obtain the following description of the space $L^2(\bA,dx)$.
\begin{Lemma}
\label{l.H(Q)=L^2(A)}
We have
\begin{equation}
\label{H(Q)=L^2(A)}
L^2(\bA,dx) ={\cH}_e(\bQ)=\bigotimes\limits_{e,\,p\in V_{\bQ}}L^2({\bQ}_{p},dx_{p}).
\end{equation}
\end{Lemma}

\begin{proof}
It is sufficient to use Lemma~\ref{X-in-H} and set
$X_p=L^2(\bZ_p)$, $H_p=L^2(\bQ_p)$, ${\cH}_e^p=L^2({\bA}^p,\mu^p)
=\bigotimes\limits_{q\leq p}L^2(\bQ_q)\otimes\bigotimes\limits_{q> p}L^2(\bZ_q)$, $p\in V_\bQ$,
where ${\bA}^p={\bQ}_\infty \times \prod_{p'\leq p}{\bQ}_{p'}\times\prod_{p'> p}{\bZ}_{p'}$,
$p\in V_{\bQ}\setminus{\infty}$ (see (\ref{A^p})),
$\mu^p$ is the restriction on ${\bA}^p$ of the Haar measure $dx$ on ${\bA}$.
\end{proof}

\subsection{Some remarks about the stabilizations.}\label{s4.4}
Consider the space (\ref{otimes-L^2(Q)})
\begin{equation}
\label{H-e(Q)}
{\cH}_e(\bQ)= \bigotimes\limits_{e,\,p\in V_{\bQ}}L^2({\bQ}_{p},dx_{p}),
\end{equation}
where the stabilizing sequence $e=(e^{(p)})_{p\in V_{\bQ}}$
has the special form $e^{(p)}(x_p)=\phi_p(x_p)\in L^2({\bQ}_{p},dx_{p})$.
We show that the {\em sequence $e$ and $M^j_{V_\bQ}e$ for
$j\in \bZ$ are not equivalent} (see Definition~\ref{equivalence}), where
\begin{equation}
\label{M^j_VQ}
M^j_{V_\bQ} e:=(M^j_p\phi^{(p)})_{p\in V_\bQ}\quad\text{and}
\quad
M^j_p\phi^{(p)}(x_p):=p^{-j/2}\phi^{(p)}(p^jx_p).
\end{equation}
Indeed, we have
$$
(\phi^{(p)},M^j_p\phi^{(p)})_{L^2(\bQ_p)}
=\int_{\bQ_p}\phi^{(p)}(x_p)p^{-j/2}\phi^{(p)}(p^jx_p)dx=
\left\{
\begin{array}{rll}
p^{-j/2}, && j\geq 0, \\
p^{j/2}, && j\leq -1.  \\
\end{array}
\right.
$$
Set $c_\infty=(\phi^{(\infty)},M^j_\infty\phi^{(\infty)})_{L^2(\bQ_p\infty)}$, then we get
$$
c_\infty^{-1}(e,M^j_{V_\bQ}e)=\prod_{p\in
V_\bQ\setminus\infty}(\phi^{(p)},M^j_p\phi^{(p)})_{L^2(\bQ_p)}
=\left\{
\begin{array}{rll}
\left(\prod_{p\in V_\bQ\setminus\infty}p\right)^{-j/2}, && j\geq 0, \\
\left(\prod_{p\in V_\bQ\setminus\infty}p\right)^{j/2}, && j \le -1,  \\
\end{array}
\right.
$$
Thus in both cases the product is divergent, i.e., $(e,M^j_{V_\bQ}e)=0$
hence the sequences $e$ and $M^j_{V_\bQ}e$ are not equivalent!

A similar argument holds if we take as the stabilizing sequence
some elements of Kozurev's basis $\psi_{\alpha_p}^{(p)}$ in the space
$L^2(\bQ_p)$, i.e., $e=(\psi_{\alpha_p}^{(p)})_{p\in V_\bQ}$. In
this case we immediately get
$(\psi_{\alpha_p}^{(p)},M^j_p\psi_{\alpha_p}^{(p)})_{L^2(\bQ_p)}=0$ for all
$j\in \bZ$. Hence the sequences
$e=(\psi_{\alpha_p}^{(p)})_{p\in V_\bQ}$ and
$M^j_{V_\bQ}e=(M^j_p\psi_{\alpha_p}^{(p)})_{p\in V_\bQ}$ are not equivalent.

The above considerations {\em force us to consider the following space}
\begin{equation}
\label{+H(Q)}
\mathcal{H}_{e,\bZ}(\bQ)=\oplus_{j\in\bZ}\mathcal{H}_{M^j_{V_\bQ} e}(\bQ)
\end{equation}
{\em to be sure that the operator $M^j_{V_\bQ}$ is well-defined}.
Further we set
\begin{equation}
\label{W_k} W_k= \mathcal{H}_{M^k_{V_\bQ} e}(\bQ)
\end{equation}
\begin{equation}
\label{V_j}
V_j=\oplus_{k=-\infty}^jW_k=\oplus_{k=-\infty}^j\mathcal{H}_{M^k_{V_\bQ} e}(\bQ).
\end{equation}

\begin{Theorem}
\label{MRA M^j_-VQ}\rm
The collection of closed spaces $V_j\subset \mathcal{H}_{e,\bZ}(\bQ)$, $j\in\bZ$,
defined by (\ref{V_j}) is a multiresolution analysis in $\mathcal{H}_{e,\bZ}(\bQ)$,
i.e., the following properties hold:

(a) $V_j\subset V_{j+1}$ for all $j\in\bZ$;

(b) $\bigcup\limits_{j\in\bZ}V_j$ is dense in $\mathcal{H}_{e,\bZ}(\bQ)$;

(c) $\bigcap\limits_{j\in\bZ}V_j=\{0\}$;

(d) $f(\cdot)\in V_j \Longleftrightarrow f(M^j_{V_\bQ}\cdot)\in V_{j+1}$
for all $j\in\bZ$;

(e) there exists a basis $(e_i)_{i\in I}$ in the space $W_0$ such that $(M^k_{V_\bQ}e_i)_{i\in I}$
is a basis in the space $W_k$, $k\in \bZ$.
\end{Theorem}

\subsection{Complete von Neumann product of Hilbert spaces $L^2({\bQ}_{p},dx_{p})$.}\label{s4.41}
It is natural to consider also the complete von Neumann product (see (\ref{complete-Q}))
of Hilbert spaces $L^2({\bQ}_{p},dx_{p}),\,\,p\in V_\bQ$
\begin{equation}
\label{complete-Q}
{\cH}(\bQ)=\bigoplus\limits_{e\in E_{/\sim}}{\cH}_e(\bQ),
\end{equation}
where ${\cH}_e(\bQ)$ is defined by (\ref{H-e(Q)}). Perhaps this space could be useful
in further development of analysis on the adele space $\bA$. The space
$\mathcal{H}_{e,\bZ}(\bQ)$ defined by (\ref{+H(Q)}) certainly, contains the space $L^2(\bA)$ and
is roughly speaking an infinite direct sum of non isomorphic copies of the space similar
to $L^2(\bA)$.

\subsection{Basis on $L^2({\bA},dx)$.}\label{s4.5}
We can construct a basis in $L^2({\bA},dx)$ using description (\ref{H(Q)=L^2(A)}) of
the space $L^2({\bA},dx)$ in Lemma \ref{l.H(Q)=L^2(A)}. Without using this description
we can proceed as follows.

We present $\bA$ as the union of some subgroups $\bA^{p}$:
$\bA=\bigcup_{p\in V_{\bQ}}\bA^p$ (see (\ref{A^p})), then we realize
$L^2(\bA^{p},\mu^{p})$ as the infinite tensor product of Hilbert
spaces (see (\ref{H^p=otimes})) and use the following considerations.

Let $(X,\mu)$ be some measurable space and let $X$ be the union of the measurable sets $X_n$,
$X=\bigcup_{n\in \bN}X_n$. Set $\mu_n=\mu|_{X_n}$,
$\nu_n=\mu|_{X_n\setminus X_{n-1}},\,\,n\geq 2,\quad \nu_1=\mu_1$ and $X_0=\emptyset$.
In this case we have
\begin{equation}
\label{H=bigcup-n(H-n)}
L^2(X,\mu)=\oplus_{n\in\bN}L^2(X_n\setminus X_{n-1},\nu_n)=\overline{\bigcup_{n\in\bN}L^2(X_n,\mu_n)}.
\end{equation}
For any prime $p$ we denote by $p_-$ and $p_+$ the previous and the following
primes.

Define the subgroup ${\bA}^p$ of the group ${\bA}$ as follows:
\begin{equation}
\label{A^p}
{\bA}^p={\bQ}_\infty \times \prod_{p'\leq p}{\bQ}_{p'}
\times\prod_{p'> p} {\bZ}_{p'},\quad p\in V_{\bQ}.
\end{equation}
We have $\bA=\bigcup_{p\in V_\bQ}\bA^p$.

Let $dx^0_p$ be the restriction of the Haar
measure $dx_p$ on ${\bQ}_p$ to the subgroup ${\bZ}_p$, let
$d\mu^p$ be the restriction to the subgroup ${\bA}^p$ of the
measure $dx$ on ${\bA}$, and let $d\nu_p$ be the restriction of the
measure $dx$ to the subset ${\bA}^p\setminus {\bA}^{p_-}$. Then we have
\begin{equation}
\label{A^p-1}
d\mu^p=dx_\infty\otimes \Big(\otimes_{p'\leq p}dx_{p'}\Big)\otimes\Big(\otimes_{p'> p}dx^0_{p'}\Big).
\end{equation}
and using (\ref{H=bigcup-n(H-n)}) we get.

\begin{Lemma}
\label{l.HH^p} For an arbitrary prime $p$ we have the following
description:
\begin{equation}
\label{HH^p}
L^2({\bA},dx)=L^2({\bA}^p,\mu^p)
\oplus\Big(\bigoplus_{p'> p}L^2({\bA}^{p'}\setminus {\bA}^{p'_{-}},\nu_{p'})\Big)
=\overline{\bigcup_{p\in V_\bQ}L^2({\bA}^p,\mu^p)}.
\end{equation}
\end{Lemma}
To construct a basis in the space $L^2({\bA},dx)$  using the
latter description, it is sufficient to construct a basis in the
space $L^2({\bA}^p,\mu^p)$. The latter space $L^2({\bA}^p,\mu^p)$
is the infinite tensor product of the spaces
$L^2({\bQ}_p,dx_p)$ and $L^2({\bZ}_p,dx^0_p)$:
\begin{equation}
\label{H^p=otimes}
L^2({\bA}^p,\mu^p)\!=\!L^2({\bQ}_\infty,dx_\infty)
\!\otimes\!\Big(\bigotimes\limits_{p'\leq p}L^2({\bQ}_{p'},dx_{p'})\Big)
\!\otimes\!\Big(\bigotimes\limits_{e,p'> p}L^2({\bZ}_{p'},dx^0_{p'})\Big).
\end{equation}

\begin{Remark}\rm
Since the measures $dx_p^0$ on ${\bZ}_p$ are {\em probability
measures}, decomposition (\ref{H^p=otimes}) allows us to use an
explicit description of the basis in the infinite tensor product
$\otimes_{k\in {\bN}}L^2(X_k,\mu_k)$ of Hilbert spaces
$L^2(X_k,\mu_k)$ with probability measures $\mu_k$, $k\in {\bN}$
(for example, see~\cite{Ber86}) in order to construct a basis
in the spaces $L^2({\bA}^p,\mu^p)$ and $L^2({\bA},dx)$.
\end{Remark}

Suppose that $\{e_{\alpha_{\infty}}^{(\infty)}\}_{\alpha_{\infty}\in {\bI}_\infty}$,
and $\{e_{\alpha_{p}}^{(p)}\}_{\alpha_{p}\in {\bI}_p}$ are {\em arbitrary}
orthonormal bases in the spaces $L^2(\bQ_{\infty})$ and $L^2(\bQ_{p})$,
respectively, ${\bI}_{\infty}$ and ${\bI}_{p}$ are the corresponding indices.
Fix the stabilizing sequence $(e^{(p)})_{p\in V_\bQ}$, where $e^{(p)}$ is some element of the basis
$\{e_{\alpha_{p}}^{(p)}\}_{\alpha_{p}\in {\bI}_p}$ for all $p\in V_{\bQ}$ such that
$e^{(p)}(x)\in L^2({\bZ}_{p},dx^0_{p})$.
We can construct a basis in the space $L^2({\bA}^p,\mu^p)$ for all
$p\in V_{\bQ}$ using decomposition (\ref{H^p=otimes}). In such
a way we can construct a basis in $L^2({\bA},dx)$ using (\ref{HH^p}).
As it was mentioned before, using Lemma~\ref{l.H(Q)=L^2(A)} and the basis in the infinite tensor product
we obtain.
\begin{Theorem}
\label{A-00-0}
The vectors
\begin{equation}
\label{1-ad-0}
e_\alpha=e_{\alpha_{\infty}}^{(\infty)}
\otimes\bigotimes\limits_{q\leq m}e_{\alpha_{q}}^{(q)}
\otimes\bigotimes\limits_{m<q}e^{(q)}, \quad \alpha\in
\Lambda=\bigcup_{m\in V_\bQ}\Lambda_m,
\end{equation}
form the orthonormal basis in the space $L^2({\bA},dx)$, where
$\Lambda$ is the set of multi-indices $\alpha=(\alpha_p)_{p\in V_\bQ}$,
such that $e_{\alpha_p}^{(p)}=e^{(p)}$  for sufficiently big $p$
depending on $\alpha$ {\rm({\em see} (\ref{d.Lambda}))},
$\Lambda_p=\{\alpha\in \Lambda: \nu_a(\alpha)=p\}$, $p\in V_\bQ$, and $\nu_a(\alpha)$ denote the minimal
$p\in V_\bQ$ such that $e_{\alpha_q}^{(q)}=e^{(q)}$ for $q>p$.
\end{Theorem}

\begin{Corollary}
\label{A-00-1}
The vectors
\begin{equation}
\label{1-ad+=A-00}
\widetilde{e}_{\alpha}=
\bigotimes\limits_{q\leq m}
e_{\alpha_{q}}^{(q)}
\otimes\bigotimes\limits_{m<q}e^{(q)}, \quad \alpha\in {\widetilde\Lambda},
\end{equation}
form an orthonormal basis in the space
$L^2({{\widetilde\bA}},dx)$, where ${\widetilde\Lambda}$ is the set of multi-indices
$\alpha=(\alpha_p)_{p\in V_\bQ\setminus\{\infty\}}$ such that
$\psi_{\alpha_p}^{(p)}=e^{(p)}$ for sufficiently big $p$ depending on $\alpha$, i.e.,
\begin{equation}
\label{A_(p,m)-phi-00}
\widetilde{\Lambda}=\bigcup_{m\in V_\bQ\setminus\infty}\widetilde{\Lambda}_{m}\quad \text{and}
\quad \widetilde{\Lambda}_m=\{\alpha\in \widetilde\Lambda: \nu_a(\alpha)=m\},
\quad m\in V_\bQ\setminus\infty.
\end{equation}
\end{Corollary}

\section{Adelic wavelet bases generated by tensor product of one-dimensional wavelet bases}
\label{s5}

To construct an adelic wavelet basis, we apply the above scheme from Subsec.~\ref{s4.5}
to the one-dimensional bases (\ref{m=1-h}), (\ref{62.0-1=111}), (\ref{62.0-1=zz-11}).
In particular, one can use the Haar wavelet bases (\ref{m=1-h}), (\ref{62.0-1=111-K}), (\ref{62.0-1=zz}).
Let
\begin{equation}
\label{real-basis}
\psi_{\alpha_{\infty}}^{(\infty)}(x_{\infty})
=\psi_{j_{\infty}a_{\infty}}^{(\infty)}(x_{\infty})
=\psi_{j_{\infty}a_{\infty}}^{H}(x_{\infty}),
\quad \alpha_{\infty}=(j_{\infty},a_{\infty})\in {\bI}_\infty:=(\bZ, \bZ);
\end{equation}
be the real Haar wavelet basis (\ref{m=1-h}) in $L^2(\bR)$,
\begin{equation}
\label{basis-phi-inQ}
\begin{array}{rcl}
\displaystyle
\psi_{\alpha_{p}}^{(p)}(x_{p})&=&\psi^{(p)}_{k_p; j_p a_p}(x_{p}),
\quad \alpha_{p}=(k_p,j_p,a_p)\in {\bI}_p^+:=({\bF}_p^\times,\bZ_{+},I_p), \\
\displaystyle
\psi_{\alpha_{p}}^{(p)}(x_{p})&=&\psi^{(p)}_{a_p}(x_{p})=\phi^{(p)}(x_p-a_p),
\quad \alpha_{p}=(0,0,a_p)\equiv a_p \in I_p; \\
\end{array}
\end{equation}
be the $p$-adic {\em modified Haar} wavelet basis (\ref{62.0-1=111}) in
$L^2(\bQ_p)$, where $I_p$ is defined by (\ref{62.0**}), and let
\begin{equation}
\label{basis-phi-inZ}
\begin{array}{rcl}
\displaystyle
\psi_{\alpha_{p}}^{(p)}(x_{p})&=&\psi^{(p)}_{k_p;\, j_p a_p}(x_{p}),
\quad \alpha_{p}=(k_p,j_p,a_p)\in {\bI}_p^+=({\bF}_p^\times,\bZ_+,I_p), \\
\displaystyle
\psi^{(p)}_{\alpha_{p}}(x_{p})&=&\psi^{(p)}_{0}(x_{p})=\phi^{(p)}(x_p),
\quad \alpha_{p}=(0,0,0)=0; \\
\end{array}
\end{equation}
be the $p$-adic Haar wavelet basis (\ref{62.0-1=zz-11}) in $L^2(\bZ_p)$,
which is the restriction of the basis (\ref{basis-phi-inQ}) to $\bZ_p$.
We recall that the restrictions of some basis elements (\ref{basis-phi-inQ})
are equal to zero (see Remark~\ref{r.restrict} and Proposition~\ref{lem-zz-1}).
{\em Here and in what follows, for a stabilizing sequence we take} (\ref{stab-phi}).

Hence we have the following bases in $L^2(\bQ_p)$ and in $L^2(\bZ_p)$, respectively,
\begin{equation}
\label{basis-111}
\psi_{\alpha_{p}}^{(p)}, \quad \alpha_{p}\in
{\bI}_p={\bI}_p^+\bigcup I_p; \qquad \psi_{\alpha_{p}}^{(p)},
\quad \alpha_{p}\in {\bI}_p^{(0)}:={\bI}_p^+\bigcup\{0\}\subset {\bI}_p,
\end{equation}
where $\{\psi_{\alpha_{p}}^{(p)}\}_{\alpha_{p}\in {\bI}_p^{(0)}}$
is the projection of $\{\psi_{\alpha_{p}}^{(p)}\}_{\alpha_{p}\in {\bI}_p}$
on $L^2(\bZ_p)$.
Now, using Lemma~\ref{l.H(Q)=L^2(A)}, we obtain the following orthonormal wavelet
basis in the space $L^2({\bA},dx)$ (see (\ref{1-ad-0})):
\begin{equation}
\label{1-ad+=11}
\Psi_\alpha=\Psi_{(\widehat{k},\widehat{j},\widehat{a})}(x)
=\psi_{\alpha_{\infty}}^{(\infty)}
\otimes\bigotimes\limits_{q\leq m}\psi_{\alpha_{q}}^{(q)}
\otimes\bigotimes\limits_{m<q}\phi^{(q)},
\quad \alpha\in \Lambda=\bigcup_{m\in V_\bQ}\Lambda_{m}.
\end{equation}
and the orthonormal wavelet basis in the space $L^2({\widetilde\bA},dx)$ (see (\ref{1-ad+=A-00}))
\begin{equation}
\label{1-ad+=11-00}
\widetilde{\Psi}_{\alpha}=\widetilde{\Psi}_{(\widehat{k},\widehat{j},\widehat{a})}(x)
=\bigotimes\limits_{2\le q\leq m}\psi_{\alpha_{q}}^{(q)}
\otimes\bigotimes\limits_{m<q}\phi^{(q)}, \quad \alpha\in \widetilde{\Lambda},
\end{equation}
with the stabilization $\phi=(\phi^{(p)})_{p\in V_\bQ\setminus\infty}$.
Here $\widehat{k}=(k_p)_p$, $\widehat{j}=(j_p)_p$, $\widehat{a}=(a_p)_p$
(see (\ref{basis-phi-inQ})).

\begin{Remark}
\label{basis-Drag}
We note that in~\cite[Sec.4, formula (4.1)]{Dr-2} the basis in the space $L^2(\bA,dx)$ is
constructed as follows
$$
\psi_{\alpha,\beta}=\psi^\infty_{n_0}(x_\infty)\prod_{p=2,3,\dots}\psi_{\alpha_p,\beta_p}(x_p),
$$
where $\psi^\infty_{n_0}(x_\infty)$ and $\psi_{\alpha_p,\beta_p}(x_p)$ are orthonormal
eigenfunctions {\rm(}of harmonic oscillators{\rm)} in the real and $p$-adic cases, respectively, where
$\psi_{\alpha_p,\beta_p}(x_p)=\Omega\big(|x_{p}|_p\big)$  for almost all $p$ .
\end{Remark}
Let us introduce the adelic set of shifts $I_{\bA}$ and dilations $\bZ_{\bA}$:
$$
I_{\bA}=\big\{\widehat{a}=(a_{\infty},a_{2},\dots,a_{p},\dots):
a_{\infty}\in I_{\infty}=\bZ,\, a_{p}\in I_p,
\qquad\qquad\qquad\qquad\quad
$$
\begin{equation}
\label{ad-sh}
\quad \text{there exists} \,\, n \,\, \text{depending on $\widehat{a}$ such that} \,
a_{p}=0 \, \text{for all} \, p> n \big\}.
\end{equation}
$$
\bZ_{\bA}=\big\{\widehat{j}=(j_{\infty},j_{2},\dots,j_{p},\dots):
j_{\infty}\in \bZ,\, j_{p}\in \bZ,
\qquad\qquad\qquad\qquad\qquad\qquad
$$
\begin{equation}
\label{ad-dil}
\quad \text{there exists}\,\, n \,\, \text{depending on $\widehat{j}$ such that} \, j_{p}=0
\, \text{for all} \, p> n \big\}.
\end{equation}
Set $m(\widehat{a})=\min\{p:a_{q}=0 \, \text{for all} \, q>p\}$ and
$m(\widehat{j}\,)=\min\{p:j_{q}=0 \, \text{for all} \, q>p\}$.

On the space $L^2(\bA)$ we define the operators of {\em shifts} $T_{\widehat{a}}$,
$\widehat{a}\in \bZ_{\bA}$, and {\em multi-dilation} $M^{\widehat{j}}$,
$\widehat{j}\in \bZ_{\bA}$ which are defined as the infinite tensor product
of one-dimensional operators (\ref{M^j_VQ}):
\begin{equation}
\label{add-m-d}
M^{\widehat{j}}=M^{-j_{\infty}}_{\infty}\otimes\bigotimes\limits_{2\leq q\le m}
M^{j_q}_{q}\otimes\bigotimes \limits_{q>m}Id_q,
\end{equation}
where $m=m(\widehat{j}\,)$ and $Id_q$ is the identity operator on
$L^2(\bQ_q)$.
We suppose that the operators $M^j_\infty$ and $M^j_p$ in the spaces $L^2(\bQ_\infty)$ and $L^2(\bQ_p)$
(on the functions $f^{(\infty)}\in L^2(\bQ_\infty)$ and $f^{(p)}\in L^2(\bQ_p)$)
act as follows
$$
(M^j_\infty f^{(\infty)})(x_\infty)=2^{-j/2}f^{(\infty)}(2^{-j}x_{\infty}),
\quad
(M^j_p f^{(p)})(x)=p^{-j/2}f^{(p)}(p^{j}x_{p}).
$$

Let
$f(x)=\otimes_{q\in V_{\bQ}}f_{q}(x_{q})\in L^2(\bA)$,
where $f^{(\infty)}\in L^2(\bQ_\infty)$, $f^{(q)}\in L^2(\bQ_q)$, $q \ge 2$,
and $f^{(q)}=\phi^{(q)}$ for almost all  $q > m$. 
Then
\begin{equation}
\label{ad-shift}
\big(T_{\widehat{a}}f\big)(x)\stackrel{def}{=}f(x-\widehat{a})
=f_{\infty}(x_{\infty}-a_{\infty})\otimes\bigotimes\limits_{2\le q}f_{q}(x_{q}-a_{q}),
\qquad\qquad\qquad\qquad
\end{equation}
$$
\big(M^{\widehat{j}}f\big)(x)
\stackrel{def}{=}2^{-j_{\infty}/2}f^{(\infty)}(2^{-j_{\infty}}x_{\infty})\otimes
\qquad\qquad\qquad\qquad\qquad\qquad\qquad
$$
\begin{equation}
\label{ad-del}
\qquad
\otimes\Big(\bigotimes\limits_{2\le q\le m}q^{-j_q/2}f^{(q)}(q^{j_q}x_{q})\Big)\otimes
\Big(\bigotimes\limits_{m<q}f^{(q)}(x_{q})\Big),
\end{equation}
$$
\big(M^{\widehat{j}}T_{\widehat{a}}f\big)(x)
\stackrel{def}{=}2^{-j_{\infty}/2}f^{(\infty)}(2^{-j_{\infty}}x_{\infty}-a_{\infty})\otimes
\qquad\qquad\qquad\qquad\qquad\qquad\quad
$$
\begin{equation}
\label{ad-del-1}
\qquad
\otimes\Big(\bigotimes\limits_{2\le q\le m}q^{-j_q/2}f^{(q)}(q^{j_q}x_{q}-a_{q})\Big)
\otimes\Big(\bigotimes\limits_{m<q}f^{(q)}(x_{q})\Big).
\end{equation}
In the latter relation, we assume that $m=m(\widehat{j}\,)=m(\widehat{a}\,)$.

Now one can obtain the adelic wavelet basis functions $\Psi_\alpha$
given in (\ref{1-ad+=11}) by all {\em adelic shifts and dilations}
of wavelet functions
\begin{equation}
\label{1-ad+=13}
\Psi_{(\widehat{k},0,0)}:=\psi^{(\infty)}_{(0,0)}(x_{\infty})
\otimes \bigotimes\limits_{q\leq m}\psi^{(q)}_{(k_q,\,0\,0)}(x_{q})
\otimes \bigotimes\limits_{m<q}\phi^{(q)}(x_q).
\end{equation}
Namely,
$$
\Psi_{\alpha}(x)=:\Psi_{(\widehat{k},\widehat{j},\widehat{a})}(x)
=\big(M^{\widehat{j}}T_{\widehat{a}}\Psi_{(\widehat{k},0,0)}\big)(x),
$$
where $\widehat{k}=(0,k_{2},\dots,k_{m},0,0,\dots)$, $k_q\in {\bF}_q^\times$;
$\widehat{a}=(a_{\infty},a_{2},\dots,a_{m},0,0,\dots)$, $a_{\infty}\in \bZ$, $a_q\in I_q$;
$\widehat{j}=(j_{\infty},j_{2},\dots,j_{m},0,0,\dots)$, $j_{\infty}\in \bZ$,
$j_q\in \bZ_{+}$; $m\in V_{\bQ}\setminus\infty$. We stress that here $m(\widehat{a})=m(\widehat{j})=m(\widehat{k})=m$.

The wavelet systems (\ref{1-ad+=11}) described above is of
considerable interest and can be useful in various situations.
Nevertheless, they do not possess all the advantages of
one-dimensional wavelet bases, in particular, the localization
property, which is of great value for applications. In the
one-dimensional case, the real Haar basis functions with large
indices $j_{\infty}\in \bZ$ and the $p$-adic Haar basis functions with
large indices $j_p\in \bZ_{+}$, $p=2,3,\dots$, have small supports.
Thus the support of the multidimensional basis function may be large
along one or several directions and small along some other directions.
To avoid these drawbacks, we shall use a different approach. Its main
idea is using the tensor product of the MRAs generating these bases
instead of the tensor product of available wavelet bases.

\section{Separable adelic MRA generated by tensor product of one-dimensional MRAs}
\label{s6}

Using the idea of constructing separable multidimensional MRA by means
of the tensor product of one-dimensional MRAs (suggested by Y.~Meyer~\cite{Meyer-1}
and S.~Mallat~\cite{Mallat-1},~\cite{Mallat-2} (see, e.g.,~\cite[\S 2.1]{NPS}),
we construct wavelet bases for the space $L^2(\bA,dx)$.

We start with some general facts. Let we have a collection of
closed subspaces $V_j^{(k)}$, $j\in\bZ$, in a Hilbert space
$H_k,\,\,k=1,2,\dots,m$, having the properties (a), (b) and (c) of
the Definition \ref{de1-Real}: (a) $V_j^{(k)}\subset
V_{j+1}^{(k)}$ for all $j\in\bZ,\,\,k=1,2,\dots,m$; (b)
$H_k=\overline{\bigcup_{j\in\bZ}V_j^{(k)}}$, (c)
$\bigcap\limits_{j\in\bZ}V_j^{(k)}=\{0\}$.
We use the following notations for the {\em finite tensor product}:
\begin{equation}
\label{61-ad1}
H^{[m]}=\bigotimes_{k=1}^m H_k,\,\, V_{j+1}^{(k)}=V_j^{(k)}\oplus
W_j^{(k)}, \quad j\in \bZ,\quad k=1,2,\dots,m,
\end{equation}
$$
V_{j+1}^{[m]}=\bigotimes_{k=1}^m  V_{j+1}^{(k)}=\bigotimes_{k=1}^m (V_j^{(k)}\oplus W_j^{(k)})
=V_j^{[m]}\oplus W_j^{[m]},
$$
where
$$
V_j^{[m]}=\bigotimes_{k=1}^m V_{j}^{(k)},\qquad
W_j^{[m]}=\bigoplus\limits_{(i_1,i_2,\dots,i_m;j)\in\{1,2\}^m\setminus \{1,1,\dots,1\}}
W_{(i_1,i_2,\dots,i_m;j)},
$$
\begin{equation}
\label{ad-wwww-11}
W_{(i_1,i_2,\dots,i_m;j)}
=Z^{(1)}_{i_1,j}\otimes Z^{(2)}_{i_2,j}\otimes\cdots\otimes Z^{(m)}_{i_m,j},\,\,
\end{equation}
and
\begin{equation}
\label{Z^1_23=}
Z^{(k)}_{1,j}=V^{(k)}_j,\,\, Z^{(k)}_{2,j}=W^{(k)}_j,\,\,k=1,2,\dots,m.
\end{equation}
For the {\em infinite tensor product} we keep the similar notations:
$$
H={\cH}_e=\bigotimes\limits_{e,n\in
{\bN}}H_n,\quad
H^{[m]}=\otimes_{k=1}^m H_k\bigotimes e^{(m+1)}\otimes  e^{(m+2)}\otimes\cdots ,\quad m\in \bN.
$$
$$
V_{j+1}^{[m]}=\bigotimes_{k=1}^m  V_{j+1}^{(k)}\otimes e^{(m+1)}\otimes  e^{(m+2)}\otimes\cdots
\qquad\qquad\qquad\qquad\qquad\qquad\qquad\qquad
$$
$$
\qquad\qquad
=\bigotimes_{k=1}^m (V_j^{(k)}\oplus W_j^{(k)})\otimes e^{(m+1)}\otimes  e^{(m+2)}\otimes\cdots
=V_j^{[m]}\oplus W_j^{[m]},
$$
where
\begin{equation}
\label{V_j^[m]}
V_j^{[m]}=\bigotimes_{k=1}^m V_{j}^{(k)}\otimes e^{(m+1)}\otimes e^{(m+2)}\otimes\cdots,
\end{equation}
$$
W_j^{[m]}=\bigoplus\limits_{(i_1,i_2,\dots,i_m)\in\{1,2\}^m\setminus \{1,1,\dots,1\}}
W_{(i_1,i_2,\dots,i_m;j)}
$$
\begin{equation}
\label{aa}
W_{(i_1,i_2,\dots,i_m;j)}
=Z^{(1)}_{i_1,j}\otimes Z^{(2)}_{i_2,j}\otimes\cdots\otimes Z^{(m)}_{i_m,j}
\otimes e^{(m+1)}\otimes  e^{(m+2)}\otimes\cdots.
\end{equation}

Define the space $V_j$ as an appropriate limit of the spaces $V_j^{[m]}$
\begin{equation}
\label{V^j=limV^{m}}
V_{j}=\overline{{\rm span}(V_j^{[m]}: m\in\bN)}.
\end{equation}
\begin{Theorem}
\label{th0-ad-mra}
Let subspaces $\{V_j^{(n)}\}_{j\in\bZ}$ of the space $H_n,\,\,n\in \bN$
satisfy the following properties:

{\rm(a)} $V_j^{(n)}\subset V_{j+1}^{(n)}$ for all $j\in\bZ$ and $n\in \bN $;

{\rm(b)} $\bigcup_{j\in\bZ}V_j^{(n)}$ is dense in $H_n$;

{\rm(c)} $\bigcap_{j\in\bZ}V_j^{(n)}=\{0\}$ in $H_n$.

Then the subspaces $V_j,j\in \bZ$ of the space ${\cH}_e=\otimes_{e,n\in {\bN}}H_n$
defined by {\rm(\ref{V^j=limV^{m}})} satisfy the following properties:

{\rm(aH)} $V_j\subset V_{j+1}$ for all $j\in\bZ$;

{\rm(bH)} $\bigcup_{j\in\bZ}V_j$ is dense in ${\cH}_e=\otimes_{e,n\in {\bN}}H_n$;

{\rm(cH)} $\bigcap_{j\in\bZ}V_j=\{0\}$.
\end{Theorem}

\begin{proof}
{\rm(aH)} Since $V_j^{(n)}\subset V_{j+1}^{(n)}$ for all $n \in \bN$,
we conclude that $V_j^{[m]}\subset  V_{j+1}^{[m]},\,\,m\in\bN$ (see
(\ref{V_j^[m]})), hence $V_{j}=\overline{{\rm span}(V_j^{[m]}:
m\in\bN)}\subset \overline{{\rm span}(V_{j+1}^{[m]}: m\in\bN)}=
V_{j+1}$.

{\rm(bH)} It is clear that for any fixed $m\in \bN$ the space $\bigcup_{j\in\bZ}V_j^{[m]}$
is dense in the space
$$
H^{[m]}=\bigotimes_{k=1}^m H_k\otimes e^{(m+1)}\otimes  e^{(m+2)}\otimes\cdots
$$
hence all vectors $e_\alpha,\,\,\alpha \in  \Lambda_m$ of the basis (\ref{bas-1})
are contained in $\overline{\cup_{j\in\bZ}V_j^{[m]}}$. To finish the proof,
we note that all elements of the basis
$e_\alpha, \alpha\in \Lambda$ in ${\cH}_e=\otimes_{e,n\in {\bN}}H_n$
are given by (\ref{bas-1}), where $\Lambda=\cup_{m\in \bN}\Lambda_m$.

{\rm(cH)} 1. Let us suppose that for any $m\in\bN$, we have
\begin{equation}
\label{cap-V-j-[m]}
\cap_{j\in\bZ}V_j^{[m]}=\{0\}.
\end{equation}
2. In this case we get
\begin{equation}
\label{cap-nV-n(H)=0}
\cap_{j\in\bZ}V_j=\cap_{j\in\bZ}\overline{{\rm span}(V_j^{[m]}: m\in\bN)}\subset
\overline{{\rm span}(\cap_{j\in\bZ}V_j^{[m]}: m\in\bN)}=\{0\}.
\end{equation}

1. To prove (\ref{cap-V-j-[m]}), we shall use (c). We set $W_j^{(n)}=V_{j+1}^{(n)}\ominus V_j^{(n)}$
for $j\in\bZ$ and $n\in \bN$. Then we get
\begin{equation}
\label{W-k(H)+V-k(H)}
H_n=\oplus_{k\in\bZ}W_k^{(n)}, \quad\text{and}\quad
V_j^{(n)}=\oplus_{k=-\infty}^{j-1}W_k^{(n)}.
\end{equation}
Let $(e^{(n)}_{\alpha_n})_{\alpha_n\in I_k^{(n)}}$ be an orthonormal
basis in $W_k$, then by construction,
$(e^{(n)}_{\alpha_n})_{\alpha_n\in I^{(n)}}$ is an orthonormal basis
in $H_n$, where $I^{(n)}=\bigcup_{k\in\bZ}I_k^{(n)}$. Define
$\Lambda(I)$ (see Definition (\ref{d.Lambda})) as the set of
multi-indices $\alpha=(\alpha_n)_{n\in \bN},$ $\alpha_n\in I^{(n)}$
such that $e_{\alpha_n}^{(n)}=e^{(n)}$ for sufficiently big $n$ depending
on $\alpha$, here $e=(e^{(n)})_{n\in\bN}$ is a stabilizing sequence. Using
(\ref{W-k(H)+V-k(H)}) and (\ref{V_j^[m]}), we get
$$
V_j^{[m]}=\Big(\bigotimes_{k=1}^m V_{j}^{(k)}\Big)\otimes e^{(m+1)}\otimes e^{(m+2)}\otimes\cdots
\qquad\qquad\qquad\qquad\qquad
$$
$$
\qquad\qquad\qquad\qquad
=\bigotimes_{k=1}^m \Big(\bigoplus_{n=-\infty}^{j-1}W_n^{(k)}\Big)
\otimes e^{(m+1)}\otimes e^{(m+2)}\otimes\cdots,
$$
so the basis in the space $V_j^{[m]}$ can be chosen as
$e_\alpha,\,\,\alpha\in \Lambda_{m,j}(I)$, where
$$
e_\alpha=e^{(1)}_{\alpha_1}\otimes e^{(2)}_{\alpha_2}
\otimes\cdots\otimes e^{(m)}_{\alpha_m}\otimes
e^{(m+1)}\otimes\cdots,
$$
$$
\Lambda_{m,j}(I)=\Big(
\alpha\in\Lambda(I):\alpha=(\alpha_n)_{n\in\bN} ,\quad
(\alpha_1,\dots,\alpha_m)\in(\bigcup_{k=-\infty}^{j-1}I_k^{(1)},
\dots,\bigcup_{k=-\infty}^{j-1}I_k^{(m)})\Big).
$$
Note that $\bigcap_{j\in \bZ}\Lambda_{m,j}(I)
=\lim_{j\to-\infty}\big(\bigcup_{k=-\infty}^{j-1}I_k^{(1)},
\dots,\bigcup_{k=-\infty}^{j-1}I_k^{(m)}\big)=\varnothing$,
this imply (\ref{cap-V-j-[m]}) (compare also with Lemma \ref{bigcap(V-j)=0}
below, where $\Lambda_{1,j}(I)=(-\infty,j-1]\cap\bZ$). Indeed, for
$f\in V_j^{[m]}$ we have
$$
f=\sum_{\alpha\in
\Lambda_{m,j}(I)}c_\alpha e_\alpha\quad \text{ with} \quad\Vert
f\Vert^2=\sum_{\alpha\in \Lambda_{m,j}(I)}\vert
c_\alpha\vert^2<\infty.
$$
If $f\in \bigcap_{j\in \bZ}V_j^{[m]}$ then
$$
\Vert f\Vert^2=\lim_{j\to-\infty}\sum_{\alpha\in
\Lambda_{m,j}(I)}\vert c_\alpha\vert^2=\sum_{\alpha\in
\varnothing}\vert c_\alpha\vert^2=0.$$

2. To prove (\ref{cap-nV-n(H)=0}), using (\ref{V_j^[m]}) we get
\begin{align*}
V_j^{[m]}&=\otimes_{k=1}^m V_{j}^{(k)}\otimes e^{(m+1)}\otimes e^{(m+2)}\otimes\cdots,\quad
m\in\bN,\\
V_j^{[1]}&=V_{j}^{(1)}\otimes e^{(2)}\otimes\cdots\otimes e^{(m)}\otimes e^{(m+1)}\otimes
\cdots,\\
V_j^{[2]}&=V_{j}^{(1)}\otimes V_{j}^{(2)}\otimes\cdots\otimes e^{(m)}\otimes e^{(m+1)}\otimes
\cdots,\\
V_j^{[m]}&=V_{j}^{(1)}\otimes V_{j}^{(2)}\otimes\cdots \otimes V_{j}^{(m)}\otimes e^{(m+1)}\otimes
\cdots.
\end{align*}
Denote by $V_{j}^{(k)}(e)={\rm span}(e^{(k)},V_{j}^{(k)})$ for $k-1\in \bN$, then we get
\begin{align*}
{\rm span}(V_j^{[1]},V_j^{[2]})&=V_{j}^{(1)}\otimes V_{j}^{(2)}(e)\otimes e^{(3)}\otimes\cdots
\otimes e^{(m)}\otimes e^{(m+1)}\otimes \cdots,\\
{\rm span}(V_j^{[1]},V_j^{[3]},V_j^{[2]})&=V_{j}^{(1)}\otimes V_{j}^{(2)}(e)
\otimes V_{j}^{(3)}(e)\otimes\cdots\otimes e^{(m)}\otimes e^{(m+1)}\otimes \cdots,\\
{\rm span}(V_j^{[k]},k=2,\dots,m)&=
V_{j}^{(1)}\otimes V_{j}^{(2)}(e)\otimes V_{j}^{(3)}(e)\otimes\cdots\otimes V_{j}^{(m)}(e)
\otimes e^{(m+1)}\otimes \cdots.
\end{align*}
Since $V_{j}^{(n)}(e)={\rm span}(e^{(n)},V_{j}^{(n)})\subset V_{0}^{(n)}$ for $j\le 0$, we 
conclude that
$$
V_j=\overline{{\rm span}(V_j^{[k]}:k\in\bN)}=
\overline{V_{j}^{(1)}\otimes\bigotimes\limits_{e,n\geq 2}V_{j}^{(n)}(e)}=
V_{j}^{(1)}\otimes\bigotimes\limits_{e,n\geq 2}V_{j}^{(n)}(e)\subset
V_{j}^{(1)}\bigotimes\limits_{e,n\geq 2}V_{0}^{(n)}.
$$
So
$$
\bigcap_{j\in \bZ}V_j=\bigcap_{j\in \bZ}\overline{{\rm span}(V_j^{[k]}:k\in\bN)}\subset
\qquad\qquad\qquad\qquad\qquad\qquad\qquad\qquad
$$
$$
\subset
\bigcap_{j\in \bZ}\big(V_{j}^{(1)}\otimes\bigotimes\limits_{e,n\geq 2}V_{0}^{(n)}\big)
=\bigcap_{j\in \bZ}\big(V_{j}^{(1)}\big)\otimes\bigotimes\limits_{e,n\geq 2}V_{0}^{(n)}=\{0\}.
$$
\end{proof}

\begin{Remark}\rm
\label{r.V(n)} We can give another definition of the spaces
$V_j,\,\,j\in\bZ$, namely, for fixed $n\in \bN$ denote by
$V_j(n),\,\,j\in\bZ$ the following family of spaces
\begin{equation}
\label{V(n)} 
V_{j}(n)=\overline{{\rm span}(V_j^{[m]}: m\in\bN,\,\,m\geq n)}.
\end{equation}
For this system of spaces, an analog of Theorem~\ref{th0-ad-mra}
holds. Moreover, in the proof of (cH) we have the following formula
{\rm(}see the end of the proof of the Theorem~\ref{th0-ad-mra}):
$$
\bigcap_{j\in \bZ}V_j(n)\subset \bigcap_{j\in
\bZ}\Big(\bigotimes\limits_{k=1}^nV_{j}^{(k)}\Big)\otimes\bigotimes\limits_{e,k > n}V_{0}^{(k)}=\{0\}.
$$
\end{Remark}

We would like to present the following elementary lemma.
\begin{Lemma}
\label{bigcap(V-j)=0}
Let we have the decomposition of the Hilbert space $H=\oplus_{k\in\bZ}W_k$
and let $(e_\alpha)_{\alpha\in I_k}$ be an orthonormal basis in the space $W_k$.
Set $V_j:=\oplus_{k=-\infty}^{j-1}W_k$, then $\bigcap_{j\in\bZ}V_j=\{0\}$.
\end{Lemma}

\begin{proof}
By construction, $(e_\alpha)_{\alpha\in J}$ is an orthonormal basis in $H$,
where $J=\bigcup_{k\in\bZ}J_k$. We have for any $f\in H$
\begin{equation}
\label{norma(f)}
f=\sum_{\alpha\in J}c_\alpha e_\alpha =
\sum_{k\in\bZ}\left(\sum_{\alpha\in J_k}
c_\alpha e_\alpha\right)\,\, \text{and}\,\,
\Vert f\Vert^2=\sum_{k\in\bZ}\left(\sum_{\alpha\in J_k}
c_\alpha^2 \right)<\infty.
\end{equation}
If $f\in V_{n}=\oplus_{k=-\infty}^{n-1}W_k$ for all $n\in \bZ$, using
(\ref{norma(f)}), we conclude that
$$
\Vert f\Vert^2=\lim_{n\to-\infty}\sum_{k=-\infty}^{n-1}
\left(\sum_{\alpha\in J_k}\vert c_\alpha\vert^2 \right)=0.
$$
\end{proof}

Let $\{V_j^{(\nu)}\}_{j\in\bZ}$ be one-dimensional MRA {\rm(}see Definitions~{\rm\ref{de1-Real},~\ref{de1}}{\rm)},
and let $\phi^{(\nu)}$ be its refinable function, $\nu \in V_\bQ$. We introduce the spaces (see (\ref{V_j^[m]}))
\begin{equation}
\label{d33-1}
V_j^{[m]}=\bigotimes_{\nu\in \{\infty,2,3,\dots,m\}}V_{j}^{(\nu)}\otimes\phi^{(m_{+})}\otimes\cdots
\subset L^2(\bA,dx), \quad m\in V_{\bQ}, \quad j\in \bZ.
\end{equation}

Now we introduce in $L^2(\bA,dx)$ the delation operator $M^j$, which
is defined by its projection $M^j_{(m)}$ on any subspace $V_j^{[m]}$
(for details, see Remark~\ref{A=lim-m-A-(m)}):
\begin{equation}
\label{ad-sh-dd}
M^j_{(m)}=M^{-j}_\infty\otimes\bigotimes\limits_{2\le q\leq m }M^{j}_{q}
\otimes\bigotimes\limits_{q>m}Id_q, \quad m\in V_{\bQ}\setminus\infty,
\end{equation}
where $Id_q$ is the identity operator in $L^2(\bQ_q)$.
For $f\in L^2(\bA)$ we define $M^{j}_{(m)}f$ and $(M^{-j}_{(m)}T_{\widehat{a}})f$
similarly to (\ref{ad-del}) and (\ref{ad-del-1}), respectively, where $I_{\bA}$ is given by (\ref{ad-sh}).
As before, we assume that $m=m(\widehat{a}\,)$.

\begin{Theorem}
\label{th1-ad-mra}
Let $\{V_j^{(\nu)}\}_{j\in\bZ}$ be the one-dimensional MRA {\rm(}see Definitions~{\rm\ref{de1-Real},~\ref{de1}}{\rm)},
and let $\phi^{(\nu)}$ be its refinable function, $\nu \in V_\bQ$. Let
\begin{equation}
\label{d3-1}
\Phi(x)=\phi^{(\infty)}(x_{\infty})\otimes\phi^{(2)}(x_{2})\otimes\cdots\otimes\phi^{(p)}
(x_{p})\otimes\cdots, \quad x\in \bA,
\end{equation}
and
$$
V_j=\overline{{\rm span}(V_j^{[m]}: m\in V_{\bQ})}
\qquad\qquad\qquad\qquad\qquad\qquad\qquad\qquad
$$
\begin{equation}
\label{d3-1-11}
=\overline{{\rm span}
\{\big(M^{-j}_{(m)}T_{\widehat{a}}\Phi\big)(\cdot):\, \widehat{a}\in
I_{\bA}, \, m(\widehat{a})=m\in V_{\bQ}\}}, \quad j\in\bZ.
\end{equation}
Then
\begin{equation}
\label{d3-1-12}
V_j=\bigotimes\limits_{e;\nu \in V_\bQ} V^{(\nu)}_j \subset L^2(\bA,dx),
\quad j\in\bZ_{+},
\end{equation}
{\rm(}where $e=(\phi^{(\nu)})_{\nu\in V_{\bQ}\setminus {\infty}}$ is the
{\em stabilization sequence} {\rm(\ref{stab-phi})} and the
subspaces~{\rm(\ref{d3-1-11})} satisfy the following properties:

{\rm(a)} $V_j\subset V_{j+1}$ for all $j\in\bZ$;

{\rm(b)} $\cup_{j\in\bZ}V_j$ is dense in $L^2(\bA,dx)$;

{\rm(c)} $\cap_{j\in\bZ}V_j=\{0\}$;

{\rm(d)} $f(\cdot)\in V_j^{[m]} \Longleftrightarrow
\big(M^{-1}_{(m)}f\big)(\cdot)\in V_{j+1}^{[m]}$
for all $j\in\bZ$, for any $m\in V_{\bQ}$;

{\rm(e)} the function $\Phi\in V_0$ is such that the system
$\{\Phi(x-\widehat{a}): \, \widehat{a}\in I_{\bA}\}$ is an orthonormal basis for $V_0$.
\end{Theorem}

\begin{proof}
1. For each $\nu\in V_\bQ$, the system of functions $\{\phi^{(\nu)}(\cdot-a_{\nu})\}_{a_{\nu}\in I_{\nu}}$
is an orthonormal basis of the space $V^{(\nu)}_0$ (see axioms $(e)$ in
Definitions~\ref{de1-Real},~\ref{de1}). Using the construction of
Sec.~\ref{s4} and taking into account that any shift
$\widehat{a}=(a_{\infty},a_{2},\dots,a_{p},\dots)\in I_{\bA}$ is
{\em finite} and $\phi^{(\nu)}\in V^{(\nu)}_0\subset L^2(\bQ_\nu,dx_\nu)$
is a {\it stabilization sequence} (\ref{stab-phi}) in $L^2({\bA},dx)$,
we conclude that the system $\{\Phi(x-\widehat{a}): \,
\widehat{a}\in I_{\bA}\}$ is an orthonormal basis in $V_0$ and
$$
V_0=\bigotimes\limits_{e;\,\nu \in V_\bQ}V^{(\nu)}_0 \subset L^2(\bA,dx).
$$
Using definition (\ref{ad-del-1}) we get for $m(\widehat{a})=m$.
$$
\big(M^{-j}_{(m)}T_{\widehat{a}}\Phi\big)(x)
=2^{j/2}\phi^{(\infty)}(2^{j}x_{\infty}-a_{\infty})\otimes
\qquad\qquad\qquad\qquad\qquad\qquad\qquad\qquad\qquad
$$
\begin{equation}
\label{otimes-mm}
\otimes
\Big(\bigotimes\limits_{\nu=2}^{m}\nu^{j/2}\phi^{(\nu)}(\nu^{-j}x_{\nu}-a_{\nu})\Big)
\otimes\phi^{(m_{+})}(x_{m_{+}})\otimes\cdots \in V_j^{[m]}\subset V_j
\end{equation}
for any $\widehat{a}\in I_{\bA}$, $m\in V_{\bQ}$.

Since $\{\nu^{j/2}\phi^{(\nu)}(\nu^{-j}x_{\nu}-a_{\nu}):a_{\nu}\in
I_\nu\}$ is the basis in $V_j^{(\nu)}$ and $\phi^{(\nu)}\in
V^{(\nu)}_0\subset V_j^{(\nu)}$ for $j\in \bN$, using
(\ref{d3-1-11}) and (\ref{otimes-mm}), we conclude that
$$
V_j=\bigotimes\limits_{e;\,\nu \in V_\bQ}V^{(\nu)}_j \subset L^2(\bA,dx), \quad j\in \bN.
$$

Taking into account axioms $(d)$ in Definitions~\ref{de1-Real},~\ref{de1},
one can see that Definitions (\ref{d3-1-11}), (\ref{ad-sh-dd}) imply that
$f\in V_0^{[m]}$ if and only if $M^{-j}_{(m)}f(\,\cdot)\in V_j^{[m]}$, $j\in \bN$, $m\in V_{\bQ}$.

Thus representation (\ref{d3-1-12}) and properties $(e)$ and $(d)$ hold.

2. Due to~\cite[Ch.~I, \S 2.2, Theorem~2.2]{Ber86}, we immediately obtain that
property $(a)$ holds for $j\in \bN$.
For any $j\in \bZ$ the property $(a)$ holds by Theorem \ref{th0-ad-mra}.
Here we give an independent proof.
According to (\ref{d3-1-11}),
$V_0=\overline{{\rm span}\{\Phi\big(x-\widehat{a}\big):\, \widehat{a}\in I_{\bA}\}}$,
where $\Phi\big(x-\widehat{a}\big)=\phi^{(\infty)}(x_{\infty}-a_{\infty})\otimes\phi^{(2)}(x_{2}-a_{2})
\otimes\cdots\otimes\phi^{(p)}(x_{p}-a_{p})\otimes\phi^{(p_{+})}(x_{p_{+}})\otimes\cdots$.
Applying the {\em refinement equation} (\ref{m=62.0-3}) to any factor of the above product and
taking into account that the numbers $\frac{a_{\nu}}{\nu}, \frac{a_{\nu}}{\nu}+\frac{r}{\nu}\in I_{\nu}$
for all $a_{\nu}\in I_{\nu}$ and $r=1,2,\dots,\nu-1$, \, $\nu=2,3,\dots$, we conclude
that $V_0\subset V_1$. By Definition (\ref{d3-1-11}) of the spaces $V_j$ and property
$(d)$, this yields property $(a)$.

3. The property $(b)$ holds by Theorem \ref{th0-ad-mra}. Here we give an independent
proof. It is clear that any element $g\in L^2(\bA,dx)$ can be approximated
by a finite linear combination of the basis elements $\Psi_{\alpha}(x)=\Psi_{(\widehat{k};
\widehat{j}\ \widehat{a})}(x)$ of $L^2(\bA,dx)$ given by (\ref{1-ad+=11}). At the same
time any element $\Psi_{\alpha}(x)=\Psi_{(\widehat{k}; \widehat{j}\ \widehat{a})}(x)$
is a finite product of the one-dimensional basis elements
(\ref{real-basis})--(\ref{basis-111}):
\begin{align*}
\psi_{\alpha_{\infty}}^{(\infty)}(x_{\infty})&
=\psi_{j_{\infty}a_{\infty}}^{(\infty)}(x_{\infty})\in L^2(\bQ_{\infty}),&\alpha_{\infty}\in (\bZ,\bZ), \\
\psi_{\alpha_{q}}^{(q)}(x_{q})&=\psi^{(q)}_{(k_q; j_q a_q)}(x_{q})
\in L^2(\bQ_{q}),&\alpha_{p}\in {\bI}_p=({\bF}_p^\times,\bZ_+,I_p)\bigcup I_p, \\
\psi_{\alpha_{q}}^{(q)}(x_{q})&=\psi^{(q)}_{(k_q; j_q a_q)}(x_{q})
\in L^2(\bZ_{q}),&\alpha_{p}\in {\bI}_p^{(0)}=({\bF}_p^\times,\bZ_+,I_p)\bigcup\{0\}.
\end{align*}
In turn, it follows from the completeness property for each MRA
$\{V_j^{(\nu)}\}_{j\in\bZ}$ (see axioms (b) in Definitions~\ref{de1-Real},~\ref{de1})
that any element $f_{\infty}\in L^2(\bQ_{\infty})$ and $f_{\nu}\in L^2(\bQ_{\nu})$
can be approximated by a finite linear combination of the functions
$2^{j/2}\phi^{(\infty)}(2^j\cdot-a_{\infty})$, $a_{\infty}\in I_{\infty}$, and
$\nu^{j/2}\phi^{(\nu)}(\nu^{-j}\cdot-a_{\nu})$,
$a_{\nu}\in I_{\nu}$, respectively, for sufficiently large $j$ ($j\in \bZ_{+}$),
$\nu\in\{2,3,\dots,p\}$. Hence any basis element
$\Psi_{\alpha}(x)=\Psi_{(\widehat{k}; \widehat{j}\ \widehat{a})}(x)\in L^2(\bA,dx)$
can be approximated by an element of the space $V_j$ for some
sufficiently large $j$ ($j\in \bZ_{+}$).
This implies that the element $g$ can be also approximated by an element
of the space $V_j$ for some $j$, i.e., the collection of all the
spaces $V_j$, $j\in\bZ$, is dense in $L^2(\bA,dx)$.

4. The property $(c)$ holds by Theorem \ref{th0-ad-mra}. Here we give an independent
proof. Let us prove that the intersection of all $V_j$ does
not contain a nonzero element. Assume that there exists an element
$g\in L^2(\bA,dx)$ such that $g\in V_j$ for all $j\in\bZ$ and
$\|g\|\ne0$. According to Sec.~\ref{s5}, the wavelet functions
$\Psi_{\alpha}=\Psi_{(\widehat{k}; \widehat{j}\ \widehat{a})}\in
L^2(\bA,dx)$ given by (\ref{1-ad+=11}) form an orthonormal basis in
$L^2(\bA,dx)$, then $g$ can be represented as $g=\sum_{\alpha\in
\Lambda}c_{\alpha}\Psi_\alpha\in V_j$ for all $j\in\bZ$, where
the coefficients $c_{\alpha}$ are independent of $j$. Here
$\alpha=(\alpha_p)_{p\in V_{\bQ}}=(\widehat{k}; \widehat{j}\
\widehat{a})$, $\alpha_{\infty}\in (\bZ,\bZ)$,
$\alpha_{p}\in {\bI}_p=({\bF}_p^\times,\bZ_+,I_p)\bigcup I_p$, $p=2,3,\dots$.

Consider a basis element $\Psi_{\alpha}=\Psi_{(\widehat{k};
\widehat{j}\ \widehat{a})}$,
$\widehat{j}=(j_{\infty},j_{2},\dots,j_{m},0,0,\dots)$. Let us
take the index $j_\alpha={\rm min}(0,j_{\infty},j_{2},j_{3},\dots,j_{m})$.
For any wavelet function given by the first formula in (\ref{basis-phi-inQ}), we
have $\psi^{(p)}_{\alpha_p}=\psi^{(p)}_{k_p; j_p a_p}\perp
V_0^{(p)}$, $j_p\in \bZ_+$, and $V_j^{(p)}\subset V_0^{(p)}$,
$p=2,3,\dots$. In view of (\ref{61-ad-Haar}), (\ref{61.1-ad-Haar})
we have
$\psi_{\alpha_{\infty}}^{(\infty)}=\psi_{j_{\infty}}^{(\infty)}\perp V_j^{(\infty)}$
for $j< j_{\infty}$. Consequently, the basis element
$\Psi_\alpha=\Psi_{(\widehat{k}; \widehat{j}\ \widehat{a})}\perp V_j$,
$j<j_\alpha$, i.e., $g \perp \Psi_\alpha$. In such a way we prove that
$g \perp \Psi_\alpha$ for all indices $\alpha\in \Lambda$. The fact
that the system $\{\Psi_\alpha:\alpha\in \Lambda\}$ is an orthonormal basis in
$L^2(\bA,dx)$ contradicts the assumption $\|g\|\ne0$.
\end{proof}

The {\em collection of the spaces} $V_j$, $j\in\bZ_{+}$, satisfying conditions $(a)$--$(e)$
of Theorem~\ref{th1-ad-mra} will be called the {\em adelic separable MRA in the space}
$L^2(\bA)$. The function (\ref{d3-1}) is a {\em refinable function} of this MRA.

Next, following the standard finite-dimensional scheme (see
~\cite[Ch.~10.1]{Daub}, \cite[\S 2.1]{NPS}), we define
the wavelet space $W_j$ as the orthogonal complement of $V_j$ in $V_{j+1}$:
\begin{equation}
\label{ad-wwww}
W_j=V_{j+1}\ominus V_j, \quad j\in\bZ_{+}.
\end{equation}
It follows from properties $(a)$ and $(b)$ of Theorem~\ref{th1-ad-mra} that
(cf. (\ref{61.1-ad-Haar}) and (\ref{61.1-ad}))
$$
L^2(\bA)=V_0\oplus\bigoplus\limits_{j\in\bZ_{+}}W_j.
$$

Let us present some evident formulas in $\otimes_{k=1}^nH_k$. Namely, we have
\begin{equation}
\label{otimes*oplus}
(a_{11}\oplus a_{12})\otimes\cdots
\otimes(a_{n1}\oplus a_{n2})=\bigoplus\limits_{(i_1,\dots,i_n)\in\{1,2\}^n}a_{1i_1}
\otimes\cdots\otimes a_{ni_n},
\end{equation}
and similarly,
\begin{equation}
\label{otimes*oplus-m+}
(\oplus _{k=1}^m a_{1k})\otimes\cdots\otimes(\oplus _{k=1}^m a_{nk})
=\bigoplus\limits_{(i_1,\dots,i_n)\in\{1,\dots,m\}^n}a_{1,i_1}\otimes\cdots\otimes a_{n,i_n},
\end{equation}

Since according to (\ref{61-ad-Haar}), (\ref{61.1-ad-Haar}), (\ref{61-ad})--(\ref{61.1-ad}),
we have $V^{(\nu)}_{j+1}=V^{(\nu)}_j\oplus W^{(\nu)}_j$, setting
$Z^{(\nu)}_{1,j}=V^{(\nu)}_j$ and $Z^{(\nu)}_{2,j}=W^{(\nu)}_j$,
using (\ref{otimes*oplus}), (\ref{otimes*oplus-m+}), and taking into account
(\ref{otimes-mm}), we obtain
$$
V_{j+1}=\bigotimes\limits_{\nu\in V_\bQ}V_{j+1}^{(\nu)}
=\bigotimes\limits_{\nu\in V_\bQ}\big(V^{(\nu)}_j\oplus W^{(\nu)}_j\big)
\qquad\qquad\qquad\qquad\qquad\qquad\qquad\qquad
$$
$$
\qquad
=V_j\oplus
\Big(\bigoplus\limits_{p\in V_\bQ} \,
\bigoplus\limits_{(i_\infty,i_2,\dots,i_p)\in\{1,2\}^{\sharp(p)+1}\setminus \{1,1,\dots,1\}}
W_{(i_\infty,i_2,\dots,i_p),j}\Big), \quad j\in\bZ_{+},
$$
where we set $\sharp(p)=\sharp\{2,3,\dots,p\}$ ($\sharp(A)$ is the number of elements of the set $A$), and
\begin{equation}
\label{ad-wwww-1}
W_{\widehat{i}_p,j}^{[p]}=W_{(i_\infty,i_2,\dots,i_p),j}^{[p]}
=Z^{(\infty)}_{i_\infty,j}\otimes Z^{(2)}_{i_2,j}\otimes\cdots\otimes Z^{(p)}_{i_p,j}
\otimes\phi^{(p_{+})}\otimes\cdots,
\end{equation}
$\widehat{i}_p=(i_\infty,i_2,\dots,i_p)\in\{1,2\}^{\sharp(p)+1}\setminus \{1,1,\dots,1\}$.
Thus the space $W_j$ given by (\ref{ad-wwww}) is a direct sum of infinite number of subspaces
(\ref{ad-wwww-1}).

Thus we have
\begin{equation}
\label{w-62.8=10**-ad-ss}
L^2(\bA)=V_0\oplus
\bigoplus\limits_{j\in\bZ_{+}}W_j=\bigoplus\limits_{j\in\bZ_{+}}
\Big(\bigoplus\limits_{p\in V_\bQ} \bigoplus\limits_{
\widehat{i}_p\in\{1,2\}^{\sharp(p)+1}\setminus \{1,1,\dots,1\}}
W_{\widehat{i}_p,j}^{[p]}\Big).
\end{equation}

It is clear that an orthonormal basis for the space
$$
W_{0}=\bigoplus\limits_{p\in V_\bQ}
\bigoplus\limits_{\widehat{i}_p=(i_\infty,i_2,\dots,i_p)\in\{1,2\}^{\sharp(p)+1}\setminus \{1,1,\dots,1\}}
W_{(i_\infty,i_2,\dots,i_p),0}^{[p]}
$$
where
$$
W_{\widehat{i}_p,0}^{[p]}=W_{(i_\infty,i_2,\dots,i_p),0}^{[p]}
=Z^{(\infty)}_{i_\infty,0}\otimes Z^{(2)}_{i_2,0}\otimes\cdots\otimes Z^{(p)}_{i_p,0}
\otimes\phi^{(p_{+})}\otimes\cdots,
$$
is formed by the shifts (with respect to $\widehat{a}\in I_{\bA}$) of the functions
\begin{equation}
\label{m=wd-62-ad}
\Psi_{\widehat{k},\widehat{i}_p}:=
\Psi_{\widehat{k},(i_\infty,i_2,\dots,i_p)}
=\vartheta^{(\infty)}_{i_\infty}\otimes
\vartheta^{(2)}_{i_2}\otimes\cdots\otimes \vartheta^{(p)}_{i_p}
\otimes\phi^{(p_{+})}\otimes\cdots,
\end{equation}
where all factors in this product starting from a certain prime $p$
constitute a {\it stabilization sequence} $(\phi^{(\nu)})_{\nu> p}$; \,
$\vartheta^{(\infty)}_{1}=\psi^H$, $\vartheta^{(\infty)}_{2}=\phi^H$ (see (\ref{m=1-h}), (\ref{m=3-h})); \,
$\vartheta^{(\nu)}_{1}=\psi^{(\nu)}_{k_{\nu}}$, $k_{\nu}\in \bF_{\nu}^{\times}$,
\, $\vartheta^{(\nu)}_{2}=\phi^{(\nu)}$, (see (\ref{62.0-1=111}), (\ref{62.0-1=zz-11})
or (\ref{62.0-1=111-K}), (\ref{62.0-1=zz})), $\nu=2,3,\dots,p$; $p=2,3,\dots$;
$\widehat{i_p}=(i_\infty,i_2,\dots,i_p)\in\{1,2\}^{\sharp(p)+1}\setminus \{1,1,\dots,1\}$.

According to (\ref{w-62.8=10**-ad-ss}), an orthonormal Haar wavelet basis in $L^2(\bA)$
is formed by all shifts $\widehat{a}\in I_{\bA}$ of the refinable function (\ref{d3-1})
$$
\phi^{(\infty)}(x_{\infty})\phi^{(2)}(x_{2})\cdots\phi^{(p_-)}(x_{p_-})\phi^{(p)}(x_{p})\phi^{(p_+)}(x_{p_+})\cdots,
$$
and all shifts $\widehat{a}\in I_{\bA}$ and dilations $j\in\bZ_{+}$ of wavelet functions (\ref{m=wd-62-ad}):
\begin{equation}
\label{w-62.8=10**-ad}
\begin{array}{lll}
\displaystyle
\psi^{(\infty)}(x_{\infty})\psi_{k_2}^{(2)}(x_2)\cdots\psi_{k_{p_{-}}}^{(p_{-})}(x_{p_{-}})\psi_{k_p}^{(p)}(x_p)
\phi^{(p_{+})}(x_{p_{+}})\cdots, && \\
\displaystyle
\psi^{(\infty)}(x_{\infty})\psi_{k_2}^{(2)}(x_2)\cdots\psi_{k_{p_{-}}}^{(p_{-})}(x_{p_{-}})\phi^{(p)}(x_p)
\phi_{k_{p_{+}}}^{(p_{+},0)}(x_{p_{+}})\cdots, && \\
\hdotsfor{1} \\
\displaystyle
\phi^{(\infty)}(x_{\infty})\psi_{k_2}^{(2)}(x_2)\cdots\psi_{k_{p_{-}}}^{(p_{-})}(x_{p_{-}})\psi_{k_p}^{(p)}(x_p)
\psi_{k_{p_{+}}}^{(p_{+},0)}(x_{p_{+}})\cdots, && \\
\hdotsfor{1} \\
\displaystyle
\psi^{(\infty)}(x_{\infty})\phi^{(2)}(x_2)\cdots\phi^{(p_{-})}(x_{p_{-}})\phi^{(p)}(x_p)
\phi^{(p_{+})}(x_{p_{+}})\cdots, && \\
\hdotsfor{1} \\
\displaystyle
\phi^{(\infty)}(x_{\infty})\phi^{(2)}(x_2)\cdots\phi^{(p_{-})}(x_{p_{-}})\psi_{k_p}^{(p)}(x_p)
\phi^{(p_{+})}(x_{p_{+}})\cdots,
\end{array}
\end{equation}
where $p=2,3,\dots$;
the real Haar wavelet function $\psi^{(\infty)}(x_{\infty})$ and refinable function
$\phi^{(\infty)}(x_{\infty})=\Omega(x_{\infty})$ are given by (\ref{m=2-h}) and (\ref{m=3-h}),
$x_{\infty}\in \bQ_{\infty}$; the $r$-adic Haar wavelet functions $\psi^{(r)}_{k_r}(x_r)$, $k_r\in \bF^{\times}$,
are given by formulas (\ref{p-101**})--(\ref{p-101***}), the refinable function
$\phi^{(r)}(x_r)=\Omega\big(|x_r|_r\big)$, $x_r\in \bQ_r$; \ $r=2,3,\dots$.

Thus, this wavelet basis has the form
\begin{equation}
\label{w-62.8=10**-ad-ss-1}
T_{\widehat{a}}\Phi,
\quad
\Psi_{\widehat{k},\widehat{i}_p; j \, \widehat{a}}
:=(2\cdot2\cdot3\cdots p)^{j/2}
\big(M^{-j}_{(p)}T_{\widehat{a}}\big)\,\Psi_{\widehat{k},\widehat{i}_p},
\end{equation}
where
$\widehat{i}_p=(i_\infty,i_2,\dots,i_p)\in\{1,2\}^{\sharp(p)+1}\setminus
\{1,1,\dots,1\}$, $\widehat{k}=(k_2,\dots,k_p)$, $k_s\in \bF_{s}^{\times}$,
$s=2,3,\dots,p$; $j\in\bZ_{+}$;
$\widehat{a}=(a_{\infty},a_{2},\dots,a_{p})\in I_{\bA}$, $p\in V_\bQ$.
We stress that here $p(\widehat{a})=p(j)=p(\widehat{k})=p$.

Due to Theorem~\ref{p-th4} and~\cite{Al-Ev-Sk}, by formula (\ref{w-62.8=10**-ad-ss-1})
all Haar bases generated by adelic Haar MRA are described.

\begin{Corollary}
\label{A-00-1-1}
Let {\rm(}see~{\rm(\ref{m=wd-62-ad}))}
\begin{equation}
\label{m=wd-62-ad-A-0}
{\widetilde\Psi}_{\widehat{k},\widehat{i}_p'}(x'):=
{\widetilde\Psi}_{\widehat{k},(i_2,\dots,i_m)}
=\vartheta^{(2)}_{i_2}\otimes\cdots\otimes \vartheta^{(m)}_{i_m}
\otimes\phi^{(m_{+})}\otimes\cdots.
\end{equation}
The vectors {\rm(}see~{\rm(\ref{w-62.8=10**-ad-ss-1}))}
\begin{equation}
\label{w-62.8=10**-ad-ss-1-A-0}
T_{\widehat{a}'}\Phi,
\quad
{\widetilde\Psi}_{\widehat{k},\widehat{i}'_p; j \, \widehat{a}'}:=(2\cdot3\cdots p)^{j/2}
\big(M^{-j}_{(p)}T_{\widehat{a}}\big)\,\Psi_{\widehat{k},\widehat{i}'},
\end{equation}
form the orthonormal basis in the space $L^2({\widetilde\bA})$, where
$\widehat{i}'_p=(i_2,\dots,i_p)\in\{1,2\}^{\sharp(p-1)+1}\setminus \{1,\dots,1\}$,
$\widehat{k}=(k_2,\dots,k_p)$, $k_s\in \bF_{s}^{\times}$,
$s=2,3,\dots,m$; $j\in\bZ_{+}$; $\widehat{a}'=(a_{2},\dots,a_{p})\in I_{\bA}$,
$p\in  V_\bQ\setminus\infty$.
We stress that here $p(\widehat{a}')=p(j)=p(\widehat{k}')=p$.
\end{Corollary}

\section{Lizorkin spaces on adeles}
\label{s7}

\subsection{Real and $p$-adic Lizorkin spaces.}\label{s7.1}
Recall some facts from~\cite{Liz1}--~\cite{Liz3},~\cite[2.]{Sam3},
~\cite[\S 25.1]{Sam-Kil-Mar}. Consider the following space
$$
\Psi(\bR)=\{\psi(\xi)\in \cS(\bR):
\psi^{(j)}(0)=0, \, j=0,1,2,\dots\},.
$$
where $\cS(\bR)$ is the real Schwartz space of tempered test functions.
The space of functions
$$
\Phi(\bR)=\{\phi: \phi=F[\psi], \, \psi\in \Psi(\bR)\}.
$$
is called the {\it real Lizorkin space of test functions}.
The Lizorkin space can be equipped with the topology of the space $\cS(\bR)$,
which makes $\Phi$ a complete space~\cite[2.2.]{Sam3},
~\cite[\S 25.1.]{Sam-Kil-Mar}.
Since the Fourier transform is a linear isomorphism ${\cS}(\bR)$ onto
${\cS}(\bR)$, this space admits the following characterization:
$\phi\in \Phi(\bR)$ if and only if $\phi\in \cS(\bR)$ is orthogonal to
polynomials, i.e.,
\begin{equation}
\label{8-11}
\int_{\bR}x^{n}\phi(x)\,dx=0, \quad n=0,1,2,\dots.
\end{equation}

The space $\Phi'(\bR)$ is called the {\it real Lizorkin space of distributions\/}.

According to~\cite{Al-Kh-Sh3},~\cite[Ch.~7]{Al-Kh-Sh=book} the {\it $p$-adic
Lizorkin space of test functions} is defined as
$\Phi(\bQ_p)=\{\phi: \phi=F[\psi], \, \psi\in \Psi(\bQ_p)\}$,
where $\Psi(\bQ_p)=\{\psi(\xi)\in \cD(\bQ_p): \psi(0)=0\}$.
The space $\Phi(\bQ_p)$ equipped with the topology of the space
$\cD(\bQ_p)$ is a complete space.

\begin{Lemma}
\label{lem1-pp}
{\rm (~\cite{Al-Kh-Sh3})} We have
{\rm (a)} $\phi\in \Phi(\bQ_p)$ iff $\phi\in \cD(\bQ_p)$ and
\begin{equation}
\label{14.1-ll}
\int_{\bQ_p}\phi(x)\,dx=0;
\end{equation}

{\rm (b)} $\phi \in {\cD}^l_N(\bQ_p)\cap\Phi(\bQ_p)$
\, iff \, $\psi=F^{-1}[\phi]\in {\cD}^{-N}_{-l}(\bQ_p)\cap\Psi(\bQ_p)$.
\end{Lemma}

The topological dual of the space $\Phi(\bQ_p)$ is the space of $p$-adic
{\em Lizorkin distributions} $\Phi'(\bQ_p)$. The space $\Phi'(\bQ_p)$ can
be obtained from $\cD'(\bQ_p)$ by ``sifting out'' constants (see~\cite{Al-Kh-Sh3}).

\subsection{Adelic Lizorkin spaces.}\label{s7.2}
Consider the subspace $\Psi(\bA)$ of the space of test functions $\cS(\bA)$
(see Subsec.~\ref{s2.4}) consisting of finite linear combinations of elementary functions (\ref{ad-1})
$\eta(\xi)=\prod_{p\in V_\bQ}\eta_{p}(\xi_{p})$ (where $V_\bQ$ is defined by (\ref{Q-valuations}))
with properties $(i)-(iii)$ of Definition~\ref{de-B-Sch}, where $\eta_{p}(\xi_{p})=\Omega(|\xi_p|_p)$
for all $p>P$) and such that
$$
\quad (iv)\,\,  \frac{d^s}{d\xi_{\infty}^s}\eta(\xi_{\infty},\xi_{2},\xi_{3},\dots,\xi_{r_{-}},\xi_{r},
\xi_{r_{+}},\dots,\xi_{P},\xi_{P_{+}},\dots)\Big|_{\xi_{r}=0}=0,
\quad s=0,1,2,\dots,
\qquad\qquad
$$
for all $r\in \{\infty,2,3,\dots,P\}$, where $P=P(\eta)$ is the {\em parameter of finiteness}
of an elementary function $\eta$.

Let $\Psi({\widetilde\bA})$ be the space of test functions  $\cS({\widetilde\bA})$ connected
with the {\em non-Archimedean part} of adeles ${\widetilde\bA}$ (see Subsec.~\ref{s2.4})
consisting of finite linear combinations of elementary functions
$\eta(\xi')=\prod_{p\in V_\bQ\setminus\infty}\eta_{p}(\xi_{p})$ with properties
$(i)-(iii)$ of Definition~\ref{de-B-Sch} (here $\eta_{p}(\xi_{p})=\Omega(|\xi_p|_p)$
for all $p>P$) and such that
$$
\quad (iv*)\,\,  \eta(\xi_{2},\xi_{3},\dots,\xi_{r_{-}},\xi_{r},
\xi_{r_{+}},\dots,\xi_{P},\xi_{P_{+}},\dots)\big|_{\xi_{r}=0}=0,
\qquad r\in \{2,3,\dots,P\},
$$
where $P=P(\eta)$ is the {\em parameter of finiteness} of an elementary function $\eta$.

\begin{Definition}
\label{de-L-1-ad} \rm
The spaces
$$
\Phi(\bA)=\bigl\{\zeta: \zeta=F[\eta], \, \eta\in \Psi(\bA)\bigr\}
\quad \text{and}\quad
\Phi({\widetilde\bA})=\bigl\{\zeta: \zeta=F[\eta], \, \eta\in \Psi({\widetilde\bA})\bigr\}
$$
are called the {\it adelic Lizorkin spaces of test functions}.
\end{Definition}

$\Phi(\bA)$ can be equipped with the topology of the space $\cS(\bA)$ (see (\ref{ad-1-test-1})). 
Since the Fourier transform is a linear isomorphism ${\cS}(\bA)$ onto
${\cS}(\bA)$ (see~\cite[Ch.~III,\S2.2]{G-Gr-P}), we have $\Phi(\bA)\subset\cS(\bA)$.

The Lizorkin spaces $\Phi(\bA)$ and $\Phi({\widetilde\bA})$ admit the following characterizations.

\begin{Lemma}
\label{lem1-10-ad}
$(1)$ Any elementary function
$$
\zeta(x)=\zeta_{\infty}(x_{\infty})\prod_{p\in V_\bQ\setminus\infty}\zeta_{p}(x_{p})\in \Phi(\bA)
$$
{\rm(}here $\zeta_{p}(x_{p})=\Omega(|x|_p)$ for all $p> P${\rm)} if and only if $\zeta\in \cS(\bA)$ and
\begin{equation}
\label{50-ad}
\int_{\bQ_{r}}x_{r}^{s}\zeta(x_{\infty},x_{2},x_{3},\dots,x_{r_{-}},x_{r},
x_{r_{+}},\dots,x_{P},x_{P_{+}},\dots)\,dx_r=0,
\end{equation}
where if $r=\infty$, then $s=0,1,2,\dots$, and if $r=2,3,\dots,P$, then $s=0$.

$(2)$ Any elementary function
$$
\widetilde{\zeta}(x')=\prod_{p\in V_\bQ\setminus\infty}\zeta_{p}(x_{p})\in \Phi({\widetilde\bA})
$$
if and only if $\widetilde{\zeta}\in \cS({\widetilde\bA})$ and
\begin{equation}
\label{50-ad-1}
\int_{\bQ_{r}}\widetilde{\zeta}(x_{2},x_{3},\dots,x_{r_{-}},x_{r},
x_{r_{+}},\dots,x_{P},x_{P_{+}},\dots)\,dx_r=0, \quad r=2,3,\dots,P.
\end{equation}
\end{Lemma}

\begin{Remark}\rm
\label{Top-Liz}
In view of condition $(iii)$, similarly to (\ref{ad-1-test-1}), the spaces $\Phi(\bA)$
and $\Phi({\widetilde\bA})$ admit a representation in the form
\begin{equation}
\label{ad-1-test-1-L}
\Phi(\bA)=\lim\limits_{m\in V_{\bQ}\setminus\infty}{\rm ind}\,\Phi^{[m]}(\bA), \qquad
\Phi({\widetilde\bA})=
\lim\limits_{m\in V_{\bQ}\setminus \infty}{\rm ind}\,\Phi^{[m]}({\widetilde\bA}),
\end{equation}
where $\Phi^{[m]}(\bA)\subset \Phi(\bA)$ and $\Phi^{[m]}({\widetilde\bA})\subset \Phi({\widetilde\bA})$
are subspaces of the corresponding test functions with the {\em parameter of finiteness} $m$, $m\in V_{\bQ}$.
The representation (\ref{ad-1-test-1-L}) equip the spaces $\Phi(\bA)$ and $\Phi({\widetilde\bA})$
with the {\em inductive limit topology}.
\end{Remark}

The spaces $\Phi'(\bA)$ and $\Phi'({\widetilde\bA})$ are called the {\it adelic Lizorkin spaces of distributions}.

We define the Fourier transform of distributions
$f\in \Phi'(\bA)$ and $g\in \Psi'(\bA)$ by the relations
\begin{equation}
\label{51-ad}
\begin{array}{rcl}
\displaystyle
\langle F[f],\eta\rangle=\langle f,F[\eta]\rangle,
&& \forall \, \eta\in \Psi(\bA), \medskip \\
\displaystyle
\langle F[g],\zeta\rangle=\langle g,F[\zeta]\rangle,
&& \forall \, \zeta\in \Phi(\bA). \\
\end{array}
\end{equation}
By definition, $F[\Phi(\bA)]=\Psi(\bA)$, \
$F[\Psi(\bA)]=\Phi(\bA)$, i.e., (\ref{51-ad})
give well defined objects. Moreover, we have
$F[\Phi'(\bA)]=\Psi'(\bA)$, $F[\Psi'(\bA)]=\Phi'(\bA)$.

\subsection{Adelic Lizorkin spaces and wavelets.}\label{s7.3}
It is well known that $p$-adic Haar wavelet functions (\ref{p-101**-basis})
from $L^2(\bQ_p)$ satisfy the condition (see~\cite[Ch.~8]{Al-Kh-Sh=book})
\begin{equation}
\label{51-1-ad-2}
\int_{\bQ_p}\psi_{k; j a}(x)\,dx=0, \quad k\in \bF_{p}^{\times}, \, j\in \bZ, \, a\in I_p,
\end{equation}
where $\psi_{k;j a}(x)=p^{j/2}\psi_{k}(p^{-j}x-a)$ and wavelet functions $\psi_{k}$,
$k\in \bF_{p}^{\times}$, are given by Theorem~\ref{p-th4}.

In view of the formula (see~\cite[Ch.~8]{Al-Kh-Sh=book} or~\cite{Koz0})
$$
\big(\Omega(|x|_p),\psi_{k; j a}(x)\big)=\left\{
\begin{array}{lll}
p^{j/2}, && a=0, \, j \le -1, \\
0, && \text{otherwise}, \\
\end{array}
\right.
$$
$p$-adic Haar wavelet functions (\ref{62.0-1=zz}) from $L^2(\bZ_p)$ satisfy the following condition:
\begin{equation}
\label{51-1-ad-3}
\int_{\bZ_p}{\widetilde\psi}^{(0)}_{k;j a}(x)\,dx=\big(\Omega(|x|_p),\psi_{k; j a}(x)\big)=0,
\quad k\in \bF_{p}^{\times}, \, j\in \bZ_{+}, \, a\in I_p^j,
\end{equation}
where $I_p^j$ is defined in Proposition~\ref{lem-zz-1}.
It is easy to see that for the $p$-adic Haar wavelet functions (\ref{62.0-1=zz-11})
from $L^2(\bZ_p)$ we also have
\begin{equation}
\label{51-1-ad-33}
\int_{\bZ_p}{\widetilde\psi}_{k;j a}(x)\,dx=0,
\quad k\in \bF_{p}^{\times}, \, j\in \bZ_{+}, \, a\in I_p^{(j)},
\end{equation}
where $I_p^{(j)}$ is described in Remark~\ref{r.restrict}.

Thus, taking into account relations (\ref{51-1-ad-2})--(\ref{51-1-ad-33}),
$\int_{\bQ_p}\phi(x)\,dx=1$, and Lemma~\ref{lem1-10-ad}, we conclude that the
adelic wavelet functions (\ref{1-ad+=11}) and (\ref{w-62.8=10**-ad-ss-1})
belong to the Lizorkin space $\Phi({\widetilde\bA})$ {\em only} if products
(\ref{1-ad+=11-00}) and (\ref{w-62.8=10**-ad-ss-1-A-0}) {\em do not contain}
the functions $\phi^{(p)}(x_p-a_p)$ as factors for $p\le m$.

\section{Characterization of the adelic Lizorkin spaces in terms of wavelets}
\label{s8}

To construct adelic wavelet bases in $L^2(\bA)$, we used the real Haar basis (\ref{m=1-h})
in $L^2(\bR)$, for which we have $\int_{\bR}\psi_{j n}^{H}(t)\,dt=0$.
Since $\int_{\bR}t^n\psi_{j n}^{H}(t)\,dt=0$ does not hold for $n\in \bN$,
it is clear that $\psi_{j n}^{H}\notin \Phi(\bR)$. Therefore, in contrast to the
$p$-adic case (see~\cite[8.14]{Al-Kh-Sh=book}), the characterization of the adelic
Lizorkin spaces in terms of wavelets is possible only for
the space $\Phi({\widetilde\bA})$.

For simplicity here and in what follows we will suppose that the adelic wavelets
{\rm(\ref{1-ad+=11-00})} are constructed by using the one-dimensional wavelet
basis (\ref{62.0-1=111-K}) and its restriction (\ref{62.0-1=zz}) to $\bZ_p$.

\begin{Lemma}
\label{lem-w-1**-ad-00}
Any test function $\widetilde{\zeta}\in \Phi({\widetilde\bA})$ can be represented in
the form of a {\em finite} sum
\begin{equation}
\label{wav-9.4=1-ad-00}
\widetilde{\zeta}(x')=\sum_{\alpha\in \widetilde{\Lambda}}c_{\alpha}\widetilde{\Psi}_{\alpha}(x'),
\quad x'=(x_{2},x_{3},\dots) \in {\widetilde\bA},
\end{equation}
where $\widetilde{\Lambda}$ is the set of indices {\rm(\ref{A_(p,m)-phi-00})};
$c_{\alpha}$ are constants; $\widetilde{\Psi}_{\alpha}(x')$ are wavelet
functions {\rm(\ref{1-ad+=11-00})}
$$
\widetilde{\Psi}_{\alpha}=\widetilde{\Psi}_{(\widehat{k}; \widehat{j}\ \widehat{a})}
=\bigotimes\limits_{2\le q\leq m}\psi_{\alpha_{q}}^{(q)}
\otimes\bigotimes\limits_{m<q}\phi^{(q)}, \quad \alpha\in \widetilde{\Lambda},
$$
which do not contain the factors $\psi_{\alpha_{q}}^{(q)}(x_{q})=\phi^{(q)}(x_q-a_q)$,
$a_p \in I_p$, {\rm(}see {\rm(\ref{basis-phi-inQ}), (\ref{basis-phi-inZ}))} for $q\le m$.
\end{Lemma}

\begin{proof}
By definition of the space $\Phi({\widetilde\bA})$ (see Subsec.~\ref{s7.2}),
it is sufficient to prove this lemma for the case of an elementary function
$\widetilde{\zeta}(x')=\prod_{p\in V_\bQ\setminus\infty}\zeta_{p}(x_{p})$, where
$\zeta_{p}(x_{p})=\phi^{(p)}(x_p)=\Omega(|x_p|_p)$ for all $p>P$, and $P=P(\widetilde{\zeta})$
is the {\em parameter of finiteness} of an elementary function $\widetilde{\zeta}$.
It is clear that $\widetilde{\zeta}\in L^2({\widetilde\bA})$. Then (\ref{wav-9.4=1-ad-00})
holds. We will prove that only the finite number of $c_{\alpha}\ne 0$.

According to Subsec.~\ref{s7.3},
$\widetilde{\Psi}_{\alpha}(x')\in \Phi({\widetilde\bA})$ only if the product
(\ref{1-ad+=11-00}) does not contain functions $\phi^{(p)}(x_p-a_p)$ as factors for $p\le m$.
Now we need to calculate the coefficients
\begin{equation}
\label{Liz-ad-1}
c_{\alpha}=c_{(\widehat{k}; \widehat{j}\ \widehat{a})}
=\big(\widetilde{\zeta}(x'), \widetilde{\Psi}_{\alpha}(x')\big)
=\prod_{p\in V_\bQ\setminus\infty}\big(\zeta_{q}(x_{q}), \, \psi_{\alpha_{q}}^{(q)}(x_q)\big),
\end{equation}
where one-dimensional wavelet functions $\psi_{\alpha_{q}}^{(q)}=\psi^{(q)}_{k_q; j_q a_q}(x_{q})$,
$\alpha_{q}=(k_q,j_q,a_q)\in {\bI}_p^+=({\bF}_q^\times,\bZ_{+},I_q)$,
are given by (\ref{62.0-1=111-K}), (\ref{62.0-1=zz}), (\ref{basis-phi-inQ}), (\ref{basis-phi-inZ}),
and among them there are no wavelet functions $\phi^{(q)}(x_q-a_q)$.

1. Suppose that $m > P$. Then we have
$$
c_{\alpha}=c_{(\widehat{k}; \widehat{j}\ \widehat{a})}=\big(\widetilde{\zeta}(x'), \widetilde{\Psi}_{\alpha}(x')\big)
=\prod_{2\le q\le P}\big(\zeta_{q}(x_{q}), \, \psi_{\alpha_{q}}^{(q)}(x_q)\big)
\times
$$
$$
\qquad\qquad
\times
\prod_{P< q\le m}\big(\phi^{(q)}(x_q), \, \psi_{\alpha_{q}}^{(q)}(x_q)\big)
\prod_{m< q}\big(\phi^{(q)}(x_q), \, \phi^{(q)}(x_q)\big).
$$
According to (\ref{51-1-ad-2}) and (\ref{51-1-ad-3}), (\ref{51-1-ad-33}), we have
$$
\int_{\bZ_q}\psi_{\alpha_{q}}^{(q)}(x_q)\,dx_q=0, \quad \text{for} \quad P< q\le m,
$$
i.e., $c_{\alpha}=0$.

2. Let $m < P$. Consider the part of product (\ref{Liz-ad-1})
$$
\prod_{m< q\le P}\big(\zeta_{q}(x_{q}), \, \phi^{(q)}(x_q)\big)
=\prod_{m< q\le P}\big(\zeta_{q}^{0}(x_{q}), \, \phi^{(q)}(x_q)\big),
$$
where $\zeta_{q}^{0}(x_{q}):=\zeta_{q}(x_{q})\big|_{\bZ_q}\in \Phi(\bZ_q)$ and
$\Phi(\bZ_q)$ is the Lizorkin space of test functions with support in
$\bZ_q$ (see Sec.~\ref{s7.1}).
Using the Parseval-Steklov theorem, we obtain
$$
\prod_{m< q\le P}\big(\zeta_{q}(x_{q}), \, \phi^{(q)}(x_q)\big)
=\prod_{m< q\le P}\big(F[\zeta_{q}^{0}](\xi_{q}), \, F[\phi^{(q)}(x_{q})](\xi_q)\big).
$$
It is clear that $\zeta_{q}^{0}(x_{q})\in \Phi(\bZ_q)$
belongs to one of the spaces ${\cD}^{l_q}_{N_q}(\bZ_q)$, $l_q \le N_q \le 0$, and satisfies
condition (\ref{14.1-ll}). Then in view of (\ref{8.2-char-ad}) and Lemma~\ref{lem1-pp},
$F[\zeta_{q}^{0}]\in \Psi(\bQ_q)\cap \cD^{-N_{q}}_{-l_{q}}(\bZ_q)$, and
$\supp\,F[\zeta_{q}^{0}]\subset B_{-l_{q}}\setminus B_{-N_{q}}$, $0\le -N_q \le -l_q$.
At the same time, $F[\phi^{(q)}](\xi_q)=\phi^{(q)}(\xi_{q})$.
Thus $\supp\,F[\zeta_{q}^{0}] \cap \supp\,\phi^{(q)}=\emptyset$ and, consequently,
$\big(F[\zeta_{q}^{0}](\xi_{q}), \, F[\phi^{(q)}(x_{q})](\xi_q)\big)=0$,
$m< q\le P$. That is in this case $c_{\alpha}=0$.

3. Let $m=P$. In this case, using the Parseval-Steklov theorem, one can
rewrite the product (\ref{Liz-ad-1}) in the form
$$
c_{\alpha}=c_{(\widehat{k}; \widehat{j}\ \widehat{a})}
=\big(\widetilde{\zeta}(x'), \widetilde{\Psi}_{\alpha}(x')\big)
\qquad\qquad\quad\qquad\qquad\qquad\qquad\qquad\quad
$$
$$
=\prod_{2\le q\le m}\big(\zeta_{q}(x_{q}), \, \psi_{\alpha_{q}}^{(q)}(x_q)\big)
=\prod_{2\le q\leq m}\big(F[\zeta_{q}](\xi_{q}), \, F[\psi_{\alpha_{q}}^{(q)}](\xi_q)\big).
$$

Let $\zeta_{q}\in \Phi(\bQ_q)$ for $2\le q\leq p$. According to Definition~\ref{de-L-1-ad},
Lemma~\ref{lem1-pp}, and relation (\ref{8.2-char-ad}), any function $\zeta_{q}(x_{q})\in \Phi(\bQ_q)$
belongs to one of the spaces ${\cD}^{l_q}_{N_q}(\bQ_q)$ and satisfies condition
(\ref{14.1-ll}), i.e., $F[\zeta_{q}]\in \Psi(\bQ_q)\cap \cD^{-N_{q}}_{-l_{q}}(\bQ_q)$,
and $\supp\,F[\zeta_{q}]\subset B_{-l_{q}}\setminus B_{-N_{q}}$, $2\le q\leq p$.

For the wavelet function (\ref{basis-phi-inQ}),
$$
\psi_{\alpha_{q}}^{(q)}(x_{q})=\psi^{(q)}_{k_q;\, j_q a_q}(x_{q})
\qquad\qquad\qquad\qquad\qquad\qquad\qquad\qquad\qquad\qquad\qquad\quad
$$
\begin{equation}
\label{o-64.8*-ad-ww}
=q^{j_q/2}\chi_q\Big(\frac{k_q}{q}(q^{-j_q}x_q-a_q)\Big)\Omega\big(|q^{-j_q}x_q-a_q|_q\big),
\quad x_q\in \bQ_q,
\end{equation}
we have
\begin{equation}
\label{o-64.8*-ad}
F[\psi^{(q)}_{k_q;\, j_q a_q}](\xi_q)=q^{-j_q/2}\chi_q\big(q^{j_q}a_q\xi_q\big)
\Omega\Big(\Big|\frac{k_q}{q}+q^{j_q}\xi_q\Big|_q\Big), \quad j_q\in \bZ_{+}.
\end{equation}
Taking into account that $\supp\,F[\zeta_{q}]\subset B_{-l_{q}}\setminus B_{-N_{q}}$
and using formula (\ref{o-64.8*-ad}), one can conclude that
\begin{equation}
\label{o-64.8*-ad-11}
\big(F[\zeta_{q}](\xi_{q}), \, F[\psi_{\alpha_{q}}^{(q)}](\xi_q)\big)\ne 0,
\end{equation}
only if $q^{-N_{q}}\le |\xi_{q}|_q\le q^{-l_{q}}$ and $\frac{k_{q}}{q}+q^{j_{q}}\xi_q=\eta_q\in \bZ_q$
for any $2\le q\leq p$. Since $\xi_q=q^{-j_{q}}\big(\eta_q-\frac{k_{q}}{q}\big)$ and
$|\xi_{q}|_q=q^{j_{q}}\big|\eta_q-\frac{k_{q}}{q}\big|_q=q^{j_{q}+1}$, one can see
that the product (\ref{o-64.8*-ad-11}) is nonzero only for {\em finite number
of indices} $j_{q}$ such that
$$
q^{-N_{q}}< |\xi_{q}|_q=q^{j_{q}+1}\le q^{-l_{q}}, \quad 2\le q\leq p.
$$

Now we consider the scalar product
$\big(F[\zeta_{q}](\xi_{q}), F[\psi_{\alpha_{q}}^{(q)}](\xi_q)\big)$,
where $\zeta_{q}\in \Phi(\bZ_q)$ and
$$
\psi_{\alpha_{q}}^{(q)}(x_q)=\psi^{(q)}_{k_q;\, j_q a_q}(x_{q})
=q^{j_q/2}\chi_q\Big(\frac{k_q}{q}(q^{-j_q}x_q-a_q)\Big)\Omega\big(|q^{-j_q}x_q-a_q|_q\big),
\quad x_q\in \bZ_q,
$$
is a wavelet function given by (\ref{basis-phi-inZ}), $p< q\leq m$.
Using the identity~\cite[VII.1]{Vl-V-Z},~\cite{Koz0}
\begin{equation}
\label{identity-ad}
\Omega\big(|p^{j}x_p-a_p|_p\big)\Omega\big(|p^{j'}x_p-a'_p|_p\big)
=\Omega\big(|p^{j}x_p-a_p|_p\big)\Omega\big(|p^{j'-j}a_p-a'_p|_p\big),
\quad j \le j',
\end{equation}
by explicit calculation we obtain
\begin{equation}
\label{o-64.8*-ad-11*}
F[\psi_{\alpha_{q}}^{(q)}\big|_{\bZ_q}](\xi_q)=q^{-j_q/2}\Omega(|q^{j_q}a_q|_q)\chi_q\big(q^{j_q}a_q\xi_q\big)
\Omega\Big(\Big|\frac{k_q}{q}+q^{j_q}\xi_q\Big|_q\Big), \quad j_q\in \bZ_{+}.
\end{equation}
Next, using (\ref{o-64.8*-ad-11*}) and repeating the above calculation almost word for word, we find that
\begin{equation}
\label{o-64.8*-ad-12}
\big(F[\zeta_{q}](\xi_{q}), \, F[\psi_{\alpha_{q}}^{(q)}\big|_{\bZ_q}](\xi_q)\big) \ne 0,
\end{equation}
only for {\em finite number of indices} $j_{q}$, $p<q\leq m$.

Thus $c_{\alpha}=c_{(\widehat{k}; \widehat{j}\ \widehat{a})}\ne 0$ for
{\em finite number of indices} $j_{q}$, $2\le q\leq m$.
Moreover, for all wavelet functions $\widetilde{\Psi}_{\alpha}(x')$,
in the product (\ref{Liz-ad-1}) we have $m=P$.

Now we consider the product (\ref{Liz-ad-1}) again
$$
c_{\alpha}=c_{(\widehat{k}; \widehat{j}\ \widehat{a})}
=\big(\widetilde{\zeta}(x'), \widetilde{\Psi}_{\alpha}(x')\big)
=\prod_{2\le q\le m}\big(\zeta_{q}(x_{q}), \, \psi_{\alpha_{q}}^{(q)}(x_q)\big),
$$
where according to the above calculation, $m=P$.
Now let calculate the expression
$\big(\zeta_{q}(x_{q}), \, \psi_{\alpha_{q}}^{(q)}(x_q)\big)$.
Here the function $\zeta_{q}(x_{q})\in \Phi(\bQ_q)$ belongs to one of the spaces
${\cD}^{l_q}_{N_q}(\bQ_q)$ and $\psi_{\alpha_{q}}^{(q)}(x_q)$ is given by
(\ref{o-64.8*-ad-ww}).
Using identity (\ref{identity-ad}) we have
$$
\Omega\big(|q^{-j_q}x_q-a_q|_q\big)\Omega\big(|q^{N_q}x_q|_q\big)
=\left\{
\begin{array}{ll}
\Omega\big(|q^{-j_q}x_q-a_q|_q\big)\Omega\big(|q^{N+j_q}a_q|_q\big), &  -j_q \le N_q, \medskip \\
\Omega\big(|q^{N_q}x_q|_q\big)\Omega\big(|a_q|_q\big), &  -j_q > N_q. \\
\end{array}
\right.
$$
This relation implies that $\big(\zeta_{q}(x_{q}), \, \psi_{\alpha_{q}}^{(q)}(x_q)\big)\ne 0$
at least for $q^{N+j_q}a_q\in\bZ_q$ ($-j_q \le N_q$) or $a_q=0$ ($-j_q > N_q$). Since
the set of indices $j_q$ is finite, the set of the above indices $a_q$ is also finite.
Thus products (\ref{o-64.8*-ad-11}), (\ref{o-64.8*-ad-12}) are nonzero only for finite number
of $a_q\in I_q$.
\end{proof}

\begin{Corollary}
\label{lem-w-1**-ad}
Any test function $\zeta \in \Phi(\bA)$ can be represented in the form of a {\em finite} sum
\begin{equation}
\label{wav-9.4=1-ad}
\zeta(x)=\sum_{\alpha\in \widetilde{\Lambda}}c_{\alpha}(x_{\infty})\widetilde{\Psi}_{\alpha}(x'),
\quad x=(x_{\infty},x')\in \bA,
\end{equation}
where $\widetilde{\Lambda}$ is the set of indices {\rm(\ref{A_(p,m)-phi-00})};
$c_{\alpha}(x_{\infty})\in \Phi(\bR)$ are some Lizorkin test functions;
$\widetilde{\Psi}_{\alpha}(x')=\widetilde{\Psi}_{(\widehat{k}; \widehat{j}\ \widehat{a})}(x')$
are wavelet functions {\rm(\ref{1-ad+=11-00})} which do not contain
factors $\psi_{\alpha_{q}}^{(q)}(x_{q})=\phi^{(q)}(x_q-a_q)$,
$a_p \in I_p$, {\rm(}see {\rm(\ref{basis-phi-inQ}), (\ref{basis-phi-inZ}))} for $q\le m$.
\end{Corollary}

Using standard results from the book~\cite{Schaefer}, we obtain the following assertion.

\begin{Proposition}
\label{pr-w-2***-ad}
Any distribution $f\in \Phi'(\bA)$ can be realized in the form of an {\em infinite} sum:
\begin{equation}
\label{wav-9.4=3*-ad}
f(x)=\sum_{\alpha\in \widetilde{\Lambda}}b_{\alpha}(x_{\infty})\widetilde{\Psi}_{\alpha}(x'),
\quad x=(x_{\infty},x')\in \bA,
\end{equation}
where $b_{\alpha}(x_{\infty})\in \Phi'(\bR)$ are some real Lizorkin distributions;
$\widetilde{\Psi}_{\alpha}(x')$ are wavelet functions {\rm(\ref{1-ad+=11-00})}.
\end{Proposition}

Here any distribution $f\in \Phi'(\bA)$ is associated with representation
(\ref{wav-9.4=3*-ad}), where the coefficients
\begin{equation}
\label{wav-9.4=4*}
b_{\alpha}(x_{\infty})\stackrel{def}{=}\bigl\langle f(x_{\infty},x'), \widetilde{\Psi}_{\alpha}(x')\bigr\rangle,
\quad \alpha\in \widetilde{\Lambda}.
\end{equation}
And vice versa, taking into account Corollary~\ref{lem-w-1**-ad} and the orthonormality of
the wavelet basis (\ref{1-ad+=11-00}), any infinite sum (\ref{wav-9.4=3*-ad}) is associated
with the distribution $f \in \Phi'(\bA)$, whose action on a test function $\zeta(x) \in \Phi(\bA)$
is defined as
$$
\big\langle f,\zeta\big\rangle
=\Big\langle
\sum_{\beta\in \widetilde{\Lambda}}b_{\beta}(x_{\infty})\widetilde{\Psi}_{\beta}(x'), \,
\sum_{\alpha\in \widetilde{\Lambda}}c_{\alpha}(x_{\infty})\widetilde{\Psi}_{\alpha}(x')\Big\rangle
$$
\begin{equation}
\label{wav-9.4=5*-ad}
=\sum_{\alpha\in \widetilde{\Lambda}}
\big\langle b_{\alpha}(x_{\infty}), c_{\alpha}(x_{\infty})\big\rangle,
\end{equation}
where the latter sum is finite.

\begin{Proposition}
\label{pr-w-2***-ad-1} Any distribution ${\widetilde f}\in \Phi'({\widetilde\bA})$ can be
realized as an {\em infinite} sum of the form
\begin{equation}
\label{wav-9.4=3*-ad-1}
{\widetilde f}(x')=\sum_{p}\sum_{\alpha\in\widetilde{\Lambda}}b_{\alpha}\widetilde{\Psi}_{\alpha}(x'),
\quad x'\in {\widetilde\bA},
\end{equation}
where $b_{\alpha}$ are constants; $\widetilde{\Psi}_{\alpha}(x')$ are wavelet functions
{\rm(\ref{1-ad+=11-00})}.
\end{Proposition}

It is clear that in Lemma~\ref{lem-w-1**-ad-00}, Corollary~\ref{lem-w-1**-ad},
and Propositions~\ref{pr-w-2***-ad},~\ref{pr-w-2***-ad-1} one can use wavelet functions
(\ref{w-62.8=10**-ad-ss-1-A-0}) instead of wavelet functions (\ref{1-ad+=11-00}).

\section{Pseudo-differential operators on adeles}
\label{s9}

\subsection{Real and $p$-adic fractional operators.}\label{s9.1}
Let us introduce a distribution from ${\cS}'(\bR)$ (we keep the notation
$|\cdot|=|\cdot|_{\infty}$.)
\begin{equation}
\label{15-11}
\kappa_{\alpha}^{(\infty)}(x)=\frac{|x|^{\alpha-1}}{\gamma_1(\alpha)},
\quad \alpha \ne -2s, \, \alpha \ne 1+2s, \quad s=0,1,2,\dots, \quad x\in \bR,
\end{equation}
called the {\it Riesz kernel\/}, where $|x|^{\alpha}$ is a
homogeneous distribution of degree~$\alpha$ defined in
~\cite[Lemma~2.9.]{Sam3}, \cite[(25.19)]{Sam-Kil-Mar},~\cite[Ch.I,\S3.9.]{Gel-S},
and $\gamma_1(\alpha)=\frac{2^{\alpha}\pi^{\frac{1}{2}}\Gamma(\frac{\alpha}{2})}
{\Gamma(\frac{1-\alpha}{2})}$.
The Riesz kernel is an entire function of the complex variable $\alpha$.

One can define the {\it Riesz kernel} (\ref{15-11}) in the {\it real Lizorkin space of
distributions\/} $\Phi'(\bR)$ (see~\cite{Liz3},~\cite[Lemma~2.13.]{Sam3},~\cite[Lemma~25.2.]{Sam-Kil-Mar}):
\begin{equation}
\label{19}
\kappa_{\alpha}^{(\infty)}(x)=\left\{
\begin{array}{lcl}
\displaystyle
\frac{|x|^{\alpha-1}\Gamma(\frac{1-\alpha}{2})}
{2^{\alpha}\sqrt{\pi}\Gamma(\frac{\alpha}{2})}
=\frac{|x|^{\alpha-1}}{2\Gamma(\alpha)\cos(\frac{\pi\alpha}{2})},
&& \alpha \ne -2s, \, \alpha \ne 1+2s, \medskip \\
\displaystyle
(-1)^{s+1}\frac{|x|^{2s}\log|x|}{\pi(2s)!},
&& \alpha=1+2s, \medskip \\
\displaystyle
(-1)^s\delta^{(2s)}(x),
&& \alpha=-2s, \, s\in \bZ_{+}. \\
\end{array}
\right.
\end{equation}
According to~\cite[Ch.II,\S3.3.,(2)]{Gel-S},~\cite[Lemma~2.13.]{Sam3},
~\cite[Lemma~25.2.]{Sam-Kil-Mar},
\begin{equation}
\label{20-11}
F[\kappa_{\alpha}^{(\infty)}(x)](\xi)=|x|^{-\alpha}, \quad x\in \bR.
\end{equation}

Define the {\em real Riesz fractional operator} $D^{\alpha}_{\infty}$ on the
Lizorkin space $\Phi(\bR)$ as a convolution
\begin{equation}
\label{22}
\big(D^{\alpha}_{\infty}\varphi\big)(x)=\big(\kappa_{-\alpha}^{(\infty)}*\varphi\big)(x)
=\bigl\langle \kappa_{-\alpha}^{(\infty)}(\cdot),\varphi(x-\cdot)\bigr\rangle,
\quad \varphi \in \Phi(\bR),
\end{equation}
where $\kappa_{\alpha}^{(\infty)}$ is given by (\ref{19}).
It is clear that~\cite[(25.2)]{Sam-Kil-Mar}
\begin{equation}
\label{23}
\big(D^{\alpha}_{\infty}\varphi\big)(x)=F^{-1}[|\xi|^{\alpha} F[\varphi](\xi) ](x),
\quad \varphi \in \Phi(\bR).
\end{equation}

\begin{Lemma}
\label{lem2-11}
The Lizorkin spaces of test functions $\Phi(\bR)$ and distributions $\Phi'(\bR)$
are invariant under the Riesz fractional operator {\rm(\ref{23})}
and $D^{\alpha}_{\infty}(\Phi'(\bR))=\Phi'(\bR)$.
\end{Lemma}

The {\em $p$-adic fractional operator} $D^{\alpha}_{p}$ was
introduced on the space of distributions ${\cD}'(\bQ_p)$ in~\cite{Taib1},
~\cite[III.4.]{Taib3}.
In~\cite{Al-Kh-Sh3}, the fractional operator $D^{\alpha}_{p}$ was defined in the
Lizorkin space of distributions $\Phi'(\bQ_p)$ for all $\alpha\in \bC$ by
the following relations:
\begin{equation}
\label{61**-ad}
\big(D^{\alpha}_{p}f\big)(x)
=F^{-1}\big[|\cdot|^{\alpha}_pF[f](\cdot)\big](x),
\quad f \in \Phi'(\bQ_p), \quad \alpha\in \bC.
\end{equation}
Representation (\ref{61**-ad}) can be rewritten as a convolution
\begin{equation}
\label{61**-ad-11}
\big(D^{\alpha}_{p}f\big)(x)=\big(\kappa_{-\alpha}^{(p)}*f\big)(x)
=\langle \kappa_{-\alpha}^{(p)}(\cdot),f(x-\cdot)\rangle,
\quad f\in \Phi'(\bQ_p), \quad \alpha \in \bC,
\end{equation}
where (see~\cite{Al-Kh-Sh3})
\begin{equation}
\label{19-p}
\kappa_{\alpha}^{(p)}(x)=\left\{
\begin{array}{lll}
\frac{|x|_p^{\alpha-1}}{\Gamma_p(\alpha)}, &&
\alpha \ne 0, \, \,  1, \\
\delta(x), && \alpha=0, \\
-\frac{1-p^{-1}}{\log p}\log|x|_p, && \alpha=1, \\
\end{array}
\right.
\quad x\in \bQ_p,
\end{equation}
is the {\it Riesz kernel}, and
$$
\Gamma_p(\alpha)\stackrel{def}{=}\int_{\bQ_p}|x|_p^{\alpha-1}\chi_p(x)\,dx
=\frac{1-p^{\alpha-1}}{1-p^{-\alpha}}
$$
is the $p$-adic $\Gamma$-{\it function} (see~\cite[III,Theorem~(4.2)]{Taib3},
~\cite[VIII,(4.4)]{Vl-V-Z}).

\begin{Lemma}
\label{lem4-11}
{\rm(\cite{Al-Kh-Sh3})}
The Lizorkin spaces of test functions $\Phi(\bQ_p)$ and distributions $\Phi'(\bQ_p)$
are invariant under the fractional operator {\rm(\ref{61**-ad})}
and $D^{\alpha}_{p}(\Phi'(\bQ_p))=\Phi'(\bQ_p)$.
\end{Lemma}

\subsection{Adelic fractional operators on the Lizorkin spaces.}\label{s9.2}
Let us introduce on the space $\Phi(\bA)$ the {\em adelic fractional operator}
$D^{\widehat{\gamma}}$ of order
$\widehat{\gamma}=(\gamma_{\infty},\gamma_{2},\dots,\gamma_{p},\dots)\in \bC^{\infty}$,
which is defined by its projection on any subspace of test functions
$\Phi^{[m]}(\bA)\subset \Phi(\bA)$ (for details, see Remark~\ref{A=lim-m-A-(m)}):
\begin{equation}
\label{ad-frac-op}
D^{\widehat{\gamma}}\big|_{\Phi^{[m]}(\bA)}
\stackrel{def}{=}D^{\widehat{\gamma}}_{(m)}
=\bigotimes\limits_{p\in \{\infty,2,3,\dots,m\}}D^{\gamma_p}_{p}\,
\,\otimes\bigotimes\limits_{p>m}Id_p, \quad m\in V_{\bQ},
\end{equation}
where $D^{\gamma_\infty}_{\infty}$ and $D^{\gamma_p}_{p}$ are
defined by (\ref{22}), (\ref{23}) and (\ref{61**-ad}),
(\ref{61**-ad-11}), respectively, $Id_p$ is the identity operator in
$\Phi(\bQ_p)$. Here the fractional operator
$D^{\widehat{\gamma}}_{(m)}$ is an infinite tensor product of
one-dimensional operators.

If
$$
\zeta(x)=\zeta_{\infty}(x_{\infty})\prod_{p\in V_\bQ\setminus\infty}\zeta_{p}(x_{p})\in \Phi^{[m]}(\bA)
$$
is an elementary function, i.e., $\zeta_{p}(x_{p})=\phi^{(p)}(\xi_p)=\Omega(|x_p|_p)$
for all $p>m(\zeta)$, then taking into account Definition~(\ref{ad-frac-op}) and (\ref{23}),
(\ref{61**-ad}), we obtain
\begin{equation}
\label{op-ad-2}
\big(D^{\widehat{\gamma}}\big|_{\Phi^{[m]}(\bA)}\zeta\big)(x)
=\big(D^{\widehat{\gamma}}_{(m)}\zeta\big)(x)
=F^{-1}\big[|\cdot|^{\widehat{\gamma}}_{(m)}F[\zeta](\cdot)\big](x),
\end{equation}
where
\begin{equation}
\label{op-ad-1}
|\xi|^{\widehat{\gamma}}_{(m)}=\prod_{p\in\{\infty,2,3,\dots,m\}}|\xi_p|_p^{\gamma_p}\prod_{p>m}\phi^{(p)}(\xi_p),
\quad \xi=(\xi_{\infty},\xi_{2},\dots \xi_{p},\dots) \in \bA,
\end{equation}
is a symbol of the operator $D^{\widehat{\gamma}}_{(m)}$, $m\in V_{\bQ}$.

It is clear that $|\xi|^{\widehat{\gamma}}_{(m)}$ for all
$\widehat{\gamma}=(\gamma_{\infty},\gamma_{2},\dots,\gamma_{p},\dots)\in \bC^{\infty}$
such that for all $p>m$, $\gamma_p=0$ is a multiplier in the space of test functions
$\Psi^{[m]}(\bQ_p)$, $m\in V_\bQ$.

According to (\ref{22}), (\ref{61**-ad-11}), (\ref{14.2}), relation (\ref{op-ad-2})
can be rewritten as
\begin{equation}
\label{op-ad-3}
\big(D^{\widehat{\gamma}}_{(m)}\zeta\big)(x)=\big(\kappa_{-\widehat{\gamma};(m)}*\zeta\big)(x)
=\bigl\langle \kappa_{-\widehat{\gamma};(m)}(\cdot),\zeta(x-\cdot)\bigr\rangle,
\quad \zeta \in \Phi(\bA),
\end{equation}
where
$$
\kappa_{\widehat{\gamma};(m)}(\xi)
=\prod_{p\in\{\infty,2,3,\dots,m\}}\kappa_{\gamma_{p}}^{(p)}(\xi_p)\prod_{p>m}\phi^{(p)}(\xi_p)
$$
is the {\it adelic Riesz kernel}, and $\kappa_{\gamma_{\infty}}^{(\infty)}$ and $\kappa_{\gamma_{p}}^{(p)}$
are given by (\ref{19}) and (\ref{19-p}), respectively.

For $f\in \Phi'(\bA)$ we define the distribution $D^{\widehat{\gamma}}f$ by the relation
\begin{equation}
\label{62**-ad}
\langle D^{\widehat{\gamma}}f,\zeta\rangle\stackrel{def}{=}
\langle f, D^{\widehat{\gamma}}_{(m)}\zeta\rangle,
\qquad \forall \, \zeta\in \Phi^{[m]}(\bA), \quad m\in V_{\bQ}\setminus\infty.
\end{equation}

It is easy to see that Lemmas~\ref{lem2-11},~\ref{lem4-11} imply the following statements.

\begin{Lemma}
\label{lem4-11-ad}
The Lizorkin spaces of test functions $\Phi(\bA)$ and distributions $\Phi'(\bA)$
are invariant under the fractional operator $D^{\widehat{\gamma}}_{(m)}$
and $D^{\widehat{\gamma}}_{(m)}(\Phi'(\bA))=\Phi'(\bA)$.
\end{Lemma}

\begin{Proposition}
\label{the-abelian-2-ad}
The family of operators $\{D^{\widehat{\gamma}}:\widehat{\gamma} \in \bC^{\infty}\}$
on the space of distributions $\Phi'(\bA)$ forms an abelian group: if $f \in \Phi'(\bA)$ then
$$
\begin{array}{rcl}
\displaystyle
D^{\widehat{\gamma}}D^{\widehat{\beta}}f&=&
D^{\widehat{\beta}}D^{\widehat{\gamma}}f=D^{\widehat{\gamma}+\widehat{\beta}}f, \medskip \\
\displaystyle
D^{\widehat{\gamma}}D^{-\widehat{\gamma}}f&=&f,
\qquad \widehat{\gamma}, \, \widehat{\beta} \in \bC^{\infty}. \\
\end{array}
$$
\end{Proposition}

For the case $\widehat{\gamma}=(\gamma_{\infty},\gamma_{2},\dots,\gamma_{p},\dots)\in \bC^{\infty}$
such that $\gamma_{\infty}=\gamma_{2}=\dots=\gamma_{p}=\dots=\gamma\in \bC$
we shall write $D^{\gamma}$ instead of $D^{\widehat{\gamma}}$.
Similarly to (\ref{ad-frac-op}), the {\em adelic fractional operator} $D^{\gamma}$ of order
$\gamma\in \bC^{\infty}$ is defined by its projection on any subspace $\Phi^{[m]}(\bA)\subset \Phi(\bA)$
(see (\ref{ad-1-test-1-L})):
\begin{equation}
\label{ad-frac-op-2}
D^{\gamma}\big|_{\Phi^{[m]}(\bA)}
\stackrel{def}{=}D^{\gamma}_{(m)}=\bigotimes\limits_{p\in
\{\infty,2,3,\dots,m\}}D^{\gamma}_{p}\,
\,\otimes\bigotimes\limits_{p>m}Id_p, \quad m\in V_{\bQ},
\end{equation}
where $D^{\gamma}_{\infty}$ and $D^{\gamma}_{p}$ are defined by
(\ref{22}), (\ref{23}) and (\ref{61**-ad}), (\ref{61**-ad-11}),
respectively, $Id_p$ is the identity operator in $\Phi(\bQ_p)$.
If
$$
\zeta(x)=\zeta_{\infty}(x_{\infty})\prod_{p\in V_\bQ\setminus\infty}\zeta_{p}(x_{p})\in \Phi^{[m]}(\bA)
$$
is an elementary function, i.e., $\zeta_{p}(x_{p})=\phi^{(p)}(\xi_p)=\Omega(|x_p|_p)$
for all $p>m(\zeta)$, then taking into account Definition~(\ref{ad-frac-op-2}) and (\ref{23}),
(\ref{61**-ad}), we obtain
\begin{equation}
\label{op-ad-2-1}
\big(D^{\gamma}\big|_{\Phi^{[m]}(\bA)}\zeta\big)(x)=\big(D^{\gamma}_{(m)}\zeta\big)(x)
=F^{-1}\big[|\cdot|^{\gamma}_{(m)}F[\zeta](\cdot)\big](x),
\end{equation}
where
\begin{equation}
\label{op-ad-1-1}
|\xi|^{\gamma}_{(m)}=\prod_{p\in\{\infty,2,3,\dots,m\}}|\xi_p|_p^{\gamma}\prod_{p>m}\phi^{(p)}(\xi_p),
\quad \xi=(\xi_{\infty},\xi_{2},\dots \xi_{p},\dots) \in \bA.
\end{equation}
is the symbol of the operator $D^{\gamma}_{(m)}$, $m\in V_{\bQ}$.

\subsection{One class of adelic pseudo-differential operators.}\label{s9.3}
On the adelic Lizorkin space of distributions $\Phi'(\bA)$ we introduce a
pseudo-differential operators $A$ with the symbol
$$
\cA(\xi)=\prod_{p\in V_{\bQ}}\cA_p(\xi_p), \quad \xi \in \bA,
$$
where $\cA_{\infty}(\xi_{\infty})\in C^{\infty}(\bR\setminus \{0\})$ and
$\big|\frac{d^s}{d\xi_{\infty}^s}\cA_{\infty}(\xi_{\infty})\big|\le M_{s}|\xi_{\infty}|^{-m_s}$
($m_s\ge 0$) in a neighborhood of $\xi_{\infty}=0$; \,
$\cA_p(\xi_p)\in \cE(\bQ_p\setminus \{0\})$, $p=2,3,\dots$.

The adelic pseudo-differential operator $A$ is defined by its projection on
any subspace of test functions $\Phi^{[m]}(\bA)\subset \Phi(\bA)$
(see (\ref{ad-1-test-1-L})):
\begin{equation}
\label{ad-frac-op-pd}
A\big|_{\Phi^{[m]}(\bA)}\stackrel{def}{=}A_{(m)}=F^{-1}\,\cA_{(m)}\,F,
\quad m\in V_{\bQ},
\end{equation}
where
\begin{equation}
\label{op-ad-1-2}
\cA_{(m)}(\xi)=\cA_{\infty}(\xi_{\infty})\prod_{p\in \{2,3,\dots,m\}}\cA_{p}(\xi_{p})\prod_{p>m}\phi^{(p)}(\xi_p),
\quad \xi \in \bA,
\end{equation}
is the symbol of the pseudo-differential operator $A_{(m)}$.
If
$$
\zeta(x)=\zeta_{\infty}(x_{\infty})\prod_{p\in V_\bQ\setminus\infty}\zeta_{p}(x_{p})\in \Phi^{[m]}(\bA)
$$
is an elementary function, i.e., $\zeta_{p}(x_{p})=\phi^{(p)}(\xi_p)=\Omega(|x_p|_p)$
for all $p>m(\zeta)$, then taking into account Definition~(\ref{ad-frac-op-pd}), (\ref{op-ad-1-2}),
we obtain
\begin{equation}
\label{62**-1-ad}
A\big|_{\Phi^{[m]}(\bA)}\zeta=A_{(m)}\zeta=\big(F^{-1}\,\cA_{(m)}\, F\big)\zeta,
\end{equation}
where the symbol $\cA_{(m)}$ is defined by (\ref{op-ad-1-2}).
For any test function $\zeta=\sum_{k=1}^r\zeta^{\,r}\in \Phi^{[m]}(\bA)$, where $\zeta^{\,r}$
are elementary functions, the operator $A$ is defined by the relation
$$
(A\big|_{\Phi^{[m]}(\bA)}\zeta)(x)=(A_{(m)}\zeta)(x)=\sum_{k=1}^r (A_{(m)}\zeta^{\,r})(x).
$$

Now we define a conjugate pseudo-differential operator $A^{T}$
on $\Phi(\bA)$. For any function $\zeta\in \Phi^{[m]}(\bA)$:
$$
(A^{T}\big|_{\Phi^{[m]}(\bA)}\zeta)(x)\stackrel{def}{=}(A^{T}_{(m)}\zeta)(x)=F^{-1}\big[\,\overline{\cA_{(m)}(-\xi)}\,F[\zeta](\xi)\big](x).
$$
Then the operator $A$ in the Lizorkin space of distributions
$\Phi'(\bA)$ is defined in the usual way: for $f \in \Phi'(\bA)$ we have
\begin{equation}
\label{64.4-ad}
\langle Af,\zeta\rangle=\langle f,A^{T}_{(m)}\zeta\rangle,
\qquad \forall \, \zeta\in \Phi^{[m]}(\bA), \quad m\in V_{\bQ}\setminus\infty.
\end{equation}
It follows from the latter relation and (\ref{alg-53A-test}) that
\begin{equation}
\label{64.4--1-ad}
Af=\big(F^{-1}\,\cA\,F\big)f, \qquad f \in \Phi'(\bA).
\end{equation}

\begin{Lemma}
\label{lem4.3-ad}
The Lizorkin spaces $\Phi(\bA)$ and $\Phi'(\bA)$ are invariant under
the operator {\rm(\ref{64.4--1-ad})}.
\end{Lemma}

\begin{proof}
It is sufficient to prove this lemma for elementary functions.
Let $\zeta\in \Phi^{[m]}(\bA)$ be an elementary function. Since
the function $\cA_{(m)}(\xi)$ defined by (\ref{op-ad-1-2}) is a multiplier
in the space $\Psi^{[m]}(\bQ_p)$, $m\in V_\bQ$, both functions $F[\zeta](\xi)$
and $\cA_{(m)}(\xi)F[\zeta](\xi)$ belong to $\Psi^{[m]}(\bA)$ and, consequently,
$A\zeta\in \Phi^{[m]}(\bA)$, $m\in V_\bQ$.
Thus the pseudo-differential operator $A$ (see (\ref{ad-frac-op-pd}), (\ref{62**-1-ad}))
is well defined and the Lizorkin space $\Phi(\bA)$ is invariant under its action.
Therefore, if $f \in \Phi'(\bA)$, then according to (\ref{64.4-ad}), (\ref{64.4--1-ad}),
$Af=(F^{-1}\,\cA\,F)f\in \Phi'(\bA)$, i.e., the Lizorkin space of distributions
$\Phi'(\bA)$ is invariant under $A$.
\end{proof}

The fractional operators (\ref{op-ad-2}) and (\ref{op-ad-2-1}) are particular cases
of operator (\ref{62**-1-ad}).

Similarly to the above constructions, we introduce the corresponding operators on
the Lizorkin spaces $\Phi({\widetilde\bA})$ and $\Phi'({\widetilde\bA})$ :
the pseudo-differential operator
\begin{equation}
\label{62**-1-ad-00}
A_0\zeta=\big(F^{-1}\,\cA_0)\,F\big)\zeta,
\quad \zeta\in \Phi({\widetilde\bA}),
\end{equation}
with the symbol
\begin{equation}
\label{op-ad-1-2-00}
\cA_0(\xi')=\prod_{p\in V_{\bQ}\setminus{\infty}}\cA_{p}(\xi_{p}), \quad
\xi'=(\xi_{2},\dots \xi_{p},\dots) \in {\widetilde\bA};
\end{equation}
the fractional operator
\begin{equation}
\label{op-ad-2-00} \big(D^{\widehat{\gamma}}_0\zeta\big)(x')
=F^{-1}\big[|\cdot|^{\widehat{\gamma}}_0F[\zeta](\cdot)\big](x'),
\quad \zeta \in \Phi({\widetilde\bA}),
\end{equation}
with the symbol
\begin{equation}
\label{op-ad-1-00}
|\xi'|^{\widehat{\gamma}}_0=\prod_{p\in V_{\bQ}\setminus{\infty}}|\xi_p|_p^{\gamma_p},
\quad \xi'=(\xi_{2},\dots \xi_{p},\dots) \in {\widetilde\bA};
\end{equation}
and the fractional operator
\begin{equation}
\label{op-ad-2-1-00} \big(D^{\gamma}_0\zeta\big)(x')
=F^{-1}\big[|\cdot|^{\gamma}_0F[\zeta](\cdot)\big](x'), \quad \zeta
\in \Phi({\widetilde\bA}),
\end{equation}
with the symbol
\begin{equation}
\label{op-ad-1-1-00-00}
|\xi'|^{\gamma}_0=\prod_{p\in V_{\bQ}\setminus{\infty}}|\xi_p|_p^{\gamma},
\quad \xi'=(\xi_{2},\dots \xi_{p},\dots) \in {\widetilde\bA}.
\end{equation}

\subsection{On the eigenfunctions of adelic pseudo-differential operators.}\label{s9.4}

\begin{Theorem}
\label{th-o4.2-ad}
Let $A_0$ be a pseudo-differential operator {\rm(\ref{62**-1-ad-00})}
with symbol {\rm(\ref{op-ad-1-2-00})}, and
$$
\widetilde{\Psi}_{\alpha}(x')=\widetilde{\Psi}_{(\widehat{k}; \widehat{j}\ \widehat{a})}(x')
=\bigotimes\limits_{2\le q\leq m}\psi_{\alpha_{q}}^{(q)}(x_q)
\otimes\bigotimes\limits_{m<q}\phi^{(q)}(x_q), \quad \alpha\in \widetilde{\Lambda},
\quad x'\in {\widetilde\bA},
$$
be wavelet functions {\rm(\ref{1-ad+=11-00})} which do not contain factors
$\psi_{\alpha_{q}}^{(q)}(x_{q})=\phi^{(q)}(x_q-a_q)$, $a_p \in I_p$,
{\rm(}see {\rm(\ref{basis-phi-inQ}), (\ref{basis-phi-inZ}))} for $q\le m$;
$\widetilde{\Lambda}$ is the indexes set {\rm(\ref{A_(p,m)-phi-00})}.
Then $\widetilde{\Psi}_{\alpha}$ is an eigenfunction of the operator
$A_0$ if and only if
\begin{equation}
\label{64***-ad}
\prod_{q\in\{2,3,\dots,m\}}\cA_{q}\big(q^{-j_q}(-q^{-1}k_q+\eta_q)\big)
=\prod_{q\in\{2,3,\dots,m\}}\cA_{q}\big(-q^{-j_q-1}k_q\big), \quad j_q\in \bZ_{+},
\end{equation}
holds for all $\eta_q \in \bZ_q$, $q=2,3,\dots,m$.
The corresponding eigenvalue is the following
$\lambda=\prod_{q\in \{2,3,\dots,m\}}\cA_{q}\big(-q^{j_q-1}k_q\big)$, i.e.,
$$
A_0\big|_{\Phi^{[m]}({\widetilde\bA})}\widetilde{\Psi}_{\alpha}(x')
\stackrel{def}{=}A_{0 (m)}\widetilde{\Psi}_{\alpha}(x')
\qquad\qquad\qquad\qquad\qquad\qquad\qquad\qquad\quad
$$
$$
\qquad\qquad\qquad\qquad\quad
=\Big(\prod_{q\in\{2,3,\dots,m\}}\cA_{q}\big(-q^{j_q-1}k_q\big)\Big)
\widetilde{\Psi}_{\alpha}(x'), \quad \alpha\in \widetilde{\Lambda},
\quad x'\in {\widetilde\bA}.
$$
\end{Theorem}

\begin{proof}
In view of (\ref{op-ad-1-2-00}) and (\ref{14.2}), we have
$$
A_0\big|_{\Phi^{[m]}({\widetilde\bA})}\widetilde{\Psi}_{\alpha}(x')\stackrel{def}{=}A_{0 (m)}\widetilde{\Psi}_{\alpha}(x')
=F^{-1}\big[\cA_{0 (m)}(\xi')F[\widetilde{\Psi}_{\alpha}(\xi')\big]\big](x')
\qquad\qquad
$$
$$
\qquad\qquad\quad
=\prod_{2\le q\leq p}F^{-1}\big[\cA_{q}F[\psi_{\alpha_{q}}^{(q)}\big]\big](x_q)
\prod_{p<q\leq m}F^{-1}\big[\cA_{q}F[\psi_{\alpha_{q}}^{(q)}\big]\big](x_q)
\prod_{q>m}\phi^{(q)}(x_q),
$$
where
$\psi_{\alpha_{q}}^{(q)}(x_{q})=\psi^{(q)}_{k_q;\, j_q a_q}(x_{q})
=q^{j_q/2}\chi_q\big(\frac{k_q}{q}(q^{-j_q}x_q-a_q)\big)\Omega\big(|q^{-j_q}x_q-a_q|_q\big)$
(see (\ref{basis-phi-inQ}), (\ref{basis-phi-inZ})), and $x_q\in \bQ_q$, $2\le q\leq p$;
$x_q\in \bZ_q$, $p<q\leq m$.

If condition (\ref{64***-ad}) holds, then using formulas (\ref{o-64.8*-ad}),
(\ref{o-64.8*-ad-11*}), we obtain
$$
A_0\big|_{\Phi^{[m]}({\widetilde\bA})}\widetilde{\Psi}_{\alpha}(x')=A_{0 (m)}\widetilde{\Psi}_{\alpha}(x')
\qquad\qquad\qquad\qquad\qquad\qquad\qquad\qquad\qquad
$$
$$
=\prod_{2\le q\leq p}q^{-j_q/2}F^{-1}\Big[\cA_{q}(\xi_q)
\chi_q\big(q^{j_q}a_q\xi_q\big)\Omega\Big(\Big|\frac{k_q}{q}+q^{j_q}\xi_q\Big|_q\Big)\Big](x_q)
\times\qquad\qquad
$$
$$
\times
\prod_{p<q\leq m}q^{-j_q/2}F^{-1}\Big[\cA_{q}(\xi_q)
\Omega(|q^{j_q}a_q|_q)\chi_q\big(q^{j_q}a_q\xi_q\big)
\Omega\Big(\Big|\frac{k_q}{q}+q^{j_q}\xi_q\Big|_q\Big)\Big](x_q)\times
$$
$$
\qquad\qquad\qquad\qquad\qquad\qquad\qquad\qquad\qquad\qquad\qquad\qquad\quad
\times
\prod_{q>m}\phi^{(q)}(x_q).
$$

Making the change of variables $\xi_q=q^{-j_q}(\eta_q-q^{-1}k_q)$ and using
formulas (\ref{alg-53A-test}), (\ref{14.2}), we have for $2\le q\leq p$
$$
q^{-j_q/2}F^{-1}\Big[\cA_{q}(\xi_q)
\chi_q\big(q^{j_q}a_q\xi_q\big)\Omega\Big(\Big|\frac{k_q}{q}+q^{j_q}\xi_q\Big|_q\Big)\Big](x_q)
\qquad\qquad\qquad\qquad\qquad
$$
$$
=\cA_{q}\big(-q^{j_q-1}k_q\big)
q^{j_q/2}\int_{\bQ_q}\chi_{q}\big((-q^{-1}k_q+\eta_q)
(-q^{-j_q}x_q+a_q)\big)\Omega(|\eta_q|_q)\,d\eta_q
$$
$$
\qquad\qquad\qquad\qquad\qquad\qquad\qquad\quad
=\cA_{q}\big(-q^{j_q-1}k_q\big)\psi_{\alpha_{q}}^{(q)}(x_q),
\quad x_q\in \bQ_q.
$$

In the same way, we obtain for $p<q\leq m$
$$
q^{-j_q/2}F^{-1}\Big[\cA_{q}(\xi_q)
\Omega(|q^{j_q}a_q|_q)\chi_q\big(q^{j_q}a_q\xi_q\big)
\Omega\Big(\Big|\frac{k_q}{q}+q^{j_q}\xi_q\Big|_q\Big)\Big](x_q)
\qquad\qquad\quad
$$
$$
=\cA_{q}\big(-q^{j_q-1}k_q\big)\Omega(|q^{j_q}a_q|_q)\psi_{\alpha_{q}}^{(q)}(x_q)
\qquad\qquad\qquad\qquad\qquad\qquad\qquad
$$
$$
\qquad
=\cA_{q}\big(-q^{j_q-1}k_q\big)\Omega(|q^{j_q}a_q|_q)
q^{j_q/2}\chi_q\Big(\frac{k_q}{q}(q^{-j_q}x_q-a_q)\Big)\Omega\big(|q^{-j_q}x_q-a_q|_q\big).
$$
The latter expression is nonzero only if $q^{j_q}a_q\in \bZ_q$ and $|x_q-q^{j_q}a_q|_q\le q^{-j_q}$.
This implies that $x_q\in \bZ_q$. Thus,
$$
q^{-j_q/2}F^{-1}\Big[\cA_{q}(\xi_q)
\Omega(|q^{j_q}a_q|_q)\chi_q\big(q^{j_q}a_q\xi_q\big)
\Omega\Big(\Big|\frac{k_q}{q}+q^{j_q}\xi_q\Big|_q\Big)\Big](x_q)
$$
$$
\qquad\qquad\qquad\qquad\qquad\qquad
=\cA_{q}\big(-q^{j_q-1}k_q\big)\psi_{\alpha_{q}}^{(q)}(x_q), \quad x_q\in \bZ_q.
$$

Consequently, $A_0\widetilde{\Psi}_{\alpha}(x')=A_{0 (m)}\widetilde{\Psi}_{\alpha}(x')
=\lambda \widetilde{\Psi}_{\alpha}(x')$,
where
$$
\lambda=\prod_{q\in\{2,3,\dots,m\}}\cA_{q}\big(-q^{j_q-1}k_q\big), \quad x'\in {\widetilde\bA}.
$$
Conversely, if $A_0\big|_{\Phi^{[m]}({\widetilde\bA})}\widetilde{\Psi}_{\alpha}(x')
=A_{0 (m)}\widetilde{\Psi}_{\alpha}(x')=\lambda \widetilde{\Psi}_{\alpha}(x')$, then
taking the Fourier transform of both left- and right-hand sides of this identity and
using formulas (\ref{o-64.8*-ad}), (\ref{o-64.8*-ad-11*}), we obtain
$$
\big(\cA_0(\xi')-\lambda\big)\prod_{2\le q\leq
m}\Omega\Big(\Big|\frac{k_q}{q}+q^{j_q}\xi_q\Big|_q\Big)=0, \quad
\xi'\in {\widetilde\bA}.
$$
If now $\frac{k_q}{q}+q^{j_q}\xi_q=\eta_q$, $\eta_q \in \bZ_q$, then
$\xi_q=\big(\eta_q-\frac{k_q}{q}\big)q^{-j_q}$ for $2\le q\leq m$.
Thus,
$$
\lambda=\cA_{0 (m)}\Big(\big(\eta_2-2^{-1}k_2\big)2^{-j_2},\dots,\big(\eta_m-m^{-1}k_m\big)m^{-j_m}\Big),
\quad \forall \, \eta_q\in \bZ_q, \,\, 2\le q\leq m.
$$
In particular, $\lambda=\cA_{0 (m)}\big(2^{-j_2-1}k_2,\dots,m^{-j_m-1}k_m\big)$ and
consequently (\ref{64***-ad}) holds.
\end{proof}

According to (\ref{op-ad-2-00}), (\ref{op-ad-1-00}), the adelic
fractional operator $D^{\widehat{\gamma}}_{0 (m)}$ has the symbol
$\cA_{0 (m)}(\xi')=|\xi'|^{\widehat{\gamma}}_{0 (m)}=\prod_{q\in\{2,3,\dots,m\}}|\xi_q|_p^{\gamma_q}\prod_{q>m}\phi^{(q)}(\xi_q)$,
\, $\xi'=(\xi_{2},\dots \xi_{p},\dots) \in {\widetilde\bA}$.
It is easy to see that the symbol $\cA_{0 (m)}(\xi')=|\xi'|^{\widehat{\gamma}}_{0 (m)}$
satisfies condition (\ref{64***-ad}):
$$
\prod_{q\in\{2,3,\dots,m\}}\cA_{q}\big(q^{-j_q}(-q^{-1}k_q+\eta_q)\big)
=\prod_{q\in\{2,3,\dots,m\}}\Big|q^{-j_q}\big(-q^{-1}k_q+\eta_q\big)\Big|_q^{\gamma_q}
$$
$$
=\prod_{q\in\{2,3,\dots,m\}}q^{j_q}\big|-q^{-1}k_q+\eta_q\big|_q^{\gamma_q}
=\prod_{q\in\{2,3,\dots,m\}}q^{\gamma_q(j_q+1)}
\qquad
$$
$$
\qquad\qquad\qquad
=\prod_{q\in\{2,3,\dots,m\}}\cA_{q}\big(-q^{-j_q-1}k_q\big),
\quad \forall \, \eta_q\in \bZ_q, \quad 2\le q\leq m.
$$

Thus according to Theorem~\ref{th-o4.2-ad}, we have

\begin{Corollary}
\label{o=cor5-ad}
The wavelet function {\rm(\ref{1-ad+=11-00})}
$$
\widetilde{\Psi}_{\alpha}(x')=\widetilde{\Psi}_{(\widehat{k}; \widehat{j}\ \widehat{a})}(x')
=\bigotimes\limits_{2\le q\leq m}\psi_{\alpha_{q}}^{(q)}(x_q)\otimes\bigotimes\limits_{m<q}\phi^{(q)}(x_q),
\quad x'\in{\widetilde\bA},
\quad \alpha\in\tilde\Lambda=\bigcup_{m\in V_{\bQ}\setminus\infty}\tilde{\Lambda}_{m}.
$$
which does not contain the factors
$\psi_{\alpha_{q}}^{(q)}(x_{q})=\phi^{(q)}(x_q-a_q)$, $a_p \in I_p$ for $q\le m$,
is an eigenfunction of the adelic fractional operator $D^{\widehat{\gamma}}_0$:
\begin{equation}
\label{eq-66.1-ad}
D^{\widehat{\gamma}}_0\big|_{\Phi^{[m]}({\widetilde\bA})}\widetilde{\Psi}_{\alpha}(x')
=D^{\widehat{\gamma}}_{0 (m)}\widetilde{\Psi}_{\alpha}(x')
=\Big(\prod_{q\in\{2,3,\dots,m\}}q^{\gamma_q(j_q+1)}\Big)\widetilde{\Psi}_{\alpha}(x'),
\quad x'\in {\widetilde\bA}.
\end{equation}
\end{Corollary}

\begin{Corollary}
\label{o=cor5-ad-1}
The wavelet function {\rm(\ref{1-ad+=11-00})}
$$
\widetilde{\Psi}_{\alpha}(x')=\widetilde{\Psi}_{(\widehat{k}; \widehat{j}\ \widehat{a})}(x')
=\bigotimes\limits_{2\le q\leq m}\psi_{\alpha_{q}}^{(q)}(x_q)
\otimes\bigotimes\limits_{m<q}\phi^{(q)}(x_q), \quad x'\in {\widetilde\bA},
\quad\alpha\in \tilde\Lambda=\bigcup_{m\in V_{\bQ}\setminus\infty}\tilde{\Lambda}_{m}.
$$
which does not contain the factors
$\psi_{\alpha_{q}}^{(q)}(x_{q})=\phi^{(q)}(x_q-a_q)$, $a_p \in I_p$, for $q\le m$,
is an eigenfunction of the adelic fractional operator $D^{\gamma}_0$:
\begin{equation}
\label{eq-66.1-ad-1}
D^{\gamma}_0\big|_{\Phi^{[m]}({\widetilde\bA})}\widetilde{\Psi}_{\alpha}(x')
=D^{\gamma}_{0 (m)}\widetilde{\Psi}_{\alpha}(x')
=\Big(\prod_{q\in\{2,3,\dots,m\}}q^{j_q+1}\Big)^{\gamma}\widetilde{\Psi}_{\alpha}(x'),
\quad x'\in {\widetilde\bA}.
\end{equation}
\end{Corollary}

\begin{center}
{\bf Acknowledgements}
\end{center}

The authors are greatly indebted to A.N.~Kochubei,
V.I.~Polischook, O.G.~Smolyanov for fruitful discussions.

\end{document}